\documentclass[11pt,oneside,reqno]{amsart}
\usepackage{amssymb}
\usepackage{amsmath}
\usepackage{amscd}
\usepackage{amsfonts}
\usepackage{mathrsfs}
\usepackage{stmaryrd}
\usepackage{longtable}
\usepackage{cancel}

\usepackage{tikz}
\usetikzlibrary{arrows}
\usepackage{pgf,tikz,pgfplots}
\pgfplotsset{compat=1.14}
\usepackage[T1]{fontenc}
\allowdisplaybreaks

\definecolor{ffffff}{rgb}{1,1,1}

\makeatletter
\DeclareFontFamily{U}{tipa}{}
\DeclareFontShape{U}{tipa}{m}{n}{<->tipa10}{}
\newcommand{\arc@char}{{\usefont{U}{tipa}{m}{n}\symbol{62}}}%

\newcommand{\arc}[1]{\mathpalette\arc@arc{#1}}

\newcommand{\arc@arc}[2]{%
  \sbox0{$\m@th#1#2$}%
  \vbox{
    \hbox{\resizebox{\wd0}{\height}{\arc@char}}
    \nointerlineskip
    \box0
  }%
}
\makeatother

\usepackage{bbm}

\theoremstyle{definition}
\newtheorem{theorem}{Theorem}[section]

\newtheorem{thm}[theorem]{Theorem}
\newtheorem{prop}[theorem]{Proposition}

\newtheorem{defn}[theorem]{Definition}
\newtheorem{lemma}[theorem]{Lemma}
\newtheorem{cor}[theorem]{Corollary}
\newtheorem{prop-def}{Proposition-Definition}[section]

\newtheorem{rema}[theorem]{Remark}

\newtheorem{conjecture}{Conjecture}
\newtheorem{exam}[theorem]{Example}
\newtheorem{nota}[theorem]{Notation}

\newcommand{\R}{{\mathbb R}}

\newcommand{\C}{{\mathbb C}}
\newcommand{\Z}{{\mathbb Z}}

\newcommand{\End}{\textrm{End}}

\newcommand{\Hol}{\text{Hol}}
\newcommand{\state}[1]{{\left| #1\right\rangle}}

\newcommand{\Ten}{\text{Ten}}

\begin{document}

\setlength{\oddsidemargin}{0cm} \setlength{\evensidemargin}{0cm}
\baselineskip=18pt

\title[Covariant derivatives of eigenfunctions along parallel tensors over space forms]{Covariant derivatives of eigenfunctions along parallel tensors over space forms and a conjecture motivated by the vertex algebraic structure}
\author{Fei Qi}

\maketitle

\begin{abstract}
We study the covariant derivatives of an eigenfunction for the Laplace-Beltrami operator on a complete, connected Riemannian manifold with nonzero constant sectional curvature. We show that along every parallel tensor, the covariant derivative is a scalar multiple of the eigenfunction. We also show that the scalar is a polynomial depending on the eigenvalue and prove some properties. A conjecture motivated by the study of vertex algebraic structure on space forms is also announced, suggesting the existence of interesting structures in these polynomials that awaits further exploration. 
\end{abstract}

\section{Introduction}

The study is motivated by Yi-Zhi Huang's construction of a meromorphic open-string vertex algebra (MOSVA hereafter) and its modules over a Riemannian manifold in \cite{H-MOSVA-Riemann}. Roughly speaking, a MOSVA is an algebraic structure of vertex operators that are associative, but not necessarily commutative (see \cite{H-MOSVA}, \cite{Q-Mod} for more details). To give a rough description on Huang's construction, let 
\begin{itemize}
    \item $M$ be a Riemannian manifold; 
    \item $TM$ be the tangent bundle with Levi-Civita connection;
    \item $TM^\C$ be the complexified tangent bundle $\C \otimes_\R TM$ with the natural connection; 
    \item $(TM^\C)^{\otimes r}$ be the tensor product bundle of $TM^\C$ of degree $r$; 
    \item $\Ten(TM^\C)$ be the tensor algebra bundle, i.e., $\Ten(TM^\C) = \oplus_{r=0}^\infty (TM^\C)^{\otimes r}$;
    \item $\Pi((TM^{\C})^{\otimes r})$ be the space of parallel $r$-tensors, i.e., parallel sections of the tensor bundle $(TM^\C)^{\otimes r}$ with respect to the natural connection; 
    \item $\Pi(\Ten(TM^\C))$ be the space of all parallel tensors, i.e., parallel sections of the tensor algebra bundle $\Ten(TM^\C)$ with respect to the natural connection. 
    \item $C^\infty(U)_\C$ be the space of complex-valued smooth functions defined on an open subset $U$ of $M$
\end{itemize}
In \cite{H-MOSVA-Riemann}, Huang constructed a MOSVA on the space of parallel sections of certain affinized bundle of $TM^{\C}$. On the space $C^\infty(U)_\C$, using the Levi-Civita connection $\nabla$ on $M$, Huang defined an action of $\Pi(\Ten(TM^\C))$ by
$$\psi_U(X)f = (\nabla^m f)(X)$$
for every $X \in \Pi((TM^{\C})^{\otimes m})$ of degree $m$. In other words, $\psi_U(X)f$ is the degree-$m$ covariant derivative of $f$ along $X$. Huang also showed that
\begin{align}
    \psi_U(X\otimes Y) = \psi_U(X)\psi_U(Y) \label{Assoc}
\end{align}
for every $X, Y \in \Pi(\Ten(TM^\C))$. In other words, the space $C^\infty(U)_\C$ is a module for the associative algebra $\Pi(\Ten(TM^\C))$. An induced module construction can then performed, giving a module for the MOSVA. 

Of particular interest are the submodules generated by an eigenfunction for the Laplace-Beltrami operator. Huang showed that the Laplace-Beltrami operator appears as a component of some vertex operator in the MOSVA he constructed. Thus starting from an eigenfunction $f$ of the Laplace-Beltrami operator, we can induce the $\Pi(\Ten(TM^\C))$-submodule of $C^\infty(U)_\C$ generated by $f$ to a module for the MOSVA. As eigenfunctions can be understood as quantum states in quantum mechanics, the modules they generate can be understood as the string-theoretic excitement to the quantum states. It is Huang's idea that the modules for the MOSVA generated by the eigenfunctions and the yet-to-be-defined intertwining operators among these modules may lead to a mathematical construction of the quantum two-dimensional nonlinear $\sigma$-model. 

Therefore, to understand the module for the MOSVA generated by an eigenfunction $f$, the first step is to understand the $\Pi(\Ten(TM^\C))$-submodule of $C^\infty(U)_\C$ generated by $f$, which is simply the space spanned by the covariant derivatives of $f$ along all parallel sections. In \cite{Q-MOSVA-2d}, the author studied the example of MOSVA and its eigenfunction modules for a two-dimensional orientable, complete, connected Riemannian manifold whose sectional curvature is constant and nonzero (or for short, a two-dimensional orientable space form with nonzero curvature), and was surprised to find that all such covariant derivatives are scalar multiples of the function $f$. In other words, the $\Pi(\Ten(TM^\C))$-submodule in $C^\infty(U)_\C$ generated by $f$ is simply the one-dimensional $\C f$. Moreover, the scalar is a polynomial depending on the eigenvalue of $f$. Using properties of the polynomials, we discovered that the irreducible modules for the MOSVA generated by eigenfunctions with eigenvalues $\lambda = Kp(p+1) (p = 0, 1, 2, ...)$ differ from those with generic eigenvalues. 

This paper serves as the first step of higher dimensional generalization of results in \cite{Q-MOSVA-2d}. We show that over higher-dimensional orientable and non-orientable space forms with nonzero curvature, every covariant derivative of an eigenfunction along a parallel tensor is a scalar multiple of the function. We also prove that the scalar is a polynomial in eigenvalues and discussed some combinatorial properties of these polynomials. Since $O(n, \R)$ and $SO(n, \R)$ are non-abelian, their invariant theories are more complicated than that for the commutative $SO(2, \R)$. So are proofs in this paper. 

It is also expected that irreducible modules generated by eigenfunctions of eigenvalue $Kp(p+n-1) (p=0, 1, 2, ...)$ are different and possess similar properties as in \cite{Q-MOSVA-2d}. However, our limited understanding of these polynomials obstructs the study. We summarize the obstruction as a conjecture, which suggests the existence of interesting structures in these polynomials that requires further exploration. 

This paper is organized as follow: 

In Section 2, we discuss the holonomy group of the tensor bundles $(TM^\C)^{\otimes r}$ and the tensor algebra bundle $\Ten(TM^\C)$. The discussion reduces the problem of finding parallel tensors to the invariant theory of $O(n,\R)$ and $SO(n, \R)$. We then use the results in \cite{LZ} and \cite{GW} to give a spanning set of the space $\Pi((TM^\C)^{\otimes r})$, thus characterizing the space $\Pi(\Ten(TM^\C))$. 

In Section 3, we discuss the fundamental lemma of covariant derivatives to be used in this work. The lemma was proved in \cite{Q-MOSVA-2d}. The proof is repeated here for the convenience of readers. 

In Section 4, using the fundamental lemma of covariant derivatives extensively, we give a proof to the main result. In particular, for orientable space forms, the results in Section 2 gives two different types of parallel tensors for orientable space forms, one type is $O(n, \R)$-invariant, and the other is not. We show that only those $O(n, \R)$-invariant ones can have nonzero actions. All parallel tensors that are not $O(n, \R)$-invariant annihilates $f$. 

In Section 5, we encode the $O(n, \R)$-invariant tensors by certain words and graphs. We derive a recursion to compute the scalar, which is a polynomial depending on the eigenvalue. We exhibit these polynomials for all parallel tensors of degrees 2, 4, and 6, together with the proof of some combinatorial properties. 

In Section 6, we announce the conjecture that obstructs the study of the MOSVAs and modules on higher-dimensional space forms. A linking operator is introduced on the graphs representing the parallel tensor. For the special graph whose polynomials has the ``largest'' highest degree component, we consider the linear combination of polynomials obtained from all possible ways of linking the right half of the graph. The coefficients are determined by a linear system, consisting of vanishing condition obtained from linking the left half of all the graphs. We conjecture that the polynomial is a polynomial that vanishes when the eigenvalues are $Kp(p+n-1) (p=0, 1, 2, ...)$. Numerical evidence and explanation of motivations are also provided.

\textbf{Acknowledgements.} The author would like to thank Yi-Zhi Huang for his long-term support. The author would also like to thank Roe Goodman and Nolan Wallach for discussions on invariant theory, and Johannes Flake for bringing the author's attention to the work \cite{LZ}. 


\section{Parallel tensors}

Let $M$ be a $n$-dimensional Riemannian manifold with constant sectional curvature $K$. For convenience, we assume $M$ is connected and complete. We will also focus on the case $K\neq 0$ and $n\geq 2$. 


\subsection{The curvature tensor} \label{R-tensor}
Fix $p\in U$, let $\{e_1,... e_n\}$ be an orthonormal basis of $T_p M$. Then for some neighborhood $U$ of $p$, let $X_1, ..., X_n: U \to TM$ be local sections such that $X_i|_p = e_i$ and for every $q\in U$, $(X_i|_q, X_j|_q) = \delta_{ij}$. For convenience, we will not distinguish the tangent vectors at a point and the sections over an open set when there is no confusion. 

Since the section curvature is constant and equal to $K$, for every $q\in U$ and every $v_1, v_2, v_3\in T_qM$, we have
$$R(v_1, v_2)v_3 = -K (g(v_1, v_3)v_2 - g(v_2, v_3)v_1)$$
(See \cite{P}, Proposition 3.1.3). In particular, for mutually distinct $i,j,k$, we have
\begin{align*}
   R(X_i, X_j)X_k &= 0, R(X_i, X_j)X_i &= -K X_j, R(X_i, X_j)X_j &= K X_i. 
\end{align*}

Regarded as a linear endomorphism on $T_qM$, the matrix of $R(X_i, X_j)$ with respect to the basis $\{X_1, ..., X_n\}$ is the skew-symmetric matrix $KE_{ij} - KE_{ji}$, where $E_{ab}$ is $n\times n$-matrix with $(a,b)$-entry being one, and all other entries being zero. In the case $K\neq 0$, the subspace spanned by the matrices of $R(X_i, X_j)$ in $\End(T_q M)$ coincides with the subspace spanned by the skew-symmetric matrices, which is precisely the Lie algebra of $SO(n, \mathbb{R})$. 

\subsection{Holonomy of the tangent bundle}

Recall that the holonomy group of a bundle $E$ based at a point $p\in M$ is the subgroup generated by all the parallel translations along piecewise smooth loops based on $p$. To determine the holonomy group of $TM$, we will use the following version of Ambrose-Singer theorem over vector bundles. 

\begin{lemma}[\cite{J}, Theorem 2.4.3(a)]\label{ASThm}
Let $M$ be a manifold, $E$ a vector bundle over $M$, and $\nabla$ a connection
on $E$. Fix $p \in M$, so that $\mathfrak{hol}_p(TM)$ is a Lie subalgebra of $\End(T_p M)$. Then
$\mathfrak{hol}_p(TM)$ is the vector subspace of $\End(T_p M)$ spanned by all elements of $\End(T_p M)$ of
the form $P_\gamma^{-1}[R(v, w)]P_\gamma$ where $R$ is the curvature tensor, $p \in M$ is a point, $\gamma : [0, 1] \to M$ is piecewise
smooth with $\gamma(0) = p$ and $\gamma(1) = q$, $P_\gamma : T_p M \to T_qM$ is the parallel translation
map, and $v,w \in T_q M$.
\end{lemma}

\begin{lemma}
For every $p\in M$, the holonomy group $\Hol_p(TM)$ of the tangent bundle $TM$ is 
$$\Hol_p(TM) =\left\{ \begin{array}{ll}
SO(n, \R) & \text{ if }M\text{ is orientable,}\\
O(n, \R) & \text{ if }M\text{ is non-orientable. }\\
\end{array}\right.$$
\end{lemma}

\begin{proof}
From Lemma \ref{ASThm} with $\gamma(t)=p$ being the trivial loop, we see that the Lie algebra $\mathfrak{hol}_p(TM)$ contains all $R(v,w)$ for $v,w\in T_pM$. It follows from the discussion in \ref{R-tensor} that $\mathfrak{hol}_p(TM)$ contains the Lie algebra of $SO(n,\mathbb{R})$. Thus $\Hol_p(TM) \supset SO(n,\mathbb{R})$. The conclusion then follows from the fact that $M$ is orientable if and only if $\Hol_p(TM) \subset SO(n,\R)$ (see \cite{P}). 
\end{proof}

\begin{lemma}
For every $p\in M$, the holonomy group $\Hol_p(TM^\C)$ of the complexified tangent bundle $TM$ is 
$$\Hol_p(TM^\C) =\left\{ \begin{array}{ll}
SO(n, \R) & \text{ if }M\text{ is orientable,}\\
O(n, \R) & \text{ if }M\text{ is non-orientable. }\\
\end{array}\right.$$
\end{lemma}

\begin{proof}
This essentially follows from the fact that as a bundle, $TM^\C = TM \oplus \sqrt{-1}TM$. 
\end{proof}

For every $r\in \Z_+$, let $\Pi((TM^\C)^{\otimes r})$ be the space of parallel sections of the tensor bundle $(TM^\C)^{\otimes r}$. For convenience, elements of $\Pi((TM^\C)^{\otimes r})$ will simply be called parallel tensors. It is well-known that $\Pi((TM^\C)^{\otimes r})$ can be identified with the fixed point subspace $((T_pM^\C)^{\otimes r})^{\Hol_p((TM^\C)^{\otimes r})}$ in $(T_pM^\C)^{\otimes r}$. We start by explicitly determining the holonomy group of $(TM^\C)^{\otimes r}$. 

\begin{lemma}\label{HolTMr} There is a natural surjective homomorphism $\Hol_p(TM^\C) \to \Hol_p((TM^\C)^{\otimes r})$ of holonomy groups, where $g\in \Hol_p(TM^\C)$ is mapped to $g^{\otimes r}: (T_pM^\C)^{\otimes r} \to (T_pM^\C)^{\otimes r}$ defined by 
$$(g^{\otimes r})(v_1\otimes \cdots \otimes v_r) = gv_1\otimes \cdots \otimes gv_r.$$
for $v_1, ..., v_r\in T_p M$. A tensor $X\in (T_pM^\C)^{\otimes r}$ is fixed by every element in $\Hol_p((TM^\C)^{\otimes r})$, if and only if $g^{\otimes r}X = X$ for every $g\in \Hol_p(TM^\C)$. 
\end{lemma}

\begin{proof}
For any piecewise smooth path $\gamma: [0, 1]\to M$ with $\gamma(0)=p$, let $P_{\gamma(t)}: T_p M \to T_{\gamma(1)}M$ be the parallel transport along $\gamma$ on the bundle $E$; let  $P_{\gamma(t)}^r: (T_pM^\C)^{\otimes r} \to T_{\gamma(1)}M^{\otimes r}$ be the parallel transport along $\gamma$ with respect to the bundle $(TM^\C)^{\otimes r}$. Then 
from the definition of the connection on $(TM^\C)^{\otimes r}$: 
$$\nabla(X_1 \otimes \cdots \otimes X_n) = \sum_{i=1}^n X_1 \otimes \cdots \otimes \nabla(X_i) \otimes \cdots \otimes X_n, $$
it follows that
$$P_{\gamma(t)}^r(v_1\otimes \cdots \otimes v_r) = P_{\gamma(t)}v_1 \otimes \cdots \otimes P_{\gamma(t)}v_r.$$
In the case that $\gamma(t)$ is a loop based at $p$, this essentially realizes every element of $\Hol_p((TM^\C)^{\otimes r})$ as $g^{\otimes r}$ for $g\in \Hol_p(TM^\C)$. So the map $g\mapsto g^{\otimes r}$ gives a natural surjective homomorphism  $\Hol_p(TM^\C) \to \Hol_p((TM^\C)^{\otimes r})$. The second conclusion follows directly from this realization. 
\end{proof}

Therefore, to identify $\Pi((TM^\C)^{\otimes r})$, it suffices to identify $((T_pM^\C)^{\otimes r})^{SO(n, \R)}$ if $M$ is orientable; $((T_pM^\C)^{\otimes r})^{O(n, \R)}$ if $M$ is non-orientable.


\subsection{Parallel tensors} We now use the first fundamental theorem of invariant theory of orthogonal groups to give a spanning set of the space of parallel tensors. We will state the theorem of $O(n,\C)$ and $SO(n, \C)$, then apply Weyl's unitary trick to reduce to $O(n, \R)$ and $SO(n,\R)$. 

For every integer $r$, we denote the symmetric group of $\{1, ..., r\}$ by $\text{Sym}_r$, which acts naturally on $(T_pM^\C)^{\otimes r}$ by permutation: 
$$\sigma(v_1\otimes\cdots\otimes v_r) = v_{\sigma^{-1}(1)}\otimes \cdots \otimes v_{\sigma^{-1}(r)}. $$
Consider now the tensors 
\begin{align*}
    \theta &= \sum_{i=1}^n X_i \otimes X_i\\
    \Lambda &= \sum_{\sigma\in \text{Sym}_n} (-1)^\sigma X_{\sigma(1)}\otimes \cdots \otimes X_{\sigma(n)}
\end{align*}
Roughly speaking, $\theta$ corresponds to the metric form; $\Lambda$ corresponds to the volume form.

\begin{lemma}[\cite{GW}, Theorem 5.3.3]\label{3-9}
For every $r\in \Z_+$,  
$$((T_pM^\C)^{\otimes r})^{O(n, \C)} = \left\{\begin{array}{ll} 
\text{span}_\C\{\sigma_{r}(\theta^{\otimes k}), \sigma_{r}\in \text{Sym}_{r}\} & \text{if $r=2k$ is even,}\\
0 & \text{otherwise.}
\end{array}\right.$$
\end{lemma}

\begin{lemma}[\cite{LZ}, Theorem 2.1]\label{3-10}
For every $r\in \Z_+$, the space $((T_pM^\C)^{\otimes r})^{SO(n, \C)}$ can be decomposed as
$$((T_pM^\C)^{\otimes r})^{O(n,\C)} \oplus ((T_pM^\C)^{\otimes r})^{O(n, \C),\det}$$
where $((T_pM^\C)^{\otimes r})^{O(n,\C)}$ is defined as in Lemma \ref{3-9}, and 
\begin{align*}
((T_pM^\C)^{\otimes r})^{O(n,\C), \text{det}} &= \left\{\begin{array}{ll}
\text{span}_\C\{\sigma_{r}  (\theta^{\otimes k} \otimes \Lambda), \sigma_r \in \text{Sym}_r \}) & \text{if $r = n+2k\geq 0$ is even, }\\
0 & \text{otherwise.}
\end{array}\right. 
\end{align*}
\end{lemma}

\begin{prop}
$((T_pM^\C)^{\otimes r})^{SO(n, \mathbb{R})} = ((T_pM^\C)^{\otimes r})^{SO(n, \mathbb{C})}$
\end{prop}

\begin{proof}
Since $SO(n, \R)\subset SO(n, \C)$, the left-hand-side contains the right-hand-side. A standard application of the unitary trick shows that that the left-hand-side is included in the right-hand-side. In greater detail, let $X$ be an element in the left-hand-side. Then $X$ is annihilated by every element in the Lie algebra $\mathfrak{so}(n, \R)$. Since $\mathfrak{so}(n, \C) = \mathfrak{so}(n, \R)\otimes_\R \C$, $X$ is also annihilated by every element in the Lie algebra $\mathfrak{so}(n, \C)$. Thus $X$ is fixed by every element in $SO(n, \mathbb{C})$.
\end{proof}

Roughly speaking, the parallel tensors are generated by applying the permutations to $\theta^k$ and $\theta^k \cdot \Lambda$. For convenience, the permutations of $\theta^k$ will be called $O(n,\R)$-invariant tensors, and the permutations of $\theta^k\otimes \Lambda$ will be called non-$O(n,\R)$-invariant tensors.






\section{Fundamental lemma of covariant derivatives}


\begin{thm}\label{Comm-Cov}
Let $f: U\to \C$ be a complex-valued smooth function. Then for $n \geq 3$, we have
$$(\nabla^n f)(Z_1, ..., Z_{n-1}, Z_n ) - (\nabla^n f)(Z_1, ..., Z_n, Z_{n-1}) = 0,$$
and for $i = 1, ..., n-2$,
\begin{align*}
    & (\nabla^n f)(Z_1, ..., Z_i, Z_{i+1}, ..., Z_n) - (\nabla^n f)(Z_1, ..., Z_{i+1}, Z_{i}, ..., Z_n)\\
    = & \sum_{j=i+2}^n(\nabla^{n-2} f)(Z_1, ..., -R(Z_i, Z_{i+1})Z_{j}, ..., Z_n) \\
    = & (\nabla^{n-2} f)(Z_1, ..., -R(Z_i, Z_{i+1})Z_{i+2}, Z_{i+3} ..., Z_n) \\
    & + (\nabla^{n-2} f)(Z_1, ..., Z_{i+2}, -R(Z_i, Z_{i+1})Z_{i+3}, ..., Z_n) + \\
    & + \cdots \cdots \\
    & + (\nabla^{n-2} f)(Z_1, ..., Z_{i+2}, Z_{i+3}, ..., -R(Z_i, Z_{i+1})Z_n)
\end{align*}
\end{thm}

\begin{proof}
We prove the first equation by induction on $n$. For $n=3$, we have
\begin{align*}
    (\nabla^3 f)(Z_1, Z_2, Z_3) &= (\nabla_{Z_1} (\nabla^2 f))(Z_2, Z_3) \\
    & = \nabla_{Z_1} ((\nabla^2 f)(Z_2, Z_3)) - (\nabla^2 f)(\nabla_{Z_1}Z_2, Z_3) - (\nabla^2 f)(Z_2, \nabla_{Z_1}Z_3) \\
    & \quad \text{(note that $\nabla^2 f(X, Y) = \nabla^2 f(Y, X)$)} \\
    & = \nabla_{Z_1} ((\nabla^2 f)(Z_3, Z_2)) - (\nabla^2 f)(Z_3, \nabla_{Z_1}Z_2) - (\nabla^2 f)(\nabla_{Z_1}Z_3, Z_2)  = (\nabla^3 f)(Z_1, Z_3, Z_2)
\end{align*}
Assume the equation holds for $n-1$, so we have
\begin{align*}
    (\nabla^n f)(Z_1, ..., Z_{n-1}, Z_n) & = (\nabla_{Z_1}(\nabla^{n-1}f))(Z_2, ..., Z_{n-1}, Z_n) \\ 
    & = \nabla_{Z_1}((\nabla^{n-1}f)(Z_2, ..., Z_{n-1}, Z_n)) - (\nabla^{n-1}f)(\nabla_{Z_1}Z_2, ..., Z_{n-1}, Z_n) \\
    & \quad -\cdots - (\nabla^{n-1}f)(Z_2, ..., \nabla_{Z_1}Z_{n-1}, Z_n) - (\nabla^{n-1}f)(Z_2, ..., Z_{n-1}, \nabla_{Z_1}Z_n)\\
    & \quad \text{(by induction hypothesis)}\\
    & = \nabla_{Z_1}((\nabla^{n-1}f)(Z_2, ..., Z_n, Z_{n-1})) - (\nabla^{n-1}f)(\nabla_{Z_1}Z_2, ..., Z_n, Z_{n-1}) \\
    & \quad -\cdots - (\nabla^{n-1}f)(Z_2, ..., Z_n, \nabla_{Z_1}Z_{n-1}) - (\nabla^{n-1}f)(Z_2, ..., \nabla_{Z_1}Z_n, Z_{n-1})\\
    & = (\nabla^n f)(Z_1, ..., Z_n, Z_{n-1})
\end{align*}
So the first equation is proved. 

For the second equation, we first consider the case $i=1$:
\begin{align}
    &\quad (\nabla^n f)(Z_1, Z_2, Z_3, \cdots, Z_n)= (\nabla_{Z_1} (\nabla^{n-1}f))(Z_2, Z_3, \cdots, Z_n)\nonumber\\
    &= \nabla_{Z_1}((\nabla^{n-1}f) (Z_2, Z_3 ..., Z_n)) - (\nabla^{n-1}f)(\nabla_{Z_1}Z_2, Z_3, ..., Z_n) - \sum_{j=3}^n (\nabla^{n-1}f)(Z_2, ..., \nabla_{Z_1}Z_j, ..., Z_n) \nonumber\\
    &= \nabla_{Z_1}\nabla_{Z_2}((\nabla^{n-2}f) (Z_3, ..., Z_n)) - \sum_{j=3}^n \nabla_{Z_1}((\nabla^{n-2}f)(Z_3, ..., \nabla_{Z_2} Z_j, ..., Z_n))\label{Line1}\\
    & \quad -\nabla_{\nabla_{Z_1}Z_2} ((\nabla^{n-2}f)(Z_3, ..., Z_n))+\sum_{j=3}^n(\nabla^{n-2}f)(Z_3,..., \nabla_{\nabla_{Z_1}Z_2}Z_j, ..., Z_n)\label{Line2}\\
    & \quad 
    - \sum_{j=3}^n \left(\nabla_{Z_2}(\nabla^{n-2}f)(Z_3, ..., \nabla_{Z_1}Z_j, ..., Z_n) - \sum_{k=3}^{j-1}(\nabla^{n-2}f)(Z_3,..., \nabla_{Z_2}Z_k, ..., \nabla_{Z_1}Z_j, ..., Z_n)\right) \label{Line3} \\
    & \quad - \sum_{j=3}^n \left(-(\nabla^{n-2}f)(Z_3, ..., \nabla_{Z_2}\nabla_{Z_1}Z_j, ..., Z_n)-\sum_{k=j+1}^{n} (\nabla^{n-2}f)(Z_3, ..., \nabla_{Z_1} Z_j , ..., \nabla_{Z_2}Z_k, ..., Z_n)\right)\label{Line4}
\end{align}
Similarly, 
\begin{align}
    &\quad (\nabla^n f)(Z_2, Z_1, Z_3, \cdots, Z_n)= (\nabla_{Z_2} (\nabla^{n-1}f))(Z_1, Z_3, \cdots, Z_n)\nonumber\\
    & = \nabla_{Z_2}\nabla_{Z_1}((\nabla^{n-2}f) (Z_3, ..., Z_n)) - \sum_{j=3}^n \nabla_{Z_2}((\nabla^{n-2}f)(Z_3, ..., \nabla_{Z_1} Z_j, ..., Z_n))\label{Line5}\\
    & \quad -\nabla_{\nabla_{Z_2}Z_1} ((\nabla^{n-2}f)(Z_3, ..., Z_n))+\sum_{j=3}^n(\nabla^{n-2}f)(Z_3,..., \nabla_{\nabla_{Z_2}Z_1}Z_j, ..., Z_n)\label{Line6}\\
& \quad 
    - \sum_{j=3}^n \left(\nabla_{Z_1}(\nabla^{n-2}f)(Z_3, ..., \nabla_{Z_2}Z_j, ..., Z_n) - \sum_{k=3}^{j-1}(\nabla^{n-2}f)(Z_3,..., \nabla_{Z_1}Z_k, ..., \nabla_{Z_2}Z_j, ..., Z_n)\right)  \label{Line7}\\
    & \quad - \sum_{j=3}^n \left(-(\nabla^{n-2}f)(Z_3, ..., \nabla_{Z_1}\nabla_{Z_2}Z_j, ..., Z_n)-\sum_{k=j+1}^{n} (\nabla^{n-2}f)(Z_3, ..., \nabla_{Z_2} Z_j , ..., \nabla_{Z_1}Z_k, ..., Z_n)\right)\label{Line8}
\end{align}
Then in the difference, the second sum in (\ref{Line1}) cancels out with the first term in the sum of (\ref{Line7}); the first term in the sum of (\ref{Line3}) cancels out with the second sum in (\ref{Line5}); the second term in the sum of (\ref{Line3}), together with second term in the sum of (\ref{Line4}), cancel out those in (\ref{Line7}) and (\ref{Line8}). So the difference is
\begin{align*}
 & \quad (\nabla^n f)(Z_1, Z_2, Z_3, ..., Z_n) - (\nabla^n f)(Z_2, Z_1, Z_3, ..., Z_n)\\
    & = (\nabla_{Z_1}\nabla_{Z_2}-\nabla_{Z_2}\nabla_{Z_1})((\nabla^{n-2}f)(Z_3, ..., Z_n)) - \nabla_{\nabla_{Z_1}Z_2 - \nabla_{Z_2}Z_1} ((\nabla^{n-2}f)(Z_3, ..., Z_n) \\
    & \quad+ \sum_{j=3}^n (\nabla^{n-2}f)(Z_3, ..., \nabla_{\nabla_{Z_1}Z_2 -\nabla_{Z_2}Z_1} Z_j, ..., Z_n) + \sum_{j=3}^n (\nabla^{n-2}f)(Z_3, ..., (\nabla_{Z_2}\nabla_{Z_1}-\nabla_{Z_1}\nabla_{Z_2})Z_j, ..., Z_n)\\
    & = \sum_{j=3}^n (\nabla^{n-2}f)(Z_3, ..., (\nabla_{Z_2}\nabla_{Z_1}-\nabla_{Z_1}\nabla_{Z_2}+ \nabla_{\nabla_{Z_1}Z_2 - \nabla_{Z_2}Z_1})Z_j, ..., Z_n)\\
    &=\sum_{j=3}^n (\nabla^{n-2}f)(Z_3, ..., -R(Z_1, Z_2)Z_j, ..., Z_n).
\end{align*}
So the case $i=1$ is proved for arbitrary $n$. 

We proceed by induction on $i$. The base case has been proved above. Now we proceed with the inductive step. 
\begin{align*}
    & \quad (\nabla^n f)(Z_1, ..., Z_i, Z_{i+1}, ..., Z_n) = (\nabla_{Z_1}(\nabla^n f))(Z_2, ..., Z_i, Z_{i+1}, ..., Z_n) \\
    &= \nabla_{Z_1}((\nabla^{n-1} f)(Z_2, ..., Z_i, Z_{i+1}, ..., Z_n)) - \sum_{k=2}^{i-1}(\nabla^{n-1} f)(Z_2, ..., \nabla_{Z_1}Z_k, ..., Z_i, Z_{i+1}, ..., Z_n) \\
    & \quad - (\nabla^{n-1} f)(Z_2, ..., \nabla_{Z_1} Z_i, Z_{i+1}, ..., Z_n) - (\nabla^{n-1} f)(Z_2, ..., Z_i, \nabla_{Z_1} Z_{i+1}, ..., Z_n) \\
    & \quad - \sum_{k=i+2}^n  (\nabla^{n-1} f)(Z_2, ..., Z_i, Z_{i+1}, ..., \nabla_{Z_1}Z_k, ..., Z_n)
\end{align*}
Similarly, 
\begin{align*}
    & \quad (\nabla^n f)(Z_1, ..., Z_{i+1}, Z_i, ..., Z_n) = (\nabla_{Z_1}(\nabla^n f))(Z_2, ..., Z_{i+1}, Z_i, ..., Z_n) \\
    &= \nabla_{Z_1}((\nabla^{n-1} f)(Z_2, ..., Z_{i+1}, Z_i, ..., Z_n)) - \sum_{k=2}^{i-1}(\nabla^{n-1} f)(Z_2, ..., \nabla_{Z_1}Z_k, ..., Z_{i+1}, Z_i, ..., Z_n) \\
    & \quad - (\nabla^{n-1} f)(Z_2, ..., \nabla_{Z_1} Z_{i+1}, Z_i, ..., Z_n) - (\nabla^{n-1} f)(Z_2, ..., Z_{i+1}, \nabla_{Z_1} Z_i, ..., Z_n) \\
    & \quad - \sum_{k=i+2}^n  (\nabla^{n-1} f)(Z_2, ..., Z_{i+1}, Z_i, ..., \nabla_{Z_1}Z_k, ..., Z_n)
\end{align*}
We use the induction hypothesis to see that the difference is expressed as
\begin{align*}
    & \nabla_{Z_1}\left(\sum_{j=i+2}^n (\nabla^{n-3}f)(Z_2, ..., -R(Z_i, Z_{i+1})Z_j, ..., Z_n)  \right) \\ 
    & - \sum_{j=i+2}^n\sum_{k=2}^{i-1}(\nabla^{n-3}f)(Z_2, ..., \nabla_{Z_1}Z_k, ..., -R(Z_i, Z_{i+1})Z_j,..., Z_n) \\
    & - \sum_{j=i+2}^n(\nabla^{n-3}f)(Z_2, ..., -R(\nabla_{Z_1}Z_i, Z_{i+1})Z_j, ..., Z_n)\\
    & - \sum_{j=i+2}^n(\nabla^{n-3}f)(Z_2, ..., -R(Z_i, \nabla_{Z_1}Z_{i+1})Z_j, ..., Z_n)\\
    & -  \sum_{k=i+2}^n\sum_{j=i+2}^{k-1} (\nabla^{n-3}f)(Z_2, ..., -R(Z_i, Z_{i+1})Z_j, ..., \nabla_{Z_1}Z_k, ..., Z_n) \\
    & -\sum_{k=i+2}^n (\nabla^{n-3}f)(Z_2, ..., Z_{i+2},..., -R(Z_i, Z_{i+1})\nabla_{Z_1}Z_k, ..., Z_n)\\ 
    & -  \sum_{k=i+2}^n\sum_{j=k+1}^n(\nabla^{n-3}f)(Z_2, ..., Z_{i+2},..., \nabla_{Z_1}Z_k, ..., -R(Z_i, Z_{i+1})Z_j, ..., Z_n)
\end{align*}
which is equal to the right-hand-side. 
\end{proof}


\section{Covariant derivatives of an eigenfunction along parallel tensors}

\subsection{Terminologies and Notations}
Let $f\in C^\infty(U)_\C$ be an eigenfunction for the Laplace-Beltrami operator, i.e., 
$$\Delta f = -\lambda f. $$
We now compute the covariant derivative of $f$ along parallel tensors. By Lemma \ref{3-10}, it suffices to consider the covariant derivatives of $f$ along the $O(n, \R)$-invariant tensors and non-$O(n,\R)$-invariant tensors. 

To avoid the clumsy double-subscript, we use $\left|i\right\rangle$ to denote the vector field $X_i$. The tensor field $X_{i_1} \otimes \cdots \otimes X_{i_r}$ will be denoted by $\state{i_1 \cdots i_r}$, as well as $\state{i_1\cdots i_j}$ $\cdot \state{i_{j+1}\cdots i_r}$ and  $\state{i_1\cdots i_j}\state{i_{j+1}\cdots i_r}$, for any $j = 1,..., r-1$.

With the new notation,  
$$\theta^{\otimes k} = \left(\sum_{i=1}^n \left|ii\right\rangle\right)^{\otimes k} = \sum_{a_1=a_2=1}^n \cdots \sum_{a_{2k-1}=a_{2k}=1}^n \state{a_1a_2\cdots a_{2k-1}a_{2k}}.$$
For any $\sigma \in \text{Sym}_{2k}$, 
$$\sigma(\theta^{\otimes k}) = \sum_{a_1=a_2=1}^n \cdots \sum_{a_{2k-1}=a_{2k}=1}^n \state{a_{\sigma^{-1}(1)}\cdots a_{\sigma^{-1}(2k)}}.$$

\subsection{Along the $O(n,\R)$-invariant tensors}

\begin{thm}\label{4-2}
For any $\sigma \in \text{Sym}_{2k}$, $(\nabla^{2k} f)(\sigma(\theta^{\otimes k})) \in \C f$
\end{thm}

\begin{proof}
The proof is by induction. In the case $k=1$, any $\sigma\in \text{Sym}_2$ stabilizes $\theta$ and $\nabla^{2}f (\theta) = \Delta f = -\lambda f$. Now we assume that the conclusion holds for degree $k-1$ and argue for $k$. Based on the induction hypothesis, we first prove the following technical proposition. 
\begin{prop}\label{4-3}
For any $\sigma\in \text{Sym}_{2k}$ and any $i=1, ..., 2k-1$,  
$$(\nabla^{2k}f)(\sigma (\theta^{\otimes k})) - (\nabla^{2k}f)((i, i+1)\sigma(\theta^{\otimes k})) \in \C f.$$
\end{prop}

\noindent \textit{Proof of the Proposition.}
By definition, 
\begin{align*}
    \sigma(\theta^{\otimes k}) &= \sum_{a_1=a_2=1}^n \cdots \sum_{a_{2k-1}=a_{2k}=1}^n \state{a_{\sigma^{-1}(1)}\cdots a_{\sigma^{-1}(2k)}},\\
    (i, i+1)\sigma(\theta^{\otimes k}) &= \sum_{a_1=a_2=1}^n \cdots \sum_{a_{2k-1}=a_{2k}=1}^n \state{a_{\sigma^{-1}(1)}\cdots a_{\sigma^{-1}(i+1)}a_{\sigma^{-1}(i)}\cdots a_{\sigma^{-1}(2k)}}.
\end{align*}
In the case $i=2k-1$, we know from the first part of Theorem \ref{Comm-Cov} that \begin{align*}
    (\nabla^{2k} f)\left(\state{a_{\sigma^{-1}(1)}\cdots a_{\sigma^{-1}(2k-1)}a_{\sigma^{-1}(2k)}}\right) 
    =(\nabla^{2k} f)\left(\state{a_{\sigma^{-1}(1)}\cdots a_{\sigma^{-1}(2k)}a_{\sigma^{-1}(2k-1)}}\right).
\end{align*} 
Thus, the difference is a sum of zeros and the conclusion follows. For all other $i=1, 2,...,2k-2$, by the second part of Theorem \ref{Comm-Cov}, 
\begin{align}
    & \sum_{\substack{a_{2l-1}=a_{2l}=1\\
    l=1,...,k}}^n(\nabla^{2k} f)\left(\state{a_{\sigma^{-1}(1)}\cdots a_{\sigma^{-1}(i)}a_{\sigma^{-1}(i+1)} \cdots a_{\sigma^{-1}(2k)}}\right) \nonumber \\
    - & \sum_{\substack{a_{2l-1}=a_{2l}=1\\
    l=1,...,k}}^n (\nabla^{2k} f)\left(\state{a_{\sigma^{-1}(1)}\cdots a_{\sigma^{-1}(i+1)}a_{\sigma^{-1}(i)} \cdots  a_{\sigma^{-1}(2k)}}\right)\nonumber\\
    = & -\sum_{\substack{a_{2l-1}=a_{2l}=1\\
    l=1,...,k}}^n\sum_{j=i+2}^{2k}(\nabla^{2k-2}f)\bigg{(}\state{a_{\sigma_{-1}(1)} \cdots a_{\sigma^{-1}(i-1)} a_{\sigma^{-1}(i+2)} \cdots a_{\sigma_{-1}(j-1)}}\nonumber\\
    & \hskip 4.8cm \cdot\state{ R(a_{\sigma^{-1}(i)}a_{\sigma^{-1}(i+1)})a_{\sigma^{-1}(j)}}\state{a_{\sigma^{-1}(j+1)}\cdots a_{\sigma^{-1}(2k)}} \bigg{)}. \label{difference}
\end{align}  
The following cases arise with respect to the arrangement of $\sigma^{-1}(i)$ and $\sigma^{-1}(i+1)$: \\
\textbf{Case 1. } $\{\sigma^{-1}(i), \sigma^{-1}(i+1)\} = \{2s-1, 2s\}$ for some $s=1, ..., k$, then since $a_{2s-1}=a_{2s}$, the curvature tensor is constantly zero. So (\ref{difference}) is simply zero. \\
\textbf{Case 2. } $\sigma^{-1}(i)\in \{2s-1,2s\}, \sigma^{-1}(i+1) \in \{2t-1, 2t\}$ for different $s, t\in \{1, ..., k\}$. Without loss of generality, assume $\sigma^{-1}(i)=2s, \sigma^{-1}(i+1)=2t$. Then (\ref{difference}) becomes
\begin{align}
    -\sum_{\substack{a_{2l-1}=a_{2l}=1\\
    l=1,...,k}}^n\sum_{j=i+2}^{2k}(\nabla^{2k-2}f)\bigg{(}& \state{a_{\sigma_{-1}(1)} \cdots a_{\sigma^{-1}(i-1)} a_{\sigma^{-1}(i+2)} \cdots a_{\sigma_{-1}(j-1)}}\nonumber\\
    & \cdot\state{ R(a_{2s}a_{2t})a_{\sigma^{-1}(j)}}\state{a_{\sigma^{-1}(j+1)}\cdots a_{\sigma^{-1}(2k)}} \bigg{)}\label{difference-2}
\end{align}

There will be a few situations that requires independent treatment. \\
\textbf{Case 2.1. } Both $2s-1$ and $2t-1$ does not appear in $\{\sigma^{-1}(i+2), ..., \sigma^{-1}(2k)\}$. Then  (\ref{difference-2}) can be rewritten as
\begin{align*}
    & \quad \begin{aligned}-\sum_{\substack{a_{2l-1}=a_{2l}=1\\
    l=1,...,k}}^n\sum_{j=i+2}^{2k}(\nabla^{2k-2}f)\bigg{(}&\state{a_{\sigma_{-1}(1)} \cdots a_{2s-1} \cdots a_{2t-1} \cdots a_{\sigma^{-1}(i-1)} a_{\sigma^{-1}(i+2)} \cdots a_{\sigma_{-1}(j-1)}}\nonumber\\
    &\cdot\state{ R(a_{2s}a_{2t})a_{\sigma^{-1}(j)}}\state{a_{\sigma^{-1}(j+1)}\cdots a_{\sigma^{-1}(2k)}} \bigg{)}
    \end{aligned}\\
    & \begin{aligned}= -\sum_{j=i+2}^{2k}\sum_{\substack{a_{2l-1}=a_{2l}=1\\
    l=1,...,k\\ 
    l\neq s, l\neq t}}^n\sum_{u,v=1}^n(\nabla^{2k-2}f)\bigg{(}&\state{a_{\sigma_{-1}(1)} \cdots u \cdots v \cdots a_{\sigma^{-1}(i-1)} a_{\sigma^{-1}(i+2)} \cdots a_{\sigma_{-1}(j-1)}}\nonumber\\
    &\cdot\state{ R(uv)a_{\sigma^{-1}(j)}}\state{a_{\sigma^{-1}(j+1)}\cdots a_{\sigma^{-1}(2k)}} \bigg{)}.
    \end{aligned}
\end{align*}
    Here the order of  $a_{2s-1}$ and $a_{2t-1}$ is not important, as will be explained later. 
    
    Now for every fixed $j$,   $\sigma^{-1}(j) \in \{2m-1,2m\}$ for some $m\in \{1,..., k\}$. Without loss of generality,  assume $\sigma^{-1}(j)=2m$ and $2m-1$ appears in $\{\sigma^{-1}(1), ..., \sigma^{-1}(i-1)\}$. From the discussion in Section 2.1, 
    $$\state{R(u,v)a_{\sigma^{-1}(j)}} = \state{R(u,v)a_{2m}} = \left\{ \begin{array}{ll}
    -K\state{v} & \text{ if }a_{2m}=u,\\
    K\state{u} & \text{ if }a_{2m} = v,\\
    0 & \text{ otherwise.}
    \end{array}\right. $$
Thus the sum can be rewritten as 
\begin{align*}
    & \begin{aligned}\quad K\sum_{j=i+2}^{2k}\sum_{\substack{a_{2l-1}=a_{2l}=1\\
    l=1,...,k\\ 
    l\neq s, l\neq t, l\neq m}}^n\sum_{u,v=1}^n(\nabla^{2k-2}f)\bigg{(}&\state{a_{\sigma_{-1}(1)} \cdots u \cdots v \cdots u \cdots a_{\sigma^{-1}(i-1)} a_{\sigma^{-1}(i+2)} \cdots a_{\sigma_{-1}(j-1)}}\nonumber\\
    &\cdot\state{ v}\state{a_{\sigma^{-1}(j+1)}\cdots a_{\sigma^{-1}(2k)}} \bigg{)}
    \end{aligned}\\
    & \begin{aligned} -K\sum_{j=i+2}^{2k}\sum_{\substack{a_{2l-1}=a_{2l}=1\\
    l=1,...,k\\ 
    l\neq s, l\neq t, l\neq m}}^n\sum_{u,v=1}^n(\nabla^{2k-2}f)\bigg{(}&\state{a_{\sigma_{-1}(1)} \cdots u \cdots v \cdots v\cdots a_{\sigma^{-1}(i-1)} a_{\sigma^{-1}(i+2)} \cdots a_{\sigma_{-1}(j-1)}}\nonumber\\
    &\cdot\state{ u}\state{a_{\sigma^{-1}(j+1)}\cdots a_{\sigma^{-1}(2k)}} \bigg{)}
    \end{aligned}
\end{align*}    
    Writing $\theta^{\otimes (k-1)}$ as    $$\theta^{\otimes (k-1)} =  \sum_{\substack{a_{2l-1}=a_{2l}=1\\
    l\neq s,t,m}}^n \state{a_{1}a_{2} \cdots \widehat{a_{2s-1}a_{2s}}\cdots \widehat{a_{2t-1}a_{2t}} \cdots \widehat{a_{2m-1}a_{2m}} \cdots a_{2k-1}a_{2k}}\sum_{u,v=1}^n\state{uuvv}.$$
    we see that for each $j=i+2, ..., 2k$, the tensors
    \begin{align}
        \sum_{\substack{a_{2l-1}=a_{2l}=1\\
        l\neq s, t, m}}^n\sum_{u,v=1}^n \state{a_{\sigma_{-1}(1)} \cdots u \cdots v \cdots u \cdots a_{\sigma^{-1}(i-1)} a_{\sigma^{-1}(i+2)} \cdots  a_{\sigma_{-1}(j-1)} v a_{\sigma^{-1}(j+1)}\cdots a_{\sigma^{-1}(2k)}} \label{actingtensor-1} \\
        \sum_{\substack{a_{2l-1}=a_{2l}=1 \\
        l\neq s, t, m}}^n\sum_{u,v=1}^n \state{a_{\sigma_{-1}(1)} \cdots u \cdots v \cdots v \cdots a_{\sigma^{-1}(i-1)} a_{\sigma^{-1}(i+2)} \cdots  a_{\sigma_{-1}(j-1)} u a_{\sigma^{-1}(j+1)}\cdots a_{\sigma^{-1}(2k)}} \label{actingtensor-2} 
    \end{align}
    are both permutations of $\theta^{\otimes (k-1)}$. From the induction hypothesis of the theorem, their actions on $f$ via $\nabla^{2k-2}$ results in a scalar multiple of $f$. So summing up different $j$ results in a scalar multiple of $f$ as well. Thus the conclusion follows in this case.
    
    Now we explain why our assumptions above bring no loss of generality. 
    \begin{enumerate}
        \item If instead, $2m-1$ is in $\{\sigma^{-1}(i+2), ..., \sigma^{-1}(j-1)\}$ (or $\{\sigma^{-1}(j+1), ...,\sigma^{-1}(2k)\}$), then the corresponding tensors (\ref{actingtensor-1}) and (\ref{actingtensor-2}) are modified to  
    \begin{align*}
        \sum_{\substack{a_{2l-1}=a_{2l}=1\\
        l\neq s, t, m}}^n\sum_{u,v=1}^n \state{a_{\sigma_{-1}(1)} \cdots u \cdots v  \cdots a_{\sigma^{-1}(i-1)} a_{\sigma^{-1}(i+2)} \cdots u\cdots  a_{\sigma_{-1}(j-1)} v a_{\sigma^{-1}(j+1)}\cdots a_{\sigma^{-1}(2k)}} \\
        \sum_{\substack{a_{2l-1}=a_{2l}=1\\
        l\neq s, t, m}}^n\sum_{u,v=1}^n \state{a_{\sigma_{-1}(1)} \cdots u \cdots v  \cdots a_{\sigma^{-1}(i-1)} a_{\sigma^{-1}(i+2)} \cdots v\cdots  a_{\sigma_{-1}(j-1)} u a_{\sigma^{-1}(j+1)}\cdots a_{\sigma^{-1}(2k)}},
    \end{align*}
    or 
        \begin{align*}
        \sum_{\substack{a_{2l-1}=a_{2l}=1\\
        l\neq s, t, m}}^n\sum_{u,v=1}^n \state{a_{\sigma_{-1}(1)} \cdots u \cdots v  \cdots a_{\sigma^{-1}(i-1)} a_{\sigma^{-1}(i+2)} \cdots  a_{\sigma_{-1}(j-1)} v a_{\sigma^{-1}(j+1)}\cdots u\cdots a_{\sigma^{-1}(2k)}} \\
        \sum_{\substack{a_{2l-1}=a_{2l}=1\\
        l\neq s, t, m}}^n\sum_{u,v=1}^n \state{a_{\sigma_{-1}(1)} \cdots u \cdots v  \cdots a_{\sigma^{-1}(i-1)} a_{\sigma^{-1}(i+2)} \cdots  a_{\sigma_{-1}(j-1)} u a_{\sigma^{-1}(j+1)}\cdots v\cdots a_{\sigma^{-1}(2k)}},
    \end{align*}    
    all of which are permutations of $\theta^{\otimes k}$. 
    \item If instead, $\sigma^{-1}(j) = 2m-1$, it is easy for the reader to check that the process is indeed verbatim, as $a_{2m}=a_{2m-1}$. We will not elaborate the process here. 
    \item If the order of $a_{2s-1}$ and $a_{2t-1}$ is swapped, this amounts to be swapping the position of the first $u$ and $v$. The corresponding tensors (\ref{actingtensor-1}) and (\ref{actingtensor-2}) stay as permutations of $\theta^{\otimes k}$. 
    \item If $(\sigma^{-1}(i), \sigma^{-1}(i+1)) = (2s-1, 2t)$ (or $(2s, 2t-1)$, or $(2s-1, 2t-1)$, respectively), then with the assumption that both $2s$ and $2t-1$ (or $2s-1$ and $2t)$, or $2s$ and $2t)$, respectively) are sitting in $\{\sigma^{-1}(1), ..., \sigma^{-1}(i-1)\}$, it is also easy for the reader to find that the process repeats verbatim, as $a_{2s}=a_{2s-1}$ and $a_{2t}=a_{2t-1}$. We will not elaborate the process here.
    \end{enumerate}
    This comment on the generality applies to all the subcases discussed below and shall not be repeated henceforth.

\noindent\textbf{Case 2.2. } One of $2s-1, 2t-1$ appears in $\{\sigma^{-1}(i+2), ..., \sigma^{-1}(2k)\}$ but the other does not. Without loss of generality, say $2s-1 \in\{\sigma^{-1}(1), ..., \sigma^{-1}(i-1)\}$ and $2t-1 = \sigma^{-1}(j_0)$ for some $j_0\geq i+2$. Then (\ref{difference}) becomes 
\begin{align}
    & \begin{aligned}-\sum_{\substack{a_{2l-1}=a_{2l}=1\\
    l=1,...,k}}^n\sum_{j=i+2}^{j_0-1}(\nabla^{2k-2}f)\bigg{(}& \state{a_{\sigma_{-1}(1)}\cdots a_{2s-1} \cdots a_{\sigma^{-1}(i-1)} a_{\sigma^{-1}(i+2)} \cdots a_{\sigma_{-1}(j-1)}}\\
    & \cdot\state{ R(a_{2s}a_{2t})a_{\sigma^{-1}(j)}}\state{a_{\sigma^{-1}(j+1)}\cdots a_{2t-1} \cdots  a_{\sigma^{-1}(2k)}} \bigg{)}
    \end{aligned}\label{difference-2-1}\\
    & \begin{aligned}-\sum_{\substack{a_{2l-1}=a_{2l}=1\\
    l=1,...,k}}^n(\nabla^{2k-2}f)\bigg{(}& \state{a_{\sigma_{-1}(1)}\cdots a_{2s-1} \cdots a_{\sigma^{-1}(i-1)} a_{\sigma^{-1}(i+2)} \cdots a_{\sigma_{-1}(j-1)}}\\
    & \cdot\state{ R(a_{2s}a_{2t})a_{2t-1}}\state{a_{\sigma^{-1}(j+1)}\cdots \cdots  a_{\sigma^{-1}(2k)}} \bigg{)}
    \end{aligned}\label{difference-2-2}\\
    & \begin{aligned}-\sum_{\substack{a_{2l-1}=a_{2l}=1\\
    l=1,...,k}}^n\sum_{j=j_0+1}^{2k}(\nabla^{2k-2}f)\bigg{(}& \state{a_{\sigma_{-1}(1)}\cdots a_{2s-1} \cdots a_{\sigma^{-1}(i-1)} a_{\sigma^{-1}(i+2)} \cdots a_{2t-1} \cdots a_{\sigma_{-1}(j-1)}}\\
    & \cdot\state{ R(a_{2s}a_{2t})a_{\sigma^{-1}(j)}}\state{a_{\sigma^{-1}(j+1)} \cdots  a_{\sigma^{-1}(2k)}} \bigg{)}.
    \end{aligned} \label{difference-2-3}
\end{align}
The sums (\ref{difference-2-1}) and (\ref{difference-2-3}) can be similarly handled as in Case 2.1. For the sum (\ref{difference-2-2}), using the fact that $a_{2t}=a_{2t-1}$, we have $$\state{R(a_{2s}a_{2t})a_{2t-1}} =\left\{\begin{array}{ll}
K\state{a_{2s}} & \text{ if }a_{2t} = a_{2t-1}\neq a_{2s},\\
0 & \text{ if }a_{2t}=a_{2t-1}=a_{2s}. \end{array}\right. $$ 
Thus (\ref{difference-2-2}) becomes 
\begin{align*}
-K \sum_{\substack{a_{2l-1}=a_{2l}=1\\
    l=1,...,k\\
    a_{2s}\neq a_{2t}}}^n(\nabla^{2k-2}f)\bigg{(}& \state{a_{\sigma_{-1}(1)}\cdots a_{2s-1} \cdots a_{\sigma^{-1}(i-1)} a_{\sigma^{-1}(i+2)} \cdots a_{\sigma_{-1}(j-1)}}\\
    & \cdot\state{ a_{2s}}\state{a_{\sigma^{-1}(j+1)}\cdots \cdots  a_{\sigma^{-1}(2k)}} \bigg{)},
\end{align*}
which can be further simplified as
\begin{align*}
-K(n-1) \sum_{\substack{a_{2l-1}=a_{2l}=1\\
    l=1,...,k, l \neq t}}^n(\nabla^{2k-2}f)\bigg{(}& \state{a_{\sigma_{-1}(1)}\cdots a_{2s-1} \cdots a_{\sigma^{-1}(i-1)} a_{\sigma^{-1}(i+2)} \cdots a_{\sigma_{-1}(j_0-1)}}\\
    & \cdot\state{a_{2s}a_{\sigma^{-1}(j_0+1)}\cdots \cdots  a_{\sigma^{-1}(2k)}} \bigg{)}.
\end{align*}
If we write $\theta^{\otimes (k-1)}$ as
$$\theta^{\otimes (k-1)} = \sum_{\substack{a_{2l-1}=a_{2l}=1\\l=1,..., k,l\neq t}}^n \state{a_1a_2 \cdots a_{2s-1}a_{2s} \cdots \widehat{a_{2t-1}a_{2t}}\cdots a_{2k-1}a_{2k}},$$
then it is immediately seen that the sum (\ref{difference-2-2}) is the action of a permutation of $\theta^{k-1}$. By induction hypothesis, (\ref{difference-2-2}) is also a scalar multiple of $f$. 

    
\noindent\textbf{Case 2.3. } Both $2s-1$ and $2t-1$ appear in $\{\sigma^{-1}(i+2), ..., \sigma^{-1}(2k)\}$. Without loss of generality, assume $\sigma^{-1}(j_1)=2s-1, \sigma^{-1}(j_2)=2t-1$ with $j_1 < j_2$. Then (\ref{difference-2}) can be written as
\begin{align}
    & \begin{aligned}
    -\sum_{\substack{a_{2l-1}=a_{2l}=1\\
    l=1,...,k}}^n\sum_{j=i+2}^{j_1-1}(\nabla^{2k-2}f)\bigg{(}& \state{a_{\sigma_{-1}(1)} \cdots a_{\sigma^{-1}(i-1)} a_{\sigma^{-1}(i+2)} \cdots a_{\sigma_{-1}(j-1)}}\\
    & \cdot\state{ R(a_{2s}a_{2t})a_{\sigma^{-1}(j)}}\state{a_{\sigma^{-1}(j+1)}\cdots a_{2s-1} \cdots a_{2t-1} \cdots a_{\sigma^{-1}(2k)}} \bigg{)}
\end{aligned}\label{difference-2-4}\\
& \begin{aligned}
    -\sum_{\substack{a_{2l-1}=a_{2l}=1\\
    l=1,...,k}}^n(\nabla^{2k-2}f)\bigg{(}& \state{a_{\sigma_{-1}(1)} \cdots a_{\sigma^{-1}(i-1)} a_{\sigma^{-1}(i+2)} \cdots a_{\sigma_{-1}(j-1)}}\\
    & \cdot\state{ R(a_{2s}a_{2t})a_{2s-1}}\state{a_{\sigma^{-1}(j+1)}\cdots a_{2t-1}\cdots a_{\sigma^{-1}(2k)}} \bigg{)}
\end{aligned}\label{difference-2-5}\\
& \begin{aligned}
    -\sum_{\substack{a_{2l-1}=a_{2l}=1\\
    l=1,...,k}}^n\sum_{j=j_1+1}^{j_2-1}(\nabla^{2k-2}f)\bigg{(}& \state{a_{\sigma_{-1}(1)} \cdots a_{\sigma^{-1}(i-1)} a_{\sigma^{-1}(i+2)} \cdots a_{2s-1}\cdots a_{\sigma_{-1}(j-1)}}\\
    & \cdot\state{ R(a_{2s}a_{2t})a_{\sigma^{-1}(j)}}\state{a_{\sigma^{-1}(j+1)}\cdots a_{2t-1}\cdots a_{\sigma^{-1}(2k)}} \bigg{)}
\end{aligned}\label{difference-2-6}\\
& \begin{aligned}
    -\sum_{\substack{a_{2l-1}=a_{2l}=1\\
    l=1,...,k}}^n(\nabla^{2k-2}f)\bigg{(}& \state{a_{\sigma_{-1}(1)} \cdots a_{\sigma^{-1}(i-1)} a_{\sigma^{-1}(i+2)} \cdots a_{2s-1} \cdots a_{\sigma_{-1}(j-1)}}\\
    & \cdot\state{ R(a_{2s}a_{2t})a_{2t-1}}\state{a_{\sigma^{-1}(j+1)}\cdots a_{\sigma^{-1}(2k)}} \bigg{)}
\end{aligned}\label{difference-2-7}\\
& \begin{aligned}
    -\sum_{\substack{a_{2l-1}=a_{2l}=1\\
    l=1,...,k}}^n\sum_{j=i+2}^{2k}(\nabla^{2k-2}f)\bigg{(}& \state{a_{\sigma_{-1}(1)} \cdots a_{\sigma^{-1}(i-1)} a_{\sigma^{-1}(i+2)} \cdots a_{2s-1}\cdots a_{2t-1} \cdots a_{\sigma_{-1}(j-1)}}\\
    & \cdot\state{ R(a_{2s}a_{2t})a_{\sigma^{-1}(j)}}\state{a_{\sigma^{-1}(j+1)}\cdots a_{\sigma^{-1}(2k)}} \bigg{)}.
\end{aligned}\label{difference-2-8}
\end{align}
The sums (\ref{difference-2-4}), (\ref{difference-2-6}), and (\ref{difference-2-8}) can be similarly processed as in Case 2.1, while the sums (\ref{difference-2-5}) and (\ref{difference-2-7}) can be processed similarly as in Case 2.2 with trivial modifications. We shall not repeat the discussion here. If instead, $j_1 > j_2$, this amounts to be swapping $a_{2s-1}$ and $a_{2t-1}$. The process is still similar.


\noindent\textit{Proof of the Theorem.} We first note that 
$$(\nabla^{2k}f)(\theta^{\otimes k}) = (-\lambda)^k f\in \C f. $$
Since $t_1 = (12), t_2 = (23), ..., t_{2k-1} = (2k-1, 2k)$ generates $\text{Sym}_{2k}$, any $\sigma\in \text{Sym}_{2k}$ can be written as a product $t_{i_k}\cdots t_{i_1}$. The proposition then allows us to conclude that
$$(\nabla^{2k} f)(t_{i_{j+1}} t_{i_j}\cdots t_{i_1} \theta^{\otimes k}) - (\nabla^{2k}f)(t_{i_j} \cdots t_{i_1} \theta^{\otimes k}) \in \C f.$$
for every $j = 0, ..., k-1$. Summed up with all $j$'s, we see that
$$(\nabla^{2k}f)(\sigma \theta^{\otimes k}) - (\nabla^{2k}f)(\theta^{\otimes k}) \in \C f.$$
It then follows from $(\nabla^{2k}f)(\theta^{\otimes k})\in \C f$ that 
$$(\nabla^{2k}f)(\sigma (\theta^{\otimes k}))\in \C f. $$
\end{proof}

\begin{rema}
Indeed the same argument shows that for every $\sigma\in \text{Sym}_{2k}$, the action of $\sigma(\theta^k)$ on $f$ is the same as the action of some polynomial in $\theta$ determined by $\sigma$. In the next section we will discuss the properties of such polynomials. 
\end{rema}

\subsection{Along the non-$O(n,\R)$-invariant tensors}

Now we look at the action of the other type of parallel tensors. We have the following theorem: 
\begin{thm}
For any $\sigma\in \text{Sym}_{2k+n}$, 
$$(\nabla^{2k+n}f) (\sigma (\theta^{\otimes k}\cdot \Lambda)) = 0. $$
\end{thm}

\begin{proof}
We argue by induction on $k$. For the base case $k=0$, it suffices to argue that 
$$(\nabla^{n}f)(\Lambda) = 0. $$
as permutations on $\Lambda$ only changes the sign of the tensor. Using the notation introduced in the proof of Proposition \ref{4-3}, 
$$\Lambda = \sum_{\mu\in \text{Sym}_n}(-1)^\mu \state{\mu(1) \cdots \mu(n)}.$$
Let $A_n$ be the subgroup formed by the even permutations. Then 
$$\text{Sym}_n = A_n \cup (n-1, n)A_n.$$ 
This allows us to write 
$$\Lambda = \sum_{\mu\in A_n} \state{\mu(1) \cdots \mu(n-1)\mu(n)} - \sum_{\mu\in A_n} \state{\mu(1) \cdots \mu(n)\mu(n-1)}. $$
Note that for every $\mu\in A_n$, it follows from the first part of Theorem \ref{Comm-Cov} that
$$(\nabla^n f)(\state{\mu(1) \cdots \mu(n-1)\mu(n)}) = (\nabla^n f)(\state{\mu(1) \cdots \mu(n)\mu(n-1)}).$$
This shows $(\nabla^n f)(\Lambda)$ is simply a sum of zeros. Thus follows the conclusion. 

Now assume the conclusion holds for $k-1$. Based on the induction hypothesis, we similar proceed to prove the following proposition. 

\begin{prop}\label{4-5}
For any $\sigma\in \text{Sym}_{2k+n}$ and any $i=1, ..., 2k+n-1$,  
$$(\nabla^{2k}f)(\sigma (\theta^{\otimes k}\otimes \Lambda)) - (\nabla^{2k}f)((i, i+1)\sigma(\theta^{\otimes k}\otimes \Lambda)) = 0.$$
\end{prop}

\noindent\textit{Proof of the Proposition. }
Let 
$$\sigma = \begin{pmatrix} 1 & 2 & \cdots & 2k & 2k+1 & 2k+2 & \cdots & 2k+n\\
\sigma(1) & \sigma(2) & \cdots & \sigma(2k) & i_1 & i_2 &\cdots & i_n
\end{pmatrix}.$$
Without loss of generality, we assume that $i_1 < i_2 < \cdots < i_n$, as permuting $i_1, ..., i_n$ only brings a sign change to the tensor $\sigma(\theta^{\otimes k} \otimes \Lambda)$. Then we can write the tensor $\sigma(\theta^{\otimes k} \otimes \Lambda)$ as 
\begin{align*}
    \sum_{\substack{a_{2l-1}=a_{2l}=1\\l=1,...,k}}^n \sum_{\mu\in \text{Sym}_n}(-1)^\mu \state{ a_{\sigma^{-1}(1)} \cdots a_{\sigma^{-1}(i_1-1)}\mu(1) a_{\sigma^{-1}(i_1+1)}\cdots \cdots a_{\sigma^{-1}(i_{n}-1)}\mu(n) a_{\sigma^{-1}(i_n+1)} \cdots a_{\sigma^{-1}(2k+n)}  }.
\end{align*}
Regarding the transposition $(i,i+1)$, essentially there are four cases to consider. For exposition purposes, we will first study four special subcases, then generalize. 

\noindent \textbf{Case 1. } $i\geq i_n+1$. In this case, 
\begin{align*}
    & \quad (\nabla^{2k+n}f)(\sigma(\theta^{\otimes k}\otimes \Lambda)) - (\nabla^{2k+n}f)((i,i+1)\sigma(\theta^{\otimes k}\otimes \Lambda)) \\
    & \begin{aligned}
    =  - \sum_{\substack{a_{2l-1}=a_{2l}=1\\l=1,...,k}}^n \sum_{\mu\in \text{Sym}_n}\sum_{j=i+2}^{2k+n} (-1)^\mu (\nabla^{2k+n-2}f)&\bigg{(}\state{ a_{\sigma^{-1}(1)} \cdots \mu(1) \cdots \cdots \mu(n) a_{\sigma^{-1}(i_n+1)}\cdots a_{\sigma^{-1}(j-1)}}\\ & \cdot \state{R(a_{\sigma^{-1}(i)}, a_{\sigma^{-1}(i+1)})a_{\sigma^{-1}(j)}}\state{a_{\sigma^{-1}(j+1)}\cdots a_{\sigma^{-1}(2k+n)}  }\bigg{)}. 
    \end{aligned}
\end{align*}
For each fixed $\mu\in \text{Sym}_n$, we can take care of the sum similarly as in Proposition \ref{4-3}. The only difference here lies on the occurrence of $\mu(1), ..., \mu(n)$ in the front, which does not change the process. Finally, $(\nabla^{2k+n}f)(\sigma(\theta^{\otimes k} \otimes \Lambda))- (\nabla^{2k+n}f)(\sigma(\theta^{\otimes k} \otimes \Lambda))$ is the same as an action of some $(2k+n-2)$-degree tensor of $f$
which is a permutation of $\theta^{\otimes (k-1)}\otimes \Lambda$. By the induction hypothesis, $(\nabla^{2k+n}f)(\sigma(\theta^{\otimes k} \otimes \Lambda))- (\nabla^{2k+n}f)(\sigma(\theta^{\otimes k} \otimes \Lambda)) = 0$. 

\noindent\textbf{Case 2. }$i = i_n.$ In this case, 
\begin{align}
    & \quad (\nabla^{2k+n}f)(\sigma(\theta^{\otimes k}\otimes \Lambda)) - (\nabla^{2k+n}f)((i,i+1)\sigma(\theta^{\otimes k}\otimes \Lambda)) \nonumber\\
    & \begin{aligned}
    =  - \sum_{\substack{a_{2l-1}=a_{2l}=1\\l=1,...,k}}^n \sum_{\mu\in \text{Sym}_n}\sum_{j=i+2}^{2k+n} (-1)^\mu (\nabla^{2k+n-2}f)&\bigg{(}\state{ a_{\sigma^{-1}(1)} \cdots \mu(1) \cdots \cdots \mu(n-1) \cdots a_{\sigma^{-1}(i_n-1)}}\\ & \cdot \state{a_{\sigma(i_n+2)}\cdots a_{\sigma^{-1}(j-1)}}\state{R(\mu(n), a_{\sigma^{-1}(i_n+1)})a_{\sigma^{-1}(j)}}\\
    & \cdot \state{a_{\sigma^{-1}(j+1)}\cdots a_{\sigma^{-1}(2k+n)}  }\bigg{)}. 
    \end{aligned}\label{difference-3}
\end{align}
The process is similar to Proposition \ref{4-3} but with a slightly less trivial modification. For exposition purposes, we will go through the details here. 

Let $\sigma^{-1}(i_{n+1})\in \{2s-1, 2s\}$ for some $s=1, ..., k$. Without loss of generality, assume $\sigma^{-1}(i_{n+1})=2s$. 

\noindent\textbf{Case 2.1. } 
$2s-1$ appears in $\{\sigma^{-1}(1), ..., \sigma^{-1}(i_{n}-1)\}$. Then (\ref{difference-3}) can be rewritten as
\begin{align}
    & \begin{aligned}
    -  \sum_{j=i+2}^{2k+n}\sum_{\substack{a_{2l-1}=a_{2l}=1\\l=1,...,k}}^n \sum_{\mu\in \text{Sym}_n} (-1)^\mu (\nabla^{2k+n-2}f)&\bigg{(}\state{ a_{\sigma^{-1}(1)} \cdots \mu(1) \cdots a_{2s-1} \cdots \mu(n-1) \cdots a_{\sigma^{-1}(i_n-1)}}\\ & \cdot \state{a_{\sigma(i_n+2)}\cdots a_{\sigma^{-1}(j-1)}}\state{R(\mu(n), a_{2s})a_{\sigma^{-1}(j)}}\\
    & \cdot \state{a_{\sigma^{-1}(j+1)}\cdots a_{\sigma^{-1}(2k+n)}  }\bigg{)}. 
    \end{aligned}\label{difference-3-1}
\end{align}
Here the exact position of $a_{2s-1}$ is not important. For each fixed $j\in \{i+2, ..., 2k+n\}$, let $\sigma^{-1}(j)\in \{2t-1,2t\}$ for some $t=\{1,...,k\}$ depending on $j$. Without loss of generality, assume $\sigma^{-1}(j)=2t$ and $2t-1$ also also appears in $\{\sigma^{-1}(1), ..., \sigma^{-1}(i_n-1)\}$. So  (\ref{difference-3-1}) can be computed as follows
\begin{align}
    & \begin{aligned}
    \quad  - \sum_{j=i+2}^{2k+n}\sum_{\substack{a_{2l-1}=a_{2l}=1\\l=1,...,k}}^n \sum_{\mu\in \text{Sym}_n} (-1)^\mu (\nabla^{2k+n-2}f)&\bigg{(}\state{ a_{\sigma^{-1}(1)} \cdots \mu(1) \cdots a_{2s-1} \cdots a_{2t-1} \cdots \mu(n-1) \cdots a_{\sigma^{-1}(i_n-1)}}\\ & \cdot \state{a_{\sigma(i_n+2)}\cdots a_{\sigma^{-1}(j-1)}}\state{R(\mu(n), a_{2s})a_{2t}}\\
    & \cdot \state{a_{\sigma^{-1}(j+1)}\cdots a_{\sigma^{-1}(2k+n)}  }\bigg{)} 
    \end{aligned}\nonumber\\
    & \begin{aligned}
    = - \sum_{j=i+2}^{2k+n}\sum_{\substack{a_{2l-1}=a_{2l}=1\\l=1,...,k\\ l\neq s,l\neq t}}^n \sum_{\mu\in \text{Sym}_n}\sum_{u,v=1}^n (-1)^\mu (\nabla^{2k+n-2}f)&\bigg{(}\state{ a_{\sigma^{-1}(1)} \cdots \mu(1) \cdots u \cdots v \cdots \mu(n-1) \cdots a_{\sigma^{-1}(i_n-1)}}\\ & \cdot \state{a_{\sigma(i_n+2)}\cdots a_{\sigma^{-1}(j-1)}}\state{R(\mu(n), u)v}\\
    & \cdot \state{a_{\sigma^{-1}(j+1)}\cdots a_{\sigma^{-1}(2k+n)}  }\bigg{)}. \end{aligned}\label{difference-3-2}
    \end{align}
Here again the exact position of $a_{2t-1}$ is not important. Note that $$\state{R(\mu(n), u)v} = \left\{\begin{array}{ll} 
K \state{\mu(n)} & \text{ if }u=v\neq \mu(n)\\
-K \state{u} & \text{ if }v=\mu(n)\\
0 & \text{ otherwise. }
\end{array}\right.$$ 
Excluding the zero terms in (\ref{difference-3-2}), the sum is equal to 
    \begin{align}
    & \begin{aligned}
    - K(n-1) \sum_{j=i+2}^{2k+n}\sum_{\substack{a_{2l-1}=a_{2l}=1\\l=1,...,k\\ l\neq s,l\neq t}}^n \sum_{\mu\in \text{Sym}_n}\sum_{u=1}^n (-1)^\mu (\nabla^{2k+n-2}f)&\bigg{(}\state{ a_{\sigma^{-1}(1)} \cdots \mu(1) \cdots u \cdots u \cdots \mu(n-1) \cdots a_{\sigma^{-1}(i_n-1)}}\\ & \cdot \state{a_{\sigma(i_n+2)}\cdots a_{\sigma^{-1}(j-1)}\mu(n)}\\
    & \cdot \state{a_{\sigma^{-1}(j+1)}\cdots a_{\sigma^{-1}(2k+n)}  }\bigg{)} \end{aligned}\label{difference-3-3}\\
     & \begin{aligned}
    + K \sum_{j=i+2}^{2k+n}\sum_{\substack{a_{2l-1}=a_{2l}=1\\l=1,...,k\\ l\neq s,l\neq t}}^n \sum_{\mu\in \text{Sym}_n}\sum_{u=1}^n (-1)^\mu (\nabla^{2k+n-2}f)&\bigg{(}\state{ a_{\sigma^{-1}(1)} \cdots \mu(1) \cdots u \cdots \mu(n) \cdots \mu(n-1) \cdots a_{\sigma^{-1}(i_n-1)}}\\ & \cdot \state{a_{\sigma(i_n+2)}\cdots a_{\sigma^{-1}(j-1)}u}\state{a_{\sigma^{-1}(j+1)}\cdots a_{\sigma^{-1}(2k+n)}  }\bigg{)}. \end{aligned}\label{difference-3-4}
\end{align}
Writing $\theta^{\otimes (k-1)}\otimes \Lambda$ as
$$\sum_{\substack{a_{2l-1}=a_{2l}=1\\l=1,...,k\\l\neq s,l\neq t}}^n\state{a_1a_2 \cdots \widehat{a_{2s-1}a_{2s}} \cdots \widehat{a_{2t-1}a_{2t}} \cdots a_{2k-1}a_{2k}}\sum_{u=1}^n \state{uu}\sum_{\mu\in \text{Sym}_n}(-1)^\mu \state{\mu(1) \cdots \mu(n)},$$
we see that for each $j=i+2, ..., 2k+n$, the corresponding summand in both (\ref{difference-3-3}) and (\ref{difference-3-4})  coincide with a scalar multiple of the action of some permutation of $\theta^{\otimes (k-1)}\otimes \Lambda$ on $f$. By induction hypothesis, both (\ref{difference-3-3}) and (\ref{difference-3-4}) are then a sum of zeros. The conclusion then follows in this case. 

Now we explain why our assumptions bring no loss of generality. 

\begin{enumerate}
    \item If instead, $\sigma^{-1}(i_{n}) = 2s-1$, then with the assumption that $2s\in \{\sigma^{-1}(1), ..., \sigma^{-1}(i_n-1)\}$, the discussion is verbatim, as $a_{2s}=a_{2s-1}$. 
    \item If $\sigma^{-1}(j)=2t$ and $2t-1$ now appears in $\{\sigma^{-1}(i_n+1), ..., \sigma^{-1}(2k+n)\}$, this amounts to be moving the first $v$ in (\ref{difference-3-2}) to either the second line or the third line. Thus, the second $u$ in (\ref{difference-3-3}) is moved to the second line or the third line, and $\mu(n)$ in (\ref{difference-3-4}) is moved into the second line, before $u$ or after $u$. The conclusion that (\ref{difference-3-3}) and (\ref{difference-3-4}) are actions of permutations of $\theta^{k-1}\otimes \Lambda$ is unchanged. 
    \item If $\sigma^{-1}(j)=2t-1$ and $2t$ appears anywhere, the discussion is also verbatim, as $a_{2t-1} = a_{2t}$.
\end{enumerate}
This comment applies to all the cases to be discussed below and shall not be repeated. 

\noindent\textbf{Case 2.2. }$\sigma^{-1}(i_n+1)=2s$ for some $s=1,..., k$ and $2s-1$ appears in $\{\sigma^{-1}(i_n+2),...,\sigma^{-1}(2k+n)\}$. Say $\sigma^{-1}(j_0)=2s-1$. Then (\ref{difference-3}) becomes
\begin{align}
    & \begin{aligned}
    - \sum_{\substack{a_{2l-1}=a_{2l}=1\\l=1,...,k}}^n \sum_{\mu\in \text{Sym}_n}\sum_{j=i+2}^{j_0-1} (-1)^\mu (\nabla^{2k+n-2}f)&\bigg{(}\state{ a_{\sigma^{-1}(1)} \cdots \mu(1) \cdots \cdots \mu(n-1) \cdots a_{\sigma^{-1}(i_n-1)}}\\ & \cdot \state{a_{\sigma(i_n+2)}\cdots a_{\sigma^{-1}(j-1)}}\state{R(\mu(n), a_{2s})a_{\sigma^{-1}(j)}}\\
    & \cdot \state{a_{\sigma^{-1}(j+1)}\cdots a_{2s-1}\cdots a_{\sigma^{-1}(2k+n)}  }\bigg{)} 
    \end{aligned}\label{difference-3-5}\\
    & \begin{aligned}
    - \sum_{\substack{a_{2l-1}=a_{2l}=1\\l=1,...,k}}^n \sum_{\mu\in \text{Sym}_n} (-1)^\mu (\nabla^{2k+n-2}f)&\bigg{(}\state{ a_{\sigma^{-1}(1)} \cdots \mu(1) \cdots \cdots \mu(n-1) \cdots a_{\sigma^{-1}(i_n-1)}}\\ & \cdot \state{a_{\sigma(i_n+2)}\cdots a_{\sigma^{-1}(j-1)}}\state{R(\mu(n), a_{2s})a_{2s-1}}\\
    & \cdot \state{a_{\sigma^{-1}(j+1)}\cdots a_{\sigma^{-1}(2k+n)}  }\bigg{)} 
    \end{aligned}\label{difference-3-6}\\
    & \begin{aligned}
    - \sum_{\substack{a_{2l-1}=a_{2l}=1\\l=1,...,k}}^n \sum_{\mu\in \text{Sym}_n}\sum_{j=j_0+1}^{2k+n} (-1)^\mu (\nabla^{2k+n-2}f)&\bigg{(}\state{ a_{\sigma^{-1}(1)} \cdots \mu(1) \cdots \cdots \mu(n-1) \cdots a_{\sigma^{-1}(i_n-1)}}\\ & \cdot \state{a_{\sigma(i_n+2)}\cdots a_{2s-1}\cdots a_{\sigma^{-1}(j-1)}}\state{R(\mu(n), a_{2s})a_{\sigma^{-1}(j)}}\\
    & \cdot \state{a_{\sigma^{-1}(j+1)}\cdots a_{\sigma^{-1}(2k+n)}  }\bigg{)}. 
    \end{aligned}\label{difference-3-7}
\end{align}
The sums (\ref{difference-3-5}) and (\ref{difference-3-7}) can be handled similarly as in Case 2.1. For the sum (\ref{difference-3-6}), note that $\state{R(\mu(n),a_{2s})a_{2s-1}}$ is nonzero only when $a_{2s}=a_{2s-1}\neq \mu(n)$, with the value being $K\state{\mu(n)}$. Thus (\ref{difference-3-4}) is simply %
\begin{align*}
    - K(n-1) \sum_{\substack{a_{2l-1}=a_{2l}=1\\l=1,...,k,l\neq s}}^n \sum_{\mu\in \text{Sym}_n} (-1)^\mu (\nabla^{2k+n-2}f)&\bigg{(}\state{ a_{\sigma^{-1}(1)} \cdots \mu(1) \cdots \cdots \mu(n-1) \cdots a_{\sigma^{-1}(i_n-1)}}\\ & \cdot \state{a_{\sigma(i_n+2)}\cdots a_{\sigma^{-1}(j-1)}\mu(n)a_{\sigma^{-1}(j+1)}\cdots a_{\sigma^{-1}(2k+n)}  }\bigg{)}. 
\end{align*}
which coincides with the action of a permutation of $\theta^{\otimes (k-1)}\otimes \Lambda$ and thus is zero. 

\noindent\textbf{Case 3. }$i=i_{n}-1>i_{n-1}$. In this case,
\begin{align*}
    & \quad(\nabla^{2k+n}f)(\sigma(\theta^{\otimes k}\otimes \Lambda) - (\nabla^{2k+n}f)((i,i+1)\sigma(\theta^{\otimes k}\otimes \Lambda) \\
    & \begin{aligned}
    = - \sum_{\substack{a_{2l-1}=a_{2l}=1\\l=1,...,k}}^n \sum_{\mu\in \text{Sym}_n}\sum_{j=i+2}^{2k+n} (-1)^\mu (\nabla^{2k+n-2}f)&\bigg{(}\state{ a_{\sigma^{-1}(1)} \cdots \mu(1) \cdots \cdots \mu(n-1) \cdots a_{\sigma^{-1}(i_n-2)} }\\ & \cdot\state{a_{\sigma^{-1}(i_n+1)} \cdots a_{\sigma^{-1}(j-1)}} \state{R(a_{\sigma^{-1}(i_n-1)}, \mu(n))a_{\sigma^{-1}(j)} }\\& \cdot \state{a_{\sigma^{-1}(j+1)}\cdots a_{\sigma^{-1}(2k+n)}  }\bigg{)}. 
    \end{aligned}
\end{align*}
Using $R(a_{\sigma^{-1}(i_n-1)},\mu(n)) = -R(\mu(n),a_{\sigma^{-1}(i_n-1)})$, we can apply the same procedure as we did for Case 2. We shall not repeat the details here. 

\noindent\textbf{Case 4.} $i=i_n-1 = i_{n-1}$. In this case, 
\begin{align}
    & \quad (\nabla^{2k+n}f)(\sigma(\theta^{\otimes k}\otimes \Lambda) - (\nabla^{2k+n}f)((i,i+1)\sigma(\theta^{\otimes k}\otimes \Lambda) \nonumber\\
    & \begin{aligned}
    = - \sum_{\substack{a_{2l-1}=a_{2l}=1\\l=1,...,k}}^n \sum_{\mu\in \text{Sym}_n}\sum_{j=i+2}^{2k+n} (-1)^\mu (\nabla^{2k+n-2}f)&\bigg{(}\state{ a_{\sigma^{-1}(1)} \cdots \mu(1) \cdots \cdots \mu(n-2) \cdots a_{\sigma^{-1}(i_n-2)} }\\ & \cdot\state{a_{\sigma^{-1}(i_n+1)} \cdots a_{\sigma^{-1}(j-1)}} \state{R(\mu(n-1), \mu(n))a_{\sigma^{-1}(j)} }\\& \cdot \state{a_{\sigma^{-1}(j+1)}\cdots a_{\sigma^{-1}(2k+n)}  }\bigg{)}. 
    \end{aligned}\label{difference-4}
\end{align}
We fix $j$. Without loss of generality, let $\sigma^{-1}(j) = 2s$ and let $2s-1$ appears before the curvature tensor. Note that $R(\mu(n-1),\mu(n))a_{2s}$ is nonzero only when $a_{2s}=\mu(n-1)$ and $a_{2s}=\mu(n)$. Thus
(\ref{difference-4}) simplifies as 
\begin{align*}
    & \begin{aligned}
    \quad - \sum_{\substack{a_{2l-1}=a_{2l}=1\\l=1,...,k}}^n \sum_{\mu\in \text{Sym}_n}\sum_{j=i+2}^{2k+n} (-1)^\mu (\nabla^{2k+n-2}f)&\bigg{(}\state{ a_{\sigma^{-1}(1)} \cdots \mu(1) \cdots a_{2s-1} \cdots \mu(n-2) \cdots a_{\sigma^{-1}(i_n-2)} }\\ & \cdot\state{a_{\sigma^{-1}(i_n+1)} \cdots a_{\sigma^{-1}(j-1)}} \state{R(\mu(n-1), \mu(n))a_{2s} }\\& \cdot \state{a_{\sigma^{-1}(j+1)}\cdots a_{\sigma^{-1}(2k+n)}  }\bigg{)} 
    \end{aligned}\\
    & \begin{aligned}
    = K \sum_{\substack{a_{2l-1}=a_{2l}=1\\l=1,...,k\\l\neq s}}^n \sum_{\mu\in \text{Sym}_n}\sum_{j=i+2}^{2k+n} (-1)^\mu (\nabla^{2k+n-2}f)&\bigg{(}\state{ a_{\sigma^{-1}(1)} \cdots \mu(1) \cdots \cdots \mu(n-1) \cdots \mu(n-2) \cdots a_{\sigma^{-1}(i_n-2)} }\\ & \cdot\state{a_{\sigma^{-1}(i_n+1)} \cdots a_{\sigma^{-1}(j-1)}} \state{\mu(n)}\\& \cdot \state{a_{\sigma^{-1}(j+1)}\cdots a_{\sigma^{-1}(2k+n)}  }\bigg{)} 
    \end{aligned}\\
    & \begin{aligned}
    \quad - K \sum_{\substack{a_{2l-1}=a_{2l}=1\\l=1,...,k}}^n \sum_{\mu\in \text{Sym}_n}\sum_{j=i+2}^{2k+n} (-1)^\mu (\nabla^{2k+n-2}f)&\bigg{(}\state{ a_{\sigma^{-1}(1)} \cdots \mu(1) \cdots \cdots \mu(n) \cdots \mu(n-2) \cdots a_{\sigma^{-1}(i_n-2)} }\\ & \cdot\state{a_{\sigma^{-1}(i_n+1)} \cdots a_{\sigma^{-1}(j-1)}} \state{\mu(n-1)}\\& \cdot \state{a_{\sigma^{-1}(j+1)}\cdots a_{\sigma^{-1}(2k+n)}  }\bigg{)}. 
    \end{aligned}
\end{align*}
Each sum is zero by the same reason at the end of the previous cases. 

Having studied the special cases above, we now investigate their generalizations. 

\noindent\textbf{Case 5. }$i_{m-1}+1 \leq i \leq i_{m}-2$ for each $m = 1, 2, ..., n$ (here we regard $i_0=0$). In this case, 
\begin{align}
    & \quad (\nabla^{2k+n}f)(\sigma(\theta^{\otimes k}\otimes \Lambda) - (\nabla^{2k+n}f)((i,i+1)\sigma(\theta^{\otimes k}\otimes \Lambda) \nonumber\\
    & \begin{aligned}
    = - \sum_{\substack{a_{2l-1}=a_{2l}=1\\l=1,...,k}}^n \sum_{\mu\in \text{Sym}_n}\sum_{j=i+2}^{i_{m}-1} (-1)^\mu (\nabla^{2k+n-2}f)&\bigg{(}\state{ a_{\sigma^{-1}(1)} \cdots \mu(1) \cdots \cdots \mu(m-1) a_{\sigma^{-1}(i_{m-1}+1)}\cdots a_{\sigma^{-1}(i-1)}}\\ & \cdot
    \state{a_{\sigma^{-1}(i+2)}\cdots a_{\sigma^{-1}(j-1)}}\state{R(a_{\sigma^{-1}(i)}, a_{\sigma^{-1}(i+1)})a_{\sigma^{-1}(j)}}\\
    & \cdot \state{a_{\sigma^{-1}(j+1)}\cdots \mu(m) \cdots \cdots \mu(n) a_{\sigma^{-1}(2k+n)}  }\bigg{)} 
    \end{aligned}\label{difference-5-1}\\
    & \begin{aligned}
    \quad - \sum_{\substack{a_{2l-1}=a_{2l}=1\\l=1,...,k}}^n \sum_{\mu\in \text{Sym}_n} (-1)^\mu (\nabla^{2k+n-2}f)&\bigg{(}\state{ a_{\sigma^{-1}(1)} \cdots \mu(1) \cdots \cdots \mu(m-1) a_{\sigma^{-1}(i_{m-1}+1)}\cdots a_{\sigma^{-1}(i-1)}}\\ & \cdot
    \state{a_{\sigma^{-1}(i+2)}\cdots a_{\sigma^{-1}(i_m-1)}}\state{R(a_{\sigma^{-1}(i)}, a_{\sigma^{-1}(i+1)})\mu(m)}\\
    & \cdot \state{a_{\sigma^{-1}(j+1)}\cdots \mu(m) \cdots \cdots \mu(n) a_{\sigma^{-1}(2k+n)}  }\bigg{)} 
    \end{aligned}\label{difference-5-2}\\
    & \begin{aligned}
    \quad - \sum_{p=m+1}^{n-1}\sum_{\substack{a_{2l-1}=a_{2l}=1\\l=1,...,k}}^n \sum_{\mu\in \text{Sym}_n}\sum_{j=i_p + 1}^{i_{p+1}-1} (-1)^\mu (\nabla^{2k+n-2}f)&\bigg{(}\state{ a_{\sigma^{-1}(1)} \cdots \mu(1) \cdots \mu(m-1) a_{\sigma^{-1}(i_{m-1}+1)}\cdots a_{\sigma^{-1}(i-1)}}\\ & \cdot
    \state{a_{\sigma^{-1}(i+2)} \cdots \mu(m) \cdots \cdots \mu(p-1) \cdots a_{\sigma^{-1}(j-1)}}\\
    &\cdot \state{R(a_{\sigma^{-1}(i)}, a_{\sigma^{-1}(i+1)})a_{\sigma^{-1}(j)}}\\
    & \cdot \state{a_{\sigma^{-1}(j+1)}\cdots \mu(p+1) \cdots \cdots \mu(n) a_{\sigma^{-1}(2k+n)}  }\bigg{)}
    \end{aligned}\label{difference-5-3}\\
    & \begin{aligned}
    \quad - \sum_{p=m+1}^{n}\sum_{\substack{a_{2l-1}=a_{2l}=1\\l=1,...,k}}^n \sum_{\mu\in \text{Sym}_n} (-1)^\mu (\nabla^{2k+n-2}f)&\bigg{(}\state{ a_{\sigma^{-1}(1)} \cdots \mu(1) \cdots \cdots \mu(m-1) a_{\sigma^{-1}(i_{m-1}+1)}\cdots a_{\sigma^{-1}(i-1)}}\\ & \cdot
    \state{a_{\sigma^{-1}(i+2)}\cdots \mu(m) \cdots \cdots \mu(p-1) a_{\sigma^{-1}(i_p-1)}}\\& \cdot\state{R(a_{\sigma^{-1}(i)}, a_{\sigma^{-1}(i+1)})\mu(p)}\\
    & \cdot \state{a_{\sigma^{-1}(j+1)}\cdots \mu(p+1) \cdots \cdots \mu(n) a_{\sigma^{-1}(2k+n)}  }\bigg{)}
    \end{aligned}\label{difference-5-4}\\
    & \begin{aligned}
    \quad - \sum_{\substack{a_{2l-1}=a_{2l}=1\\l=1,...,k}}^n \sum_{\mu\in \text{Sym}_n}\sum_{j=i_n + 1}^{2k+n} (-1)^\mu (\nabla^{2k+n-2}f)&\bigg{(}\state{ a_{\sigma^{-1}(1)} \cdots \mu(1) \cdots \cdots \mu(m-1) a_{\sigma^{-1}(i_{m-1}+1)}\cdots a_{\sigma^{-1}(i-1)}}\\ & \cdot
    \state{a_{\sigma^{-1}(i+2)} \cdots \mu(m) \cdots \cdots \mu(n) \cdots a_{\sigma^{-1}(j-1)}}\\
    &\cdot \state{R(a_{\sigma^{-1}(i)}, a_{\sigma^{-1}(i+1)})a_{\sigma^{-1}(j)}} \state{a_{\sigma^{-1}(j+1)}\cdots  a_{\sigma^{-1}(2k+n)}  }\bigg{)} 
    \end{aligned}\label{difference-5-5}
\end{align}
The sums (\ref{difference-5-1}), (\ref{difference-5-3}) and (\ref{difference-5-5}) can be handled as in Proposition \ref{4-3} or Case 1 here. Also note that the sum (\ref{difference-5-2}) is simply a special case for (\ref{difference-5-4}) when $p=m$. Thus we will handle the summand of (\ref{difference-5-4}) for each $p=m, ..., n$. 

\noindent\textbf{Case 5.1. } $\{\sigma^{-1}(i), \sigma^{-1}(i+1)\} = \{2s-1, 2s\}$ for some $s=1,..., k$. Similar as Case 2.1 in Proposition \ref{4-3}, we have zero. 

\noindent\textbf{Case 5.2. }
$\sigma^{-1}(i)=2s, \sigma^{-1}(i+1)=2t$ for some $s, t=1,...,k, s\neq t$. Without loss of generality, we assume $2s-1$ and $2t-1$ both appear in $\{\sigma^{-1}(1), ..., \sigma^{-1}(i-1)\}$. Then the sum of (\ref{difference-5-2}) and (\ref{difference-5-4}) can be rewritten as
\begin{align*}
    & \begin{aligned}
    \quad - \sum_{p=m}^{n}\sum_{\substack{a_{2l-1}=a_{2l}=1\\l=1,...,k\\
    l\neq s, l\neq t}}^n \sum_{\mu\in \text{Sym}_n}\sum_{u,v=1}^n (-1)^\mu (\nabla^{2k+n-2}f)&\bigg{(}\state{ a_{\sigma^{-1}(1)} \cdots \mu(1) \cdots u  \cdots v \cdots \mu(m-1) \cdots a_{\sigma^{-1}(i-1)}}\\ & \cdot
    \state{a_{\sigma^{-1}(i+2)}\cdots \mu(m) \cdots \cdots \mu(p-1) \cdots a_{\sigma^{-1}(i_p-1)}}\\& \cdot\state{R(u,v)\mu(p)}  \state{a_{\sigma^{-1}(j+1)}\cdots \mu(p+1) \cdots \mu(n) \cdots a_{\sigma^{-1}(2k+n)}  }\bigg{)}
    \end{aligned}
\end{align*}
Note that 
\begin{align*}
    \state{R(u,v)\mu(p)} = \left\{\begin{array}{ll} 
    -K\state{v} & \text{ if } u = \mu(p), v\neq \mu(p), \\
    K\state{u}& \text{ if }u\neq \mu(p), v=\mu(p),\\
    0 & \text{ otherwise.}
    \end{array}
    \right.
\end{align*}
Thus, (\ref{difference-5-4}) is equal to
\begin{align*}
    & \begin{aligned}
    \quad K \sum_{p=m}^{n}\sum_{\substack{a_{2l-1}=a_{2l}=1\\l=1,...,k\\
    l\neq s, l\neq t}}^n \sum_{\mu\in \text{Sym}_n}\sum_{\substack{v=1\\v\neq \mu(p)}}^n (-1)^\mu (\nabla^{2k+n-2}f)&\bigg{(}\state{ a_{\sigma^{-1}(1)} \cdots \mu(1) \cdots \cdots \mu(p) \cdots v \cdots \mu(m-1) \cdots a_{\sigma^{-1}(i-1)}}\\ & \cdot
    \state{a_{\sigma^{-1}(i+2)}\cdots \mu(m) \cdots \cdots \mu(p-1) \cdots a_{\sigma^{-1}(i_p-1)}}\\& \cdot\state{v}  \state{a_{\sigma^{-1}(j+1)}\cdots \mu(p+1) \cdots \cdots \mu(n) \cdots a_{\sigma^{-1}(2k+n)}  }\bigg{)}
    \end{aligned}\\
    & \begin{aligned}
    \quad -K \sum_{p=m}^{n}\sum_{\substack{a_{2l-1}=a_{2l}=1\\l=1,...,k\\
    l\neq s, l\neq t}}^n \sum_{\mu\in \text{Sym}_n}\sum_{\substack{u=1\\u\neq \mu(p)}}^n (-1)^\mu (\nabla^{2k+n-2}f)&\bigg{(}\state{ a_{\sigma^{-1}(1)} \cdots \mu(1) \cdots u  \cdots \mu(p) \cdots \mu(m-1) \cdots a_{\sigma^{-1}(i-1)}}\\ & \cdot
    \state{a_{\sigma^{-1}(i+2)}\cdots \mu(m) \cdots \cdots \mu(p-1) \cdots a_{\sigma^{-1}(i_p-1)}}\\& \cdot\state{u}  \state{a_{\sigma^{-1}(j+1)}\cdots \mu(p+1) \cdots \cdots \mu(n) \cdots a_{\sigma^{-1}(2k+n)}  }\bigg{)}
    \end{aligned}
\end{align*}
Notice that the requirements $v\neq \mu(p)$ in the first sum and $u\neq \mu(p)$ in the second sum can be eliminated, as the extra terms introduced adds up to zero. Then for each $p=m, ..., n$, the summands of both the first and the second sums are actions of a permutation of $\theta^{\otimes (k-1)}\otimes \Lambda$ on $f$. By the induction hypothesis, they are all zero.


\noindent\textbf{Case 6. } $i=i_{m-1}<i_{m}-1$ for $m=2, ..., n-1$. In this case, 
\begin{align}
    & \quad (\nabla^{2k+n}f)(\sigma(\theta^{\otimes k}\otimes \Lambda) - (\nabla^{2k+n}f)((i,i+1)\sigma(\theta^{\otimes k}\otimes \Lambda) \nonumber\\
    & \begin{aligned}
    = - \sum_{\substack{a_{2l-1}=a_{2l}=1\\l=1,...,k}}^n \sum_{\mu\in \text{Sym}_n}\sum_{j=i+2}^{i_{m}-1} (-1)^\mu (\nabla^{2k+n-2}f)&\bigg{(}\state{ a_{\sigma^{-1}(1)} \cdots \mu(1) \cdots \cdots \mu(m-2) \cdots a_{\sigma^{-1}(j-1)}}\\ & \cdot
    \state{R(\mu(m-1), a_{\sigma^{-1}(i_{m-1}+1)})a_{\sigma^{-1}(j)}}\\
    & \cdot \state{a_{\sigma^{-1}(j+1)}\cdots \mu(m) \cdots \cdots \mu(n) a_{\sigma^{-1}(2k+n)}  }\bigg{)} 
    \end{aligned}\label{difference-6-1}\\
    & \begin{aligned}
    \quad - \sum_{\substack{a_{2l-1}=a_{2l}=1\\l=1,...,k}}^n \sum_{\mu\in \text{Sym}_n} (-1)^\mu (\nabla^{2k+n-2}f)&\bigg{(}\state{ a_{\sigma^{-1}(1)} \cdots \mu(1) \cdots \cdots \mu(m-2) \cdots a_{\sigma^{-1}(i_{m-1}-1)}}\\ & \cdot
    \state{R(\mu(m-1), a_{\sigma^{-1}(i_{m-1}+1)})\mu(m)}\\
    & \cdot \state{a_{\sigma^{-1}(i_m+1)}\cdots \mu(m+1) \cdots \cdots \mu(n) a_{\sigma^{-1}(2k+n)}  }\bigg{)} 
    \end{aligned}\label{difference-6-2}\\
    & \begin{aligned}
    \quad - \sum_{p=m}^{n-1}\sum_{\substack{a_{2l-1}=a_{2l}=1\\l=1,...,k}}^n \sum_{\mu\in \text{Sym}_n}\sum_{j=i_p+1}^{i_{p+1}-1} (-1)^\mu (\nabla^{2k+n-2}f)&\bigg{(}\state{ a_{\sigma^{-1}(1)} \cdots \mu(1) \cdots  \mu(m-2) \cdots \mu(m) \cdots \mu(p-1)}\\
    & \cdot \state{a_{\sigma^{-1}(i_{p-1}+1)}\cdots a_{\sigma^{-1}(j-1)}} \\ 
    & \cdot \state{R(\mu(m-1), a_{\sigma^{-1}(i_{m-1}+1)})a_{\sigma^{-1}(j)}}\\
    & \cdot \state{a_{\sigma^{-1}(j+1)}\cdots \mu(p+1) \cdots \cdots \mu(n) a_{\sigma^{-1}(2k+n)}  }\bigg{)} 
    \end{aligned}\label{difference-6-3}\\
    & \begin{aligned}
    \quad - \sum_{p=m+1}^{n}\sum_{\substack{a_{2l-1}=a_{2l}=1\\l=1,...,k}}^n \sum_{\mu\in \text{Sym}_n} (-1)^\mu (\nabla^{2k+n-2}f)&\bigg{(}\state{ a_{\sigma^{-1}(1)} \cdots \mu(1) \cdots \mu(m-2)\cdots \mu(m) \cdots \mu(p-1)}\\& \cdot \state{a_{\sigma^{-1}(i_p+1)} \cdots a_{\sigma^{-1}(i_{p}-1)}}
    \state{R(\mu(m-1), a_{\sigma^{-1}(i_{m-1}+1)})\mu(p)}\\
    & \cdot \state{a_{\sigma^{-1}(i_p+1)}\cdots \mu(p+1) \cdots \cdots \mu(n) a_{\sigma^{-1}(2k+n)}  }\bigg{)} 
    \end{aligned}\label{difference-6-4}\\
    & \begin{aligned}
    \quad - \sum_{\substack{a_{2l-1}=a_{2l}=1\\l=1,...,k}}^n \sum_{\mu\in \text{Sym}_n}\sum_{j=i_n+1}^{2k+n} (-1)^\mu (\nabla^{2k+n-2}f)&\bigg{(}\state{ a_{\sigma^{-1}(1)} \cdots \mu(1) \cdots \mu(m-2) \cdots \mu(m) \cdots \mu(n) \cdots a_{\sigma^{-1}(j-1)}}\\ & \cdot
    \state{R(\mu(m-1), a_{\sigma^{-1}(i_{m-1}+1)})a_{\sigma^{-1}(j)}}\\
    & \cdot \state{a_{\sigma^{-1}(j+1)}\cdots  a_{\sigma^{-1}(2k+n)}  }\bigg{)} 
    \end{aligned}\label{difference-6-5}
\end{align}
The sums (\ref{difference-6-1}), (\ref{difference-6-3}) and (\ref{difference-6-5}) can be handled as in Case 2. The sum (\ref{difference-6-2}) is simply a special case for (\ref{difference-6-4}) when $p=m$. We will handle the summand of (\ref{difference-6-4}) for each $p=m, ..., n$. 

Without loss of generality, let $\sigma^{-1}(i_{m-1}) = 2s$ and assume $2s-1$ appears in $\{\sigma^{-1}(1), ..., \sigma^{-1}(i)\}$. Then the sum of (\ref{difference-6-2}) and (\ref{difference-6-4}) can be rewritten as

\begin{align*}
    & \begin{aligned}    \quad - \sum_{p=m}^{n}\sum_{\substack{a_{2l-1}=a_{2l}=1\\l=1,...,k}}^n \sum_{\mu\in \text{Sym}_n}\sum_{u=1}^n (-1)^\mu (\nabla^{2k+n-2}f)&\bigg{(}\state{ a_{\sigma^{-1}(1)} \cdots \mu(1) \cdots u \cdots \mu(m-2) \cdots \mu(m) \cdots \mu(p-1)}\\ & \cdot \state{a_{\sigma^{-1}(i_{p-1}+1)} \cdots a_{\sigma^{-1}(i_{p}-1)}}
    \state{R(\mu(m-1), u)\mu(p)}\\
    & \cdot \state{a_{\sigma^{-1}(i_p+1)}\cdots \mu(p+1) \cdots \cdots \mu(n) a_{\sigma^{-1}(2k+n)}  }\bigg{)} 
    \end{aligned}\\
    & \begin{aligned}    = -K \sum_{p=m}^{n}\sum_{\substack{a_{2l-1}=a_{2l}=1\\l=1,...,k}}^n \sum_{\mu\in \text{Sym}_n} (-1)^\mu (\nabla^{2k+n-2}f)&\bigg{(}\state{ a_{\sigma^{-1}(1)} \cdots \mu(1)  \cdots \mu(p) \cdots \mu(m-2) \cdots \mu(m)\cdots \mu(p-1)}\\ & \cdot \state{a_{\sigma^{-1}(i_{p-1}+1)} \cdots a_{\sigma^{-1}(i_{p}-1)}}\state{\mu(m-1)}\\ & \cdot
    \state{a_{\sigma^{-1}(i_p+1)}\cdots \mu(p+1) \cdots \cdots \mu(n) a_{\sigma^{-1}(2k+n)}  }\bigg{)} 
    \end{aligned}
\end{align*}
This is again zero. 

\noindent\textbf{Case 7. }$i=i_m - 1 > i_{m-1}$ for $m=1, 2, ..., n-1$. In this case, 
\begin{align*}
    & \quad (\nabla^{2k+n}f)(\sigma(\theta^{\otimes k}\otimes \Lambda) - (\nabla^{2k+n}f)((i,i+1)\sigma(\theta^{\otimes k}\otimes \Lambda) \nonumber\\
    & \begin{aligned}
    = - \sum_{\substack{a_{2l-1}=a_{2l}=1\\l=1,...,k}}^n \sum_{\mu\in \text{Sym}_n}\sum_{j=i+2}^{i_{m+1}-1} (-1)^\mu (\nabla^{2k+n-2}f)&\bigg{(}\state{ a_{\sigma^{-1}(1)} \cdots \mu(1) \cdots \cdots \mu(m-1) \cdots a_{\sigma^{-1}(i_{m}-2)}}\\ & \cdot
    \state{R(a_{\sigma^{-1}(i_{m}-1)}, \mu(m))a_{\sigma^{-1}(j)}}\\
    & \cdot \state{a_{\sigma^{-1}(j+1)}\cdots \mu(m+1) \cdots \cdots \mu(n) a_{\sigma^{-1}(2k+n)}  }\bigg{)} 
    \end{aligned}\\
    & \begin{aligned}
    \quad - \sum_{\substack{a_{2l-1}=a_{2l}=1\\l=1,...,k}}^n \sum_{\mu\in \text{Sym}_n} (-1)^\mu (\nabla^{2k+n-2}f)&\bigg{(}\state{ a_{\sigma^{-1}(1)} \cdots \mu(1) \cdots \cdots \mu(m-1) \cdots a_{\sigma^{-1}(i_{m}-2)}}\\ & \cdot
    \state{R(a_{\sigma^{-1}(i_{m}-1)}, \mu(m))\mu(m+1)}\\
    & \cdot \state{a_{\sigma^{-1}(j+1)}\cdots \mu(m+1) \cdots \cdots \mu(n) a_{\sigma^{-1}(2k+n)}  }\bigg{)} 
    \end{aligned}\\
    & \begin{aligned}
    \quad - \sum_{p=m+1}^{n-1}\sum_{\substack{a_{2l-1}=a_{2l}=1\\l=1,...,k}}^n \sum_{\mu\in \text{Sym}_n}\sum_{j=i_p+1}^{i_{p+1}-1} (-1)^\mu (\nabla^{2k+n-2}f)&\bigg{(}\state{ a_{\sigma^{-1}(1)} \cdots \mu(1) \cdots \mu(m-1) \cdots \mu(m+1) \cdots \mu(p-1)}\\
    & \cdot \state{a_{\sigma^{-1}(i_{p-1}+1)} \cdots a_{\sigma^{-1}(j-1)}}
    \state{R( a_{\sigma^{-1}(i_{m}-1)},\mu(m))a_{\sigma^{-1}(j)}}\\
    & \cdot \state{a_{\sigma^{-1}(j+1)}\cdots \mu(p+1) \cdots \cdots \mu(n) a_{\sigma^{-1}(2k+n)}  }\bigg{)} 
    \end{aligned}\\
    & \begin{aligned}
    \quad - \sum_{p=m+1}^{n}\sum_{\substack{a_{2l-1}=a_{2l}=1\\l=1,...,k}}^n \sum_{\mu\in \text{Sym}_n} (-1)^\mu (\nabla^{2k+n-2}f)&\bigg{(}\state{ a_{\sigma^{-1}(1)} \cdots \mu(1) \cdots \mu(m-1) \cdots \mu(m+1) \cdots \mu(p-1)}\\&\cdot \state{ a_{\sigma^{-1}(i_{p-1}+1)} \cdots a_{\sigma^{-1}(i_{p}-1)}}
    \state{R( a_{\sigma^{-1}(i_{m}-1)}, \mu(m))\mu(p)}\\
    & \cdot \state{a_{\sigma^{-1}(i_p+1)}\cdots \mu(p+1) \cdots \cdots \mu(n) a_{\sigma^{-1}(2k+n)}  }\bigg{)} 
    \end{aligned}\\
    & \begin{aligned}
    \quad - \sum_{\substack{a_{2l-1}=a_{2l}=1\\l=1,...,k}}^n \sum_{\mu\in \text{Sym}_n}\sum_{j=i_n+1}^{2k+n} (-1)^\mu (\nabla^{2k+n-2}f)&\bigg{(}\state{ a_{\sigma^{-1}(1)} \cdots \mu(1) \cdots \mu(m-1) \cdots \mu(m+1) \cdots \mu(n)}\\
    & \cdot \state{a_{\sigma^{-1}(i_n-1)+1}\cdots  a_{\sigma^{-1}(j-1)}}    \state{R( a_{\sigma^{-1}(i_{m}-1)},\mu(m))a_{\sigma^{-1}(j)}}\\
    & \cdot \state{a_{\sigma^{-1}(j+1)}\cdots  a_{\sigma^{-1}(2k+n)}  }\bigg{)} 
    \end{aligned}
\end{align*}
Using $R(a_{\sigma^{-1}(i_m-1)}, \mu(m))=-R(\mu(m), a_{\sigma^{-1}(i_m-1)})$, this case can be handled similarly as in Case 6. We will not repeat the discussion here. 

\noindent\textbf{Case 8. } $ i = i_{m-1} = i_m - 1$ for $m=2, ..., n-1$. In this case, 
\begin{align}
    & \quad (\nabla^{2k+n}f)(\sigma(\theta^{\otimes k}\otimes \Lambda) - (\nabla^{2k+n}f)((i,i+1)\sigma(\theta^{\otimes k}\otimes \Lambda) \nonumber\\
    & \begin{aligned}
    = - \sum_{\substack{a_{2l-1}=a_{2l}=1\\l=1,...,k}}^n \sum_{\mu\in \text{Sym}_n}\sum_{j=i+2}^{i_{m+1}-1} (-1)^\mu (\nabla^{2k+n-2}f)&\bigg{(}\state{ a_{\sigma^{-1}(1)} \cdots \mu(1) \cdots \cdots \mu(m-2) \cdots a_{\sigma^{-1}(i_{m-1}-1)}}\\ & \cdot
    \state{R(\mu(m-1), \mu(m))a_{\sigma^{-1}(j)}}\\
    & \cdot \state{a_{\sigma^{-1}(j+1)}\cdots \mu(m+1) \cdots \cdots \mu(n) a_{\sigma^{-1}(2k+n)}  }\bigg{)} 
    \end{aligned}\label{difference-9-1}\\
    & \begin{aligned}
    \quad - \sum_{\substack{a_{2l-1}=a_{2l}=1\\l=1,...,k}}^n \sum_{\mu\in \text{Sym}_n} (-1)^\mu (\nabla^{2k+n-2}f)&\bigg{(}\state{ a_{\sigma^{-1}(1)} \cdots \mu(1) \cdots \cdots \mu(m-2) \cdots a_{\sigma^{-1}(i_{m}-2)}}\\ & \cdot
    \state{R(\mu(m-1), \mu(m))\mu(m+1)}\\
    & \cdot \state{a_{\sigma^{-1}(j+1)}\cdots \mu(m+1) \cdots \cdots \mu(n) a_{\sigma^{-1}(2k+n)}  }\bigg{)} 
    \end{aligned}\label{difference-9-2}\\
    & \begin{aligned}
    \quad - \sum_{p=m+1}^{n-1}\sum_{\substack{a_{2l-1}=a_{2l}=1\\l=1,...,k}}^n \sum_{\mu\in \text{Sym}_n}\sum_{j=i_p+1}^{i_{p+1}-1} (-1)^\mu (\nabla^{2k+n-2}f)&\bigg{(}\state{ a_{\sigma^{-1}(1)} \cdots \mu(1) \cdots \mu(m-2) \cdots \mu(m+1) \cdots \mu(p-1)} \\ & \cdot \state{a_{\sigma^{-1}(i_{p-1}+1)} \cdots a_{\sigma^{-1}(j-1)}}
    \state{R(\mu(m-1), \mu(m))a_{\sigma^{-1}(j)}}\\
    & \cdot \state{a_{\sigma^{-1}(j+1)}\cdots \mu(p+1) \cdots \cdots \mu(n) a_{\sigma^{-1}(2k+n)}  }\bigg{)} 
    \end{aligned}\label{difference-9-3}\\
    & \begin{aligned}
    \quad - \sum_{p=m+2}^{n}\sum_{\substack{a_{2l-1}=a_{2l}=1\\l=1,...,k}}^n \sum_{\mu\in \text{Sym}_n} (-1)^\mu (\nabla^{2k+n-2}f)&\bigg{(}\state{ a_{\sigma^{-1}(1)} \cdots \mu(1) \cdots \mu(m-2) \cdots \mu(m+1) \cdots \mu(p-1)}\\ & \cdot \state{a_{\sigma^{-1}(i_{p-1}+1)} \cdots a_{\sigma^{-1}(i_{p}-1)}}
    \state{R(\mu(m-1), \mu(m))\mu(p)}\\
    & \cdot \state{a_{\sigma^{-1}(i_p+1)}\cdots \mu(p+1) \cdots \cdots \mu(n) a_{\sigma^{-1}(2k+n)}  }\bigg{)} 
    \end{aligned}\label{difference-9-4}\\
    & \begin{aligned}
    \quad - \sum_{\substack{a_{2l-1}=a_{2l}=1\\l=1,...,k}}^n \sum_{\mu\in \text{Sym}_n}\sum_{j=i_n+1}^{2k+n} (-1)^\mu (\nabla^{2k+n-2}f)&\bigg{(}\state{ a_{\sigma^{-1}(1)} \cdots \mu(1) \cdots \mu(m-2) \cdots \mu(m+1) \cdots \mu(n)}\\
    & \cdot \state{a_{\sigma^{-1}(i_n+1)}\cdots a_{\sigma^{-1}(j-1)}}    \state{R(\mu(m-1), \mu(m))a_{\sigma^{-1}(j)}}\\
    & \cdot \state{a_{\sigma^{-1}(j+1)}\cdots  a_{\sigma^{-1}(2k+n)}  }\bigg{)} 
    \end{aligned}\label{difference-9-5}
\end{align}
The sums (\ref{difference-9-1}), (\ref{difference-9-3}) and (\ref{difference-9-5}) can all be handled similarly as in Case 4, while (\ref{difference-9-2}) and (\ref{difference-9-4}) are all zero, as $\mu(m-1), \mu(m)$ and $\mu(p)$ are all distinct for $p=m+1, ..., n$. 

\noindent\textit{Proof of the Theorem.} Write $\sigma$ as a product of transpositions and proceed with induction similar to Theorem \ref{4-2}. We will not repeat the arguments here.
\end{proof}

In conclusion, we have the following theorem. 

\begin{thm}\label{ParallelScalar}
Let $f$ be an eigenfunction of the Laplace-Beltrami operator. Then as a vector space, $\Pi(T(TM))f = \C f$.  
\end{thm}

\begin{rema}
The same argument also shows that 
$$\Pi(T(TM))f = \R f$$
for eigenfunctions $f$ with real eigenvalues. 
\end{rema}

\section{Scalars of $O(n,\R)$-invariant tensors}

In this section we study the scalar given by the action of $O(n, \R)$-invariant tensors. The proof of Theorem \ref{4-2} showed that the scalar is indeed a polynomial in $\lambda$. We will study the polynomial by an alternative combinatorial characterization. 

\subsection{Words, their graph and parallel tensors}

\begin{defn}
A word is a sequence of letters where each letter appears exactly twice. The number of letters appearing in a word is called the length of the word. It is clear that the length of a word must be even. For each word of length $2k$, we associate a graph with $2k$ vertices. An edge exists between two vertices if and only if the letters associated with the vertices are the same. 
\end{defn}

For example, the word $``aabccb"$ is associated with the following graph\\

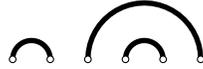
\begin{center}
\begin{tikzpicture}[line cap=round,line join=round,>=triangle 45,x=1cm,y=1cm, scale = 0.5]
\draw [shift={(1.5,0)},line width=2pt]  plot[domain=0:3.141592653589793,variable=\t]({1*0.5*cos(\t r)+0*0.5*sin(\t r)},{0*0.5*cos(\t r)+1*0.5*sin(\t r)});
\draw [shift={(4.5,0)},line width=2pt]  plot[domain=0:3.141592653589793,variable=\t]({1*1.5*cos(\t r)+0*1.5*sin(\t r)},{0*1.5*cos(\t r)+1*1.5*sin(\t r)});
\draw [shift={(4.5,0)},line width=2pt]  plot[domain=0:3.141592653589793,variable=\t]({1*0.5*cos(\t r)+0*0.5*sin(\t r)},{0*0.5*cos(\t r)+1*0.5*sin(\t r)});
\begin{scriptsize}
\draw [fill=ffffff] (1,0) circle (2.5pt);
\draw [fill=ffffff] (6,0) circle (2.5pt);
\draw [fill=ffffff] (2,0) circle (2.5pt);
\draw [fill=ffffff] (5,0) circle (2.5pt);
\draw [fill=ffffff] (3,0) circle (2.5pt);
\draw [fill=ffffff] (4,0) circle (2.5pt);
\end{scriptsize}
\end{tikzpicture}
\end{center}
Here we arrange the vertices all in a row and with edges lying above. 

We now associate each word with a parallel tensor. First, the word $``a_1a_2 \cdots a_{2k-1}a_{2k}"$ with $a_1=a_2, ..., a_{2k-1}=a_{2k}$ is associated with
$$\theta^k = \sum_{a_1=a_2=1}^n \cdots \sum_{\substack{a_{2k-1}=a_{2k}=1}}^n \state{a_1a_2\cdots a_{2k-1}a_{2k}}.$$
For a general word $``a_1a_2\cdots a_{2k-1}a_{2k}"$ with $a_{i_1}=a_{i_2}, ...., a_{i_{2k-1}}=a_{i_{2k}}$, let $\sigma\in \text{Sym}_{2k}$ be the inverse of the permutation $$\begin{pmatrix}
1 & 2 & \cdots & 2k-1 & 2k\\
i_1 & i_2 & \cdots & i_{2k-1} & i_{2k}
\end{pmatrix}.$$
The parallel tensor associated to this word is simply 
$$\sigma(\theta^k) = \sum_{a_1=a_2=1}^n \cdots \sum_{\substack{a_{2k-1}=a_{2k}=1}}^n \state{a_{\sigma^{-1}(1)}a_{\sigma^{-1}(2)}\cdots a_{\sigma^{-1}(2k-1)}a_{\sigma^{-1}(2k)}}.$$
For example the parallel tensors associated to the word $``aabccb"$ is simply 
$$\sum_{a_1=a_2=1}^n \sum_{a_3=a_4=1}^n \sum_{a_{5}=a_{6}=1}^n \state{a_1a_2 a_3 a_5a_6 a_4}$$
Moreover, if we arrange the vertices of the graph of the word in two rows, with first $k$ vertices in the first and last $k$ vertices in the second, then we get the $(k, k)$-Brauer diagrams for parallel tensors. Here is the example for $``aabccb"$: \\

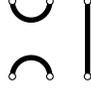
\begin{center}
    \begin{tikzpicture}[line cap=round,line join=round,>=triangle 45,x=1cm,y=1cm, scale = 0.5]
\draw [shift={(1.5,0)},line width=2pt]  plot[domain=3.141592653589793:2*3.141592653589793,variable=\t]({1*0.5*cos(\t r)+0*0.5*sin(\t r)},{0*0.5*cos(\t r)+1*0.5*sin(\t r)});
\draw [shift={(1.5,-2)},line width=2pt]  plot[domain=0:3.141592653589793,variable=\t]({1*0.5*cos(\t r)+0*0.5*sin(\t r)},{0*0.5*cos(\t r)+1*0.5*sin(\t r)});
\draw [line width=2pt] (3,0)-- (3,-2);
\begin{scriptsize}
\draw [fill=ffffff] (1,0) circle (2.5pt);
\draw [fill=ffffff] (3,-2) circle (2.5pt);
\draw [fill=ffffff] (2,-2) circle (2.5pt);
\draw [fill=ffffff] (2,0) circle (2.5pt);
\draw [fill=ffffff] (3,0) circle (2.5pt);
\draw [fill=ffffff] (1,-2) circle (2.5pt);
\end{scriptsize}
\end{tikzpicture}
\end{center}
It is well known that for each $k\in \Z_+$, the number of such words with length $2k$ is 
$$(2k-1)!! = 1 \cdot 3 \cdot 5 \cdot \cdots \cdot (2k-3) \cdot (2k-1). $$
It is also shown in \cite{LZ-Annals} that parallel tensors associated with $(k,k)$-Brauer diagrams span the space of parallel tensors. The spanning set is linearly independent if $k\leq n$. When $k\geq n+1$, a linear relation among these spanning tensors can also be described. Thus the parallel tensors can be studied by using the representation by words or by the graph. 

\subsection{Action of parallel tensors as a polynomial of $\theta$} Let $``a_1a_2 \cdots a_{2k-1}a_{2k}"$ be a word. As shown in Theorem \ref{4-2}, the associated parallel tensor acts on $f$ as a scalar. The same argument also shows that the scalar is a polynomial in the eigenvalue of $\lambda$. 

\begin{nota}
Let $``a_1 a_2 \cdots a_{2k-1}a_{2k}"$ be a word. We abuse the notation $\left|a_1 a_2 \cdots a_{2k-1}a_{2k}\right\rangle$ for the scalar of the action by the associated parallel tensor. If $``b_1b_2 \cdots b_{2l-1}b_{2l}"$ is another word, then it follows from Formula (\ref{Assoc}) that 
\begin{align} \label{Brauer-tensor}
    |a_1 a_2 \cdots a_{2k-1}a_{2k}b_1b_2 \cdots b_{2l-1}b_{2l}\rangle = |a_1 a_2 \cdots a_{2k-1}a_{2k}\rangle \cdot |b_1b_2 \cdots b_{2l-1}b_{2l}\rangle
\end{align}
We will also use the bracket notation to denote the ``commutator'', namely, 
\begin{align*}
    |a_1 \cdots a_{i-1}[a_i a_{i+1}]a_{i+2} \cdots a_{2k}\rangle & := 
    |a_1 \cdots a_{i-1}a_i a_{i+1} a_{i+2}\cdots a_{2k}\rangle - |a_1 \cdots a_{i-1} a_{i+1} a_i a_{i+2}\cdots a_{2k}\rangle
\end{align*}
For convenience, we will use $\theta$ as the variable of the polynomial appearing in $\left|a_1 a_2 \cdots a_{2k-1}a_{2k}\right\rangle$ instead of $\lambda$. 
\end{nota}

We now describe the algorithm of obtaining this polynomial in terms of recursive operations on words. 
\begin{prop}
For every word $``a_1a_2 \cdots a_{2k-1}a_{2k}"$, the polynomial $|a_1a_2 \cdots a_{2k-1}a_{2k}\rangle$ can be obtained via the following recursion: 
\begin{align}
    & |a_1a_1a_2a_2 \cdots a_ka_k\rangle = \theta^k\\
    & |a_1 \cdots a_{2k-2} a_{2k-1}a_{2k}\rangle = |a_1 \cdots a_{2k-2} a_{2k}a_{2k-1}\rangle \\
    & |a_1 \cdots a_{i-1} [a_{i}a_{i+1}]a_{i+2} \cdots a_{2k}\rangle = -\sum_{j=i+2}^{2k} |a_1\cdots a_{i-1} a_{i+2} \cdots a_{j-1}, R(a_{i},a_{i+1})a_j, a_{j+1}\cdots a_{2k}\rangle \label{recur-2}
\end{align}
Here for each $j= i+2, ..., 2k$ 
\begin{align*}
    & |a_1\cdots a_{i-1} a_{i+2} \cdots a_{j-1}, R(a_{i},a_{i+1})a_j, a_{j+1}\cdots a_{2k}\rangle \\
    & = \left\{\begin{aligned}
    & - K(n-1)|a_1\cdots a_{i-1} a_{i+2} \cdots a_{j-1}, a_{i+1}, a_{j+1}\cdots a_{2k}\rangle && \text{ if }a_j = a_{i},\\
    & K(n-1)|a_1\cdots a_{i-1} a_{i+2} \cdots a_{j-1}, a_{i}, a_{j+1}\cdots a_{2k}\rangle && \text{ if }a_j = a_{i+1},\\
    & K \state{I_{a_ja_{i+1}} a_1 \cdots a_{i-1}a_{i+2} \cdots a_{j-1}, a_i, a_{j+1} \cdots a_{2k}} & \\
    & \quad - K \state{I_{a_ja_i} a_1 \cdots a_{i-1}a_{i+2} \cdots a_{j-1}, a_{i+1}, a_{j+1} \cdots a_{2k}} && \text{ otherwise.} 
    \end{aligned}\right. 
\end{align*}
where in the last formula, the operator $I_{a_j c}$ replaces the other occurrence of $a_j$ by $c$, for $c=a_i$ or $a_{i+1}$. 
\end{prop}

\begin{proof}
This follows from a straightforward computation. 
\end{proof}

\subsection{Examples}
Here we list the polynomials obtained for words of length 2, 4 and 6: 

\renewcommand\arraystretch{3.8}
\begin{center}
\begin{longtable}{c c c}
Words & Graph & Polynomial\\
\hline 
$|aa\rangle$ &\raisebox{-.3\height}{\begin{tikzpicture}[line cap=round,line join=round,>=triangle 45,x=1cm,y=1cm, scale=0.5]
\draw [shift={(2.5,0)},line width=2pt]  plot[domain=0:3.141592653589793,variable=\t]({1*0.5*cos(\t r)+0*0.5*sin(\t r)},{0*0.5*cos(\t r)+1*0.5*sin(\t r)});
\begin{scriptsize}
\draw [fill=ffffff] (2,0) circle (2.5pt);
\draw [fill=ffffff] (3,0) circle (2.5pt);
\end{scriptsize}
\end{tikzpicture}} & $\theta$
\\
\hline 
$|aabb\rangle$ &\raisebox{-.3\height}{\begin{tikzpicture}[line cap=round,line join=round,>=triangle 45,x=1cm,y=1cm, scale = 0.5]
\draw [shift={(1.5,0)},line width=2pt]  plot[domain=0:3.141592653589793,variable=\t]({1*0.5*cos(\t r)+0*0.5*sin(\t r)},{0*0.5*cos(\t r)+1*0.5*sin(\t r)});
\draw [shift={(3.5,0)},line width=2pt]  plot[domain=0:3.141592653589793,variable=\t]({1*0.5*cos(\t r)+0*0.5*sin(\t r)},{0*0.5*cos(\t r)+1*0.5*sin(\t r)});
\begin{scriptsize}
\draw [fill=ffffff] (1,0) circle (2.5pt);
\draw [fill=ffffff] (2,0) circle (2.5pt);
\draw [fill=ffffff] (3,0) circle (2.5pt);
\draw [fill=ffffff] (4,0) circle (2.5pt);
\end{scriptsize}
\end{tikzpicture}} & $\theta^2$
\\
\hline 
$|abab\rangle$ &\raisebox{-.3\height}{\begin{tikzpicture}[line cap=round,line join=round,>=triangle 45,x=1cm,y=1cm, scale = 0.5]
\draw [shift={(2,0)},line width=2pt]  plot[domain=0:3.141592653589793,variable=\t]({1*1*cos(\t r)+0*1*sin(\t r)},{0*1*cos(\t r)+1*1*sin(\t r)});
\draw [shift={(3,0)},line width=2pt]  plot[domain=0:3.141592653589793,variable=\t]({1*1*cos(\t r)+0*1*sin(\t r)},{0*1*cos(\t r)+1*1*sin(\t r)});
\begin{scriptsize}
\draw [fill=ffffff] (1,0) circle (2.5pt);
\draw [fill=ffffff] (2,0) circle (2.5pt);
\draw [fill=ffffff] (3,0) circle (2.5pt);
\draw [fill=ffffff] (4,0) circle (2.5pt);
\end{scriptsize}
\end{tikzpicture}} & $\theta(\theta+K(n-1))$
\\
\hline 
$|abba\rangle$ &\raisebox{-.3\height}{\begin{tikzpicture}[line cap=round,line join=round,>=triangle 45,x=1cm,y=1cm, scale = 0.5]
\draw [shift={(2.5,0)},line width=2pt]  plot[domain=0:3.141592653589793,variable=\t]({1*1.5*cos(\t r)+0*1.5*sin(\t r)},{0*1.5*cos(\t r)+1*1.5*sin(\t r)});
\draw [shift={(2.5,0)},line width=2pt]  plot[domain=0:3.141592653589793,variable=\t]({1*0.5*cos(\t r)+0*0.5*sin(\t r)},{0*0.5*cos(\t r)+1*0.5*sin(\t r)});
\begin{scriptsize}
\draw [fill=ffffff] (1,0) circle (2.5pt);
\draw [fill=ffffff] (2,0) circle (2.5pt);
\draw [fill=ffffff] (3,0) circle (2.5pt);
\draw [fill=ffffff] (4,0) circle (2.5pt);
\end{scriptsize}
\end{tikzpicture}} & $\theta(\theta+K(n-1))$
\\

\hline 
$|aabbcc\rangle$ &\raisebox{-.3\height}{\begin{tikzpicture}[line cap=round,line join=round,>=triangle 45,x=1cm,y=1cm, scale = 0.5]
\draw [shift={(3.5,0)},line width=2pt]  plot[domain=0:3.141592653589793,variable=\t]({1*0.5*cos(\t r)+0*0.5*sin(\t r)},{0*0.5*cos(\t r)+1*0.5*sin(\t r)});
\draw [shift={(1.5,0)},line width=2pt]  plot[domain=0:3.141592653589793,variable=\t]({1*0.5*cos(\t r)+0*0.5*sin(\t r)},{0*0.5*cos(\t r)+1*0.5*sin(\t r)});
\draw [shift={(5.5,0)},line width=2pt]  plot[domain=0:3.141592653589793,variable=\t]({1*0.5*cos(\t r)+0*0.5*sin(\t r)},{0*0.5*cos(\t r)+1*0.5*sin(\t r)});
\begin{scriptsize}
\draw [fill=ffffff] (1,0) circle (2.5pt);
\draw [fill=ffffff] (6,0) circle (2.5pt);
\draw [fill=ffffff] (2,0) circle (2.5pt);
\draw [fill=ffffff] (5,0) circle (2.5pt);
\draw [fill=ffffff] (3,0) circle (2.5pt);
\draw [fill=ffffff] (4,0) circle (2.5pt);
\end{scriptsize}
\end{tikzpicture}} & $\theta^3$
\\

\hline 
$|aabcbc\rangle$ &\raisebox{-.3\height}{\begin{tikzpicture}[line cap=round,line join=round,>=triangle 45,x=1cm,y=1cm, scale = 0.5]
\draw [shift={(1.5,0)},line width=2pt]  plot[domain=0:3.141592653589793,variable=\t]({1*0.5*cos(\t r)+0*0.5*sin(\t r)},{0*0.5*cos(\t r)+1*0.5*sin(\t r)});
\draw [shift={(4,0)},line width=2pt]  plot[domain=0:3.141592653589793,variable=\t]({1*1*cos(\t r)+0*1*sin(\t r)},{0*1*cos(\t r)+1*1*sin(\t r)});
\draw [shift={(5,0)},line width=2pt]  plot[domain=0:3.141592653589793,variable=\t]({1*1*cos(\t r)+0*1*sin(\t r)},{0*1*cos(\t r)+1*1*sin(\t r)});
\begin{scriptsize}
\draw [fill=ffffff] (1,0) circle (2.5pt);
\draw [fill=ffffff] (6,0) circle (2.5pt);
\draw [fill=ffffff] (2,0) circle (2.5pt);
\draw [fill=ffffff] (5,0) circle (2.5pt);
\draw [fill=ffffff] (3,0) circle (2.5pt);
\draw [fill=ffffff] (4,0) circle (2.5pt);
\end{scriptsize}
\end{tikzpicture}} & $\theta^2(\theta+K(n-1))$
\\

\hline 
$|aabccb\rangle$ &\raisebox{-.3\height}{\begin{tikzpicture}[line cap=round,line join=round,>=triangle 45,x=1cm,y=1cm, scale = 0.5]
\draw [shift={(1.5,0)},line width=2pt]  plot[domain=0:3.141592653589793,variable=\t]({1*0.5*cos(\t r)+0*0.5*sin(\t r)},{0*0.5*cos(\t r)+1*0.5*sin(\t r)});
\draw [shift={(4.5,0)},line width=2pt]  plot[domain=0:3.141592653589793,variable=\t]({1*1.5*cos(\t r)+0*1.5*sin(\t r)},{0*1.5*cos(\t r)+1*1.5*sin(\t r)});
\draw [shift={(4.5,0)},line width=2pt]  plot[domain=0:3.141592653589793,variable=\t]({1*0.5*cos(\t r)+0*0.5*sin(\t r)},{0*0.5*cos(\t r)+1*0.5*sin(\t r)});
\begin{scriptsize}
\draw [fill=ffffff] (1,0) circle (2.5pt);
\draw [fill=ffffff] (6,0) circle (2.5pt);
\draw [fill=ffffff] (2,0) circle (2.5pt);
\draw [fill=ffffff] (5,0) circle (2.5pt);
\draw [fill=ffffff] (3,0) circle (2.5pt);
\draw [fill=ffffff] (4,0) circle (2.5pt);
\end{scriptsize}
\end{tikzpicture}} & $\theta^2(\theta+K(n-1))$
\\

\hline 
$|ababcc\rangle$ &\raisebox{-.3\height}{\begin{tikzpicture}[line cap=round,line join=round,>=triangle 45,x=1cm,y=1cm, scale = 0.5]
\draw [shift={(2,0)},line width=2pt]  plot[domain=0:3.141592653589793,variable=\t]({1*1*cos(\t r)+0*1*sin(\t r)},{0*1*cos(\t r)+1*1*sin(\t r)});
\draw [shift={(5.5,0)},line width=2pt]  plot[domain=0:3.141592653589793,variable=\t]({1*0.5*cos(\t r)+0*0.5*sin(\t r)},{0*0.5*cos(\t r)+1*0.5*sin(\t r)});
\draw [shift={(3,0)},line width=2pt]  plot[domain=0:3.141592653589793,variable=\t]({1*1*cos(\t r)+0*1*sin(\t r)},{0*1*cos(\t r)+1*1*sin(\t r)});
\begin{scriptsize}
\draw [fill=ffffff] (1,0) circle (2.5pt);
\draw [fill=ffffff] (6,0) circle (2.5pt);
\draw [fill=ffffff] (2,0) circle (2.5pt);
\draw [fill=ffffff] (5,0) circle (2.5pt);
\draw [fill=ffffff] (3,0) circle (2.5pt);
\draw [fill=ffffff] (4,0) circle (2.5pt);
\end{scriptsize}
\end{tikzpicture}} & $\theta^2(\theta+K(n-1))$
\\
\hline 
$|abacbc\rangle$ &\raisebox{-.3\height}{\begin{tikzpicture}[line cap=round,line join=round,>=triangle 45,x=1cm,y=1cm, scale = 0.5]
\draw [shift={(2,0)},line width=2pt]  plot[domain=0:3.141592653589793,variable=\t]({1*1*cos(\t r)+0*1*sin(\t r)},{0*1*cos(\t r)+1*1*sin(\t r)});
\draw [shift={(3.5,0)},line width=2pt]  plot[domain=0:3.141592653589793,variable=\t]({1*1.5*cos(\t r)+0*1.5*sin(\t r)},{0*1.5*cos(\t r)+1*1.5*sin(\t r)});
\draw [shift={(5,0)},line width=2pt]  plot[domain=0:3.141592653589793,variable=\t]({1*1*cos(\t r)+0*1*sin(\t r)},{0*1*cos(\t r)+1*1*sin(\t r)});
\begin{scriptsize}
\draw [fill=ffffff] (1,0) circle (2.5pt);
\draw [fill=ffffff] (6,0) circle (2.5pt);
\draw [fill=ffffff] (2,0) circle (2.5pt);
\draw [fill=ffffff] (5,0) circle (2.5pt);
\draw [fill=ffffff] (3,0) circle (2.5pt);
\draw [fill=ffffff] (4,0) circle (2.5pt);
\end{scriptsize}
\end{tikzpicture}} & $\theta(\theta+K(n-1))^2$
\\
\hline 
$|abaccb\rangle$ &\raisebox{-.3\height}{\begin{tikzpicture}[line cap=round,line join=round,>=triangle 45,x=1cm,y=1cm, scale = 0.5]
\draw [shift={(2,0)},line width=2pt]  plot[domain=0:3.141592653589793,variable=\t]({1*1*cos(\t r)+0*1*sin(\t r)},{0*1*cos(\t r)+1*1*sin(\t r)});
\draw [shift={(4,0)},line width=2pt]  plot[domain=0:3.141592653589793,variable=\t]({1*2*cos(\t r)+0*2*sin(\t r)},{0*2*cos(\t r)+1*2*sin(\t r)});
\draw [shift={(4.5,0)},line width=2pt]  plot[domain=0:3.141592653589793,variable=\t]({1*0.5*cos(\t r)+0*0.5*sin(\t r)},{0*0.5*cos(\t r)+1*0.5*sin(\t r)});
\begin{scriptsize}
\draw [fill=ffffff] (1,0) circle (2.5pt);
\draw [fill=ffffff] (6,0) circle (2.5pt);
\draw [fill=ffffff] (2,0) circle (2.5pt);
\draw [fill=ffffff] (5,0) circle (2.5pt);
\draw [fill=ffffff] (3,0) circle (2.5pt);
\draw [fill=ffffff] (4,0) circle (2.5pt);
\end{scriptsize}
\end{tikzpicture}} & $\theta(\theta+K(n-1))^2$
\\
\hline 
$|abbacc\rangle$ &\raisebox{-.3\height}{\begin{tikzpicture}[line cap=round,line join=round,>=triangle 45,x=1cm,y=1cm, scale = 0.5]
\draw [shift={(2.5,0)},line width=2pt]  plot[domain=0:3.141592653589793,variable=\t]({1*1.5*cos(\t r)+0*1.5*sin(\t r)},{0*1.5*cos(\t r)+1*1.5*sin(\t r)});
\draw [shift={(2.5,0)},line width=2pt]  plot[domain=0:3.141592653589793,variable=\t]({1*0.5*cos(\t r)+0*0.5*sin(\t r)},{0*0.5*cos(\t r)+1*0.5*sin(\t r)});
\draw [shift={(5.5,0)},line width=2pt]  plot[domain=0:3.141592653589793,variable=\t]({1*0.5*cos(\t r)+0*0.5*sin(\t r)},{0*0.5*cos(\t r)+1*0.5*sin(\t r)});
\begin{scriptsize}
\draw [fill=ffffff] (1,0) circle (2.5pt);
\draw [fill=ffffff] (6,0) circle (2.5pt);
\draw [fill=ffffff] (2,0) circle (2.5pt);
\draw [fill=ffffff] (5,0) circle (2.5pt);
\draw [fill=ffffff] (3,0) circle (2.5pt);
\draw [fill=ffffff] (4,0) circle (2.5pt);
\end{scriptsize}
\end{tikzpicture}} & $\theta^2(\theta+K(n-1))$
\\
\hline 
$|abcabc\rangle$ &\raisebox{-.3\height}{\begin{tikzpicture}[line cap=round,line join=round,>=triangle 45,x=1cm,y=1cm, scale = 0.5]
\draw [shift={(2.5,0)},line width=2pt]  plot[domain=0:3.141592653589793,variable=\t]({1*1.5*cos(\t r)+0*1.5*sin(\t r)},{0*1.5*cos(\t r)+1*1.5*sin(\t r)});
\draw [shift={(3.5,0)},line width=2pt]  plot[domain=0:3.141592653589793,variable=\t]({1*1.5*cos(\t r)+0*1.5*sin(\t r)},{0*1.5*cos(\t r)+1*1.5*sin(\t r)});
\draw [shift={(4.5,0)},line width=2pt]  plot[domain=0:3.141592653589793,variable=\t]({1*1.5*cos(\t r)+0*1.5*sin(\t r)},{0*1.5*cos(\t r)+1*1.5*sin(\t r)});
\begin{scriptsize}
\draw [fill=ffffff] (1,0) circle (2.5pt);
\draw [fill=ffffff] (6,0) circle (2.5pt);
\draw [fill=ffffff] (2,0) circle (2.5pt);
\draw [fill=ffffff] (5,0) circle (2.5pt);
\draw [fill=ffffff] (3,0) circle (2.5pt);
\draw [fill=ffffff] (4,0) circle (2.5pt);
\end{scriptsize}
\end{tikzpicture}} & $\theta(\theta^2 + 3K(n-1)\theta + K^2(n-1)(2n-1))$
\\
\hline 
$|abcacb\rangle$ &\raisebox{-.3\height}{\begin{tikzpicture}[line cap=round,line join=round,>=triangle 45,x=1cm,y=1cm, scale = 0.5]
\draw [shift={(2.5,0)},line width=2pt]  plot[domain=0:3.141592653589793,variable=\t]({1*1.5*cos(\t r)+0*1.5*sin(\t r)},{0*1.5*cos(\t r)+1*1.5*sin(\t r)});
\draw [shift={(4,0)},line width=2pt]  plot[domain=0:3.141592653589793,variable=\t]({1*2*cos(\t r)+0*2*sin(\t r)},{0*2*cos(\t r)+1*2*sin(\t r)});
\draw [shift={(4,0)},line width=2pt]  plot[domain=0:3.141592653589793,variable=\t]({1*1*cos(\t r)+0*1*sin(\t r)},{0*1*cos(\t r)+1*1*sin(\t r)});
\begin{scriptsize}
\draw [fill=ffffff] (1,0) circle (2.5pt);
\draw [fill=ffffff] (6,0) circle (2.5pt);
\draw [fill=ffffff] (2,0) circle (2.5pt);
\draw [fill=ffffff] (5,0) circle (2.5pt);
\draw [fill=ffffff] (3,0) circle (2.5pt);
\draw [fill=ffffff] (4,0) circle (2.5pt);
\end{scriptsize}
\end{tikzpicture}} & $\theta(\theta^2 + 3K(n-1)\theta + K^2(n-1)(2n-1))$
\\
\hline 
$|abbcac\rangle$ &\raisebox{-.3\height}{\begin{tikzpicture}[line cap=round,line join=round,>=triangle 45,x=1cm,y=1cm, scale = 0.5]
\draw [shift={(2.5,0)},line width=2pt]  plot[domain=0:3.141592653589793,variable=\t]({1*0.5*cos(\t r)+0*0.5*sin(\t r)},{0*0.5*cos(\t r)+1*0.5*sin(\t r)});
\draw [shift={(3,0)},line width=2pt]  plot[domain=0:3.141592653589793,variable=\t]({1*2*cos(\t r)+0*2*sin(\t r)},{0*2*cos(\t r)+1*2*sin(\t r)});
\draw [shift={(5,0)},line width=2pt]  plot[domain=0:3.141592653589793,variable=\t]({1*1*cos(\t r)+0*1*sin(\t r)},{0*1*cos(\t r)+1*1*sin(\t r)});
\begin{scriptsize}
\draw [fill=ffffff] (1,0) circle (2.5pt);
\draw [fill=ffffff] (6,0) circle (2.5pt);
\draw [fill=ffffff] (2,0) circle (2.5pt);
\draw [fill=ffffff] (5,0) circle (2.5pt);
\draw [fill=ffffff] (3,0) circle (2.5pt);
\draw [fill=ffffff] (4,0) circle (2.5pt);
\end{scriptsize}
\end{tikzpicture}
} & $\theta(\theta+K(n-1))^2$
\\
\hline 
$|abcbac\rangle$ &\raisebox{-.3\height}{\begin{tikzpicture}[line cap=round,line join=round,>=triangle 45,x=1cm,y=1cm, scale = 0.5]
\draw [shift={(3,0)},line width=2pt]  plot[domain=0:3.141592653589793,variable=\t]({1*2*cos(\t r)+0*2*sin(\t r)},{0*2*cos(\t r)+1*2*sin(\t r)});
\draw [shift={(3,0)},line width=2pt]  plot[domain=0:3.141592653589793,variable=\t]({1*1*cos(\t r)+0*1*sin(\t r)},{0*1*cos(\t r)+1*1*sin(\t r)});
\draw [shift={(4.5,0)},line width=2pt]  plot[domain=0:3.141592653589793,variable=\t]({1*1.5*cos(\t r)+0*1.5*sin(\t r)},{0*1.5*cos(\t r)+1*1.5*sin(\t r)});
\begin{scriptsize}
\draw [fill=ffffff] (1,0) circle (2.5pt);
\draw [fill=ffffff] (6,0) circle (2.5pt);
\draw [fill=ffffff] (2,0) circle (2.5pt);
\draw [fill=ffffff] (5,0) circle (2.5pt);
\draw [fill=ffffff] (3,0) circle (2.5pt);
\draw [fill=ffffff] (4,0) circle (2.5pt);
\end{scriptsize}
\end{tikzpicture}} & $\theta(\theta^2 + 3K(n-1)\theta + K^2(n-1)(2n-1))$
\\
\hline 
$|abccab\rangle$ &\raisebox{-.3\height}{\begin{tikzpicture}[line cap=round,line join=round,>=triangle 45,x=1cm,y=1cm, scale = 0.5]
\draw [shift={(3,0)},line width=2pt]  plot[domain=0:3.141592653589793,variable=\t]({1*2*cos(\t r)+0*2*sin(\t r)},{0*2*cos(\t r)+1*2*sin(\t r)});
\draw [shift={(4,0)},line width=2pt]  plot[domain=0:3.141592653589793,variable=\t]({1*2*cos(\t r)+0*2*sin(\t r)},{0*2*cos(\t r)+1*2*sin(\t r)});
\draw [shift={(3.5,0)},line width=2pt]  plot[domain=0:3.141592653589793,variable=\t]({1*0.5*cos(\t r)+0*0.5*sin(\t r)},{0*0.5*cos(\t r)+1*0.5*sin(\t r)});
\begin{scriptsize}
\draw [fill=ffffff] (1,0) circle (2.5pt);
\draw [fill=ffffff] (6,0) circle (2.5pt);
\draw [fill=ffffff] (2,0) circle (2.5pt);
\draw [fill=ffffff] (5,0) circle (2.5pt);
\draw [fill=ffffff] (3,0) circle (2.5pt);
\draw [fill=ffffff] (4,0) circle (2.5pt);
\end{scriptsize}
\end{tikzpicture}} & $\theta(\theta^2 + 3K(n-1)\theta + 2K^2n(n-1)$
\\
\hline 
$|abbcca\rangle$ & \raisebox{-.3\height}{\begin{tikzpicture}[line cap=round,line join=round,>=triangle 45,x=1cm,y=1cm, scale = 0.5]
\draw [shift={(4.5,0)},line width=2pt]  plot[domain=0:3.141592653589793,variable=\t]({1*0.5*cos(\t r)+0*0.5*sin(\t r)},{0*0.5*cos(\t r)+1*0.5*sin(\t r)});
\draw [shift={(2.5,0)},line width=2pt]  plot[domain=0:3.141592653589793,variable=\t]({1*0.5*cos(\t r)+0*0.5*sin(\t r)},{0*0.5*cos(\t r)+1*0.5*sin(\t r)});
\draw [shift={(3.5,0)},line width=2pt]  plot[domain=0:3.141592653589793,variable=\t]({1*2.5*cos(\t r)+0*2.5*sin(\t r)},{0*2.5*cos(\t r)+1*2.5*sin(\t r)});
\begin{scriptsize}
\draw [fill=ffffff] (1,0) circle (2.5pt);
\draw [fill=ffffff] (6,0) circle (2.5pt);
\draw [fill=ffffff] (2,0) circle (2.5pt);
\draw [fill=ffffff] (5,0) circle (2.5pt);
\draw [fill=ffffff] (3,0) circle (2.5pt);
\draw [fill=ffffff] (4,0) circle (2.5pt);
\end{scriptsize}
\end{tikzpicture}} & $\theta(\theta+K(n-1))^2$
\\
\hline 
$|abcbca\rangle$ &\raisebox{-.3\height}{\begin{tikzpicture}[line cap=round,line join=round,>=triangle 45,x=1cm,y=1cm, scale = 0.5]
\draw [shift={(3.5,0)},line width=2pt]  plot[domain=0:3.141592653589793,variable=\t]({1*2.5*cos(\t r)+0*2.5*sin(\t r)},{0*2.5*cos(\t r)+1*2.5*sin(\t r)});
\draw [shift={(3,0)},line width=2pt]  plot[domain=0:3.141592653589793,variable=\t]({1*1*cos(\t r)+0*1*sin(\t r)},{0*1*cos(\t r)+1*1*sin(\t r)});
\draw [shift={(4,0)},line width=2pt]  plot[domain=0:3.141592653589793,variable=\t]({1*1*cos(\t r)+0*1*sin(\t r)},{0*1*cos(\t r)+1*1*sin(\t r)});
\begin{scriptsize}
\draw [fill=ffffff] (1,0) circle (2.5pt);
\draw [fill=ffffff] (6,0) circle (2.5pt);
\draw [fill=ffffff] (2,0) circle (2.5pt);
\draw [fill=ffffff] (5,0) circle (2.5pt);
\draw [fill=ffffff] (3,0) circle (2.5pt);
\draw [fill=ffffff] (4,0) circle (2.5pt);
\end{scriptsize}
\end{tikzpicture}} & $\theta(\theta^2 + 3K(n-1)\theta + K^2 (n-1)(2n-1)\theta)$
\\
\hline
$|abccba\rangle$ &  \raisebox{-.3\height}{\begin{tikzpicture}[line cap=round,line join=round,>=triangle 45,x=1cm,y=1cm, scale = 0.5]
\draw [shift={(3.5,0)},line width=2pt]  plot[domain=0:3.141592653589793,variable=\t]({1*2.5*cos(\t r)+0*2.5*sin(\t r)},{0*2.5*cos(\t r)+1*2.5*sin(\t r)});
\draw [shift={(3.5,0)},line width=2pt]  plot[domain=0:3.141592653589793,variable=\t]({1*1.5*cos(\t r)+0*1.5*sin(\t r)},{0*1.5*cos(\t r)+1*1.5*sin(\t r)});
\draw [shift={(3.5,0)},line width=2pt]  plot[domain=0:3.141592653589793,variable=\t]({1*0.5*cos(\t r)+0*0.5*sin(\t r)},{0*0.5*cos(\t r)+1*0.5*sin(\t r)});
\begin{scriptsize}
\draw [fill=ffffff] (1,0) circle (2.5pt);
\draw [fill=ffffff] (6,0) circle (2.5pt);
\draw [fill=ffffff] (2,0) circle (2.5pt);
\draw [fill=ffffff] (5,0) circle (2.5pt);
\draw [fill=ffffff] (3,0) circle (2.5pt);
\draw [fill=ffffff] (4,0) circle (2.5pt);
\end{scriptsize}
\end{tikzpicture}} & $\theta(\theta^2 + 3K(n-1)\theta + 2K^2 n(n-1))$ \\
\hline
\end{longtable}
\end{center}

\renewcommand\arraystretch{1}

\subsection{Some properties of the polynomials}

\begin{prop}
Let $x, y, z$ be mutually non-identical letters. Then 
$$|a_1 \cdots x \cdots y \cdots R(x,y)z \cdots z \cdots a_{2r}\rangle + |a_1 \cdots x \cdots y \cdots z \cdots R(x,y)z \cdots a_{2r}\rangle = 0$$
The locations of $x,y$ outside of $R(x,y)z$ are not important. They can be interchanged, individually placed before or after $R(x,y)z$. 
\end{prop}

\begin{proof}
The first term is equal to 
\begin{align*}
    & \state{a_1 \cdots x \cdots y \cdots R(x, y)x \cdots x \cdots a_{2k}} + \state{a_1 \cdots x \cdots y \cdots R(x, y)y \cdots y \cdots a_{2k}} \\
    = & -K \state{a_1 \cdots x \cdots y \cdots y \cdots x \cdots a_{2k}} + K\state{a_1 \cdots x \cdots y \cdots x \cdots y \cdots a_{2k}},
\end{align*}
while the second term is equal to 
\begin{align*}
    & \state{a_1 \cdots x \cdots y \cdots x \cdots R(x, y)x \cdots a_{2k}} + \state{a_1 \cdots x \cdots y \cdots y \cdots R(x, y)y \cdots a_{2k}} \\
    = & -K \state{a_1 \cdots x \cdots y \cdots x \cdots y \cdots a_{2k}} + K\state{a_1 \cdots x \cdots y \cdots y \cdots x \cdots a_{2k}}
\end{align*}
Their sum is zero. 
\end{proof}

\begin{rema}
This result simplifies the computation of the commutator $\state{a_1 \cdots a_{i-1} [a_i a_{i+1}] a_{i+2} \cdots a_{2k}}$. It suffices to consider those $a_j$'s that appears only once in $\{a_{i+2}, ..., a_{2k}\}$. 
\end{rema}

\begin{prop} \label{Brauer-Involn}
For every word $``a_1 \cdots a_{2k}",$ 
$$|a_1\cdots a_{2k}\rangle = |a_{2k}\cdots a_1\rangle $$
In other words, inverting the order of words does not change the polynomial. 
\end{prop}

\begin{proof}
We argue by induction. Assume the equality holds for all smaller $k$. Choose $i\in \{1, ..., 2k-1\}$ such that $a_{i} = a_{2k}$. If $i = 2k-1$ then the equation holds trivially. Now assume the equality holds for all larger $i$'s. Let $Z=a_i = a_{2k}$. Then the left-hand-side is
\begin{align*}
    & |a_1\cdots a_{i-1} Z a_{i+1}\cdots a_{2k-1}Z\rangle\\
    =& |a_1\cdots a_{i-1} a_{i+1} Z\cdots a_{2k-1} Z\rangle+ |a_1\cdots a_{i-1} [Z a_{i+1}]\cdots a_{2k-1}Z\rangle
\end{align*}
The first term, by induction hypothesis, is 
$$|Za_{2k-1} \cdots Z a_{i+1}a_{i-1} \cdots a_1\rangle,$$
while the second term is
\begin{align}
    & - \sum_{j=i+2}^{2k-1}  |a_1\cdots a_{i-1} a_{i+2} \cdots R(Z, a_{i+1})a_j \cdots a_{2k-1}Z\rangle - |a_1\cdots a_{i-1} a_{i+2} \cdots a_{2k-1} R(Z, a_{i+1}) Z\rangle\nonumber\\
    =& - \sum_{j=i+2}^{2k-1}  |a_1\cdots a_{i-1} a_{i+2} \cdots R(Z, a_{i+1})a_j \cdots a_{2k-1}Z\rangle + K(n-1) |a_1\cdots a_{i-1} a_{i+2} \cdots a_{2k-1}a_{i+1}\rangle \label{formula-1}
\end{align}
From the previous lemma, it suffices to consider the letters $a_j$ with its copy falling among $a_1, ..., a_{i-1}$. Let $j'$ be the position of the copy of $a_j$. Then for each such $j$ with $a_j \neq a_{i+1}$, the summand is
\begin{align}
    & - |a_1 \cdots a_{j'}\cdots a_{i-1}  a_{i+2} \cdots R(Z, a_{i+1})a_j \cdots a_{2k-1}Z\rangle \nonumber\\
    =& K|a_1 \cdots Z\cdots a_{i-1}  a_{i+2} \cdots a_{i+1} \cdots a_{2k-1}Z\rangle - K |a_1 \cdots a_{i+1} \cdots a_{i-1} Z a_{i+2} \cdots Z\cdots a_{2k-1}Z\rangle \label{LHS-summand}
\end{align}
If for some such $j$ we have $a_j = a_{i+1}$, i.e., the copy of the letter $a_{i+1}$ falls in $a_{i+2}, ..., a_{2k-1}$, then the summand is 
$$-K(n-1)|a_1 \cdots a_{i-1}a_{i+2} \cdots Z \cdots a_{2k-1} Z\rangle$$
Replacing $Z$ by $a_{i+1}$, we see that this term cancels out with the second term in (\ref{formula-1}). If instead, the copy of the letter $a_{i+1}$ does not fall in $a_{i+2}, ..., a_{2k-1}$, then the second term in (\ref{formula-1}) is kept. 

We now look at the right-hand-side: 
\begin{align*}
    & |Za_{2k-1} \cdots a_{i+1} Z a_{i-1} \cdots a_1\rangle \\
    =& |Za_{2k-1} \cdots Z a_{i+1} a_{i-1} \cdots a_1 \rangle + |Z a_{2k-1} \cdots [a_{i+1} Z] a_{i-1} \cdots a_1\rangle 
\end{align*}
The second term can be further expanded as 
$$-\sum_{j'=1}^{i-1}  |Z a_{2k-1}\cdots a_{i+2} a_{i-1} \cdots R(a_{i+1}, Z)a_{j'} \cdots a_1\rangle$$
We similarly consider only those $j'$ such that $a_{j'}$ appears among $a_{2k-1}, ..., a_{i+2}$. Denote the position of $a_{j'}$ by $j$. We see for each such $j'$ with $a_{j'} \neq a_{i+1}$, each summand is
\begin{align*}
    & -|Za_{2k-1}\cdots a_{i+2}  a_{i-1} \cdots R(a_{i+1}, Z)a_{j'} \cdots a_1\rangle\\
    = & - |Za_{2k-1}\cdots a_j \cdots  a_{i+2} a_{i-1} \cdots R(a_{i+1}, Z)a_{j'} \cdots a_1\rangle\\
    = & K |Za_{2k-1}\cdots a_{i+1} \cdots  a_{i+2} a_{i-1} \cdots Z \cdots a_1\rangle - K |Za_{2k-1}\cdots Z \cdots  a_{i+2} a_{i-1} \cdots a_{i+1} \cdots a_1\rangle
\end{align*}
Induction hypothesis shows that each such guy is identical to (\ref{LHS-summand}). Now if $a_{i+1}$ appears in $a_{i+2}, ..., a_{2k-1}$,  then it does not appear in the sum of the right-hand-side. If $a_{i+1}$ does not appears in $a_{i+2}, ..., a_{2k-1}$, then there is one extra summand that will cancel out with the second term in (\ref{formula-1}). So we conclude the proof. 
\end{proof}

\begin{rema}
If we interpret the parallel tensors in terms of $(0, 2k)$-Brauer diagrams as in \cite{LZ-Category}, then (\ref{Brauer-tensor}) can be understood as
$$|D_1 \otimes D_2 \rangle = |D_1\rangle \cdot |D_2\rangle$$
for two diagrams $D_1, D_2$ of type $(0, 2k)$. 
The result of Proposition \ref{Brauer-Involn} can be understood as 
$$|D^\#\rangle = |D\rangle$$
where the operator $\#$ is an involution on Brauer diagrams that sends a diagram to its reflection about a vertical line. 
\end{rema}

Now we view $|a_1 \cdots a_{2k}\rangle$ as a polynomial in $\theta$ and $n$ and study the highest degree homogeneous component. 

\begin{prop}\label{dominant}
Let $a_1, ..., a_k$ be distinct letters. Then for every $\sigma\in \text{Sym}_k$, the degree-$k$ component of 
\begin{align}
    |a_1 \cdots a_k a_{\sigma(1)}\cdots a_{\sigma(k)}\rangle \label{actualpolynomial}
\end{align}
is 
\begin{align}
    \theta(\theta+Kn))\cdots(\theta+(k-1)Kn), \label{dominantpolynomial}
\end{align}
\end{prop}

\begin{proof}
We prove by induction. The case $k=1$ holds trivially. Assume the statement holds for every smaller $k$. Consider first the special case of $$|a_1 \cdots a_k a_k \cdots a_1\rangle.$$
We attempt to move the central $a_ka_k$ to the right using the commutator formula (\ref{recur-2}): 
\begin{align*}
    & |a_1 \cdots a_{k-1}a_k a_k a_{k-1}\cdots a_1\rangle\\ 
    & \quad = |a_1 \cdots a_{k-1} a_k a_{k-1}a_k \cdots a_1\rangle + |a_1 \cdots a_{k-1} a_k [a_k a_{k-1}] \cdots a_1\rangle \\
    &  \quad = |a_1 \cdots  a_{k-1}a_{k-1}a_ka_k \cdots a_1\rangle +|a_1 \cdots a_{k-1} [a_k a_{k-1}] a_k \cdots a_1\rangle + |a_1 \cdots a_{k-1} a_k [a_k a_{k-1}] \cdots a_1\rangle\\
    & \quad = |a_1 \cdots  a_{k-1}a_{k-1}a_ka_k \cdots a_1\rangle -|a_1 \cdots a_{k-1} ,R(a_k, a_{k-1}) a_k, \cdots a_1\rangle + 2|a_1 \cdots a_{k-1} a_k [a_k a_{k-1}] \cdots a_1\rangle\\
    & \quad = |a_1 \cdots  a_{k-1}a_{k-1}a_ka_k \cdots a_1\rangle +K(n-1) |a_1 \cdots a_{k-1} a_{k-1} \cdots a_1\rangle + 2|a_1 \cdots a_{k-1} a_k [a_k a_{k-1}] \cdots a_1\rangle
\end{align*}
From the induction hypothesis, the degree-$k$ component of the second term is 
\begin{align}
    Kn \theta (\theta+Kn)\cdots (\theta+(k-2)Kn). \label{contrib}
\end{align} 
We show now that the degree-$k$ component of the third term is zero:
\begin{align*}
    & |a_1 \cdots a_{k-1} a_k [a_k a_{k-1}] \cdots a_1\rangle\\ & \quad = \sum_{j=1}^{k-2}|a_1 \cdots a_{j-1}, a_j, a_{j+1} \cdots a_{k-2},  a_{k-1} a_k, a_{k-2}\cdots a_{j+1},  R(a_{k-1}, a_k)a_{j}, a_{j-1} \cdots a_1\rangle \\
    & \quad = -\sum_{j=1}^{k-2}|a_1 \cdots a_{j-1}, a_j, a_{j+1} \cdots a_{k-2}, a_{k-1} a_k, a_{k-2}\cdots a_{j+1},  R(a_{k-1}, a_k)a_{j}, a_{j-1} \cdots a_1\rangle \\
    & \quad = K\sum_{j=1}^{k-2}|a_1 \cdots a_{j-1}, a_{k-1}, a_{j+1} \cdots a_{k-2}, a_{k-1} a_k, a_{k-2}\cdots a_{j+1},  a_k, a_{j-1} \cdots a_1\rangle \\
    & \qquad - K\sum_{j=1}^{k-2}|a_1 \cdots a_{j-1}, a_{k}, a_{j+1} \cdots a_{k-2}, a_{k-1} a_k, a_{k-2}\cdots a_{j+1},  a_{k-1}, a_{j-1} \cdots a_1\rangle
\end{align*}
The polynomials involved on the right-hand-side are at most of degree $k-1$ and thus do not contain any degree-$k$ component. 

Now we handle the first term by further moving $a_ka_k$ one step right. By a similar computation we see that 
\begin{align*}
    & |a_1 \cdots a_{k-2} a_{k-1}a_{k-1}a_ka_k a_{k-2} \cdots a_1\rangle \\
    & \quad = |a_1 \cdots a_{k-2} a_{k-1}a_{k-1} a_{k-2}a_ka_k \cdots a_1\rangle + K(n-1)|a_1 \cdots a_{k-2} a_{k-1}a_{k-1} a_{k-2} \cdots a_1\rangle \\
    & \qquad + 2|a_1 \cdots a_{k-2} a_{k-1}a_{k-1}a_k[a_k a_{k-2}] \cdots a_1\rangle 
\end{align*}
Similarly, the degree-$k$ contribution of the second term is (\ref{contrib}), and that of the third term is zero. Repeating the process to the first term, until the first term becomes 
$$|a_1 \cdots a_{k-1}a_{k-1}\cdots a_1 a_k a_k\rangle.$$
The degree-$k$ component is then 
$$\theta \cdot \theta (\theta+Kn)\cdots (\theta+(k-2)Kn)$$
Since the move is performed $k-1$ times, we conclude that the degree-$k$ component of $|a_1 \cdots a_k a_k\cdots a_1\rangle$ is
\begin{align*}
    & \theta \cdot \theta(\theta+Kn)\cdots (\theta+(k-2)Kn) + (k-1)Kn \cdot  \theta(\theta+Kn)\cdots (\theta+(k-2)Kn)\\
    & \qquad = \theta(\theta+Kn)\cdots (\theta+(k-2)Kn)(\theta+(k-1)Kn)
\end{align*}

Now we consider the general case. It suffices to notice that for every $\sigma\in \text{Sym}_k$ and for every $i=1, ..., k$, 
\begin{align*}
    & |a_1 \cdots a_k a_{\sigma(1)} \cdots [a_{\sigma(i)}a_{\sigma(i+1)}] \cdots a_{\sigma(k)}\rangle \\
    & \quad = -\sum_{j=i+2}^k |a_1 \cdots a_{\sigma(j)} \cdots a_k a_{\sigma(1)} \cdots a_{\sigma(i-1)}  a_{\sigma(i+1)} \cdots R(a_{\sigma(i)}a_{\sigma(i+1)}) a_{\sigma(j)}\cdots  a_{\sigma(k)}\rangle \\
    & \quad = K \sum_{j=i+2}^k |a_1 \cdots a_{\sigma(i)} \cdots a_k a_{\sigma(1)} \cdots a_{\sigma(i-1)}  a_{\sigma(i+2)} \cdots a_{\sigma(i+1)} \cdots  a_{\sigma(k)}\rangle \\
    & \qquad -
    K \sum_{j=i+2}^k |a_1 \cdots a_{\sigma(i+1)} \cdots a_k a_{\sigma(1)} \cdots a_{\sigma(i-1)}  a_{\sigma(i+2)} \cdots a_{\sigma(i)}\cdots  a_{\sigma(k)}\rangle 
\end{align*}
is of degree at most $k-1$. Thus switching the position of adjacent letters on the right-half of the word does not change the degree-$k$ component. This concludes the proof. 
\end{proof}

Similar computations yield the following results. 
\begin{cor}
\begin{enumerate}
    \item The degree-$k$ component of 
$$|a_1 \cdots a_i a_{i+1}\cdots a_{k-1}a_{k-1} \cdots a_{i+1}  a_k a_k a_i \cdots a_1\rangle$$
is 
$$\theta(\theta+Kn)\cdots (\theta+(k-2)Kn) (\theta+iKn)$$
\item The degree-$k$ component of 
$$|a_1 \cdots a_i a_{i+1}\cdots a_ja_{k-1}a_{k-1}a_{j+1}\cdots a_{k-2}a_{k-2} \cdots a_{j+1}a_j \cdots a_{i+1}  a_k a_k a_i \cdots a_1\rangle$$
is 
$$\theta(\theta+Kn)\cdots (\theta+(k-3)Kn) (\theta+iKn)(\theta+jKn)$$
\item For every $\sigma \in \text{Sym}\{1, ..., k-2\}$, the degree-$k$ component of 
$$|a_1 \cdots a_{k-1} \cdots a_j a_{k-1}a_{j+1}\cdots a_{k-2}a_{\sigma^{-1}(k-2)} \cdots a_{\sigma^{-1}(i+1)}  a_k  a_{\sigma^{-1}(i)} \cdots a_k \cdots  a_{\sigma^{-1}(1)}\rangle$$
is 
$$\theta(\theta+Kn)\cdots (\theta+(k-3)Kn) (\theta+iKn)(\theta+jKn)$$
\end{enumerate}
\end{cor}

\begin{rema}
Part (3) gives the highest degree component of the polynomial for every graph with an arc on the left half and on the right half. The component is determined by the position of the vertex of the arc that is closer to the center, e.g., the filled vertices shown in the following picture. \\

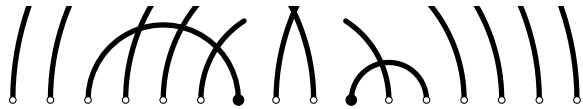
\begin{center}
\begin{tikzpicture}[line cap=round,line join=round,>=triangle 45,x=1cm,y=1cm, scale = 0.5]
\clip(0, -1) rectangle (17, 2.5);
\draw [shift={(5,0)},line width=2pt]  plot[domain=0:3.141592653589793,variable=\t]({1*2*cos(\t r)+0*2*sin(\t r)},{0*2*cos(\t r)+1*2*sin(\t r)});
\draw [shift={(11,0)},line width=2pt]  plot[domain=0:3.141592653589793,variable=\t]({1*1*cos(\t r)+0*1*sin(\t r)},{0*1*cos(\t r)+1*1*sin(\t r)});
\draw [shift={(8.5,0)},line width=2pt]  plot[domain=0:1,variable=\t]({1*7.5*cos(\t r)+0*7.5*sin(\t r)},{0*7.5*cos(\t r)+1*7.5*sin(\t r)});
\draw [shift={(8.5,0)},line width=2pt]  plot[domain=2.141592653589793:3.141592653589793,variable=\t]({1*7.5*cos(\t r)+0*7.5*sin(\t r)},{0*7.5*cos(\t r)+1*7.5*sin(\t r)});
\draw [shift={(8.5,0)},line width=2pt]  plot[domain=0:1,variable=\t]({1*6.5*cos(\t r)+0*6.5*sin(\t r)},{0*6.5*cos(\t r)+1*6.5*sin(\t r)});
\draw [shift={(9,0)},line width=2pt]  plot[domain=0:1,variable=\t]({1*5*cos(\t r)+0*5*sin(\t r)},{0*5*cos(\t r)+1*5*sin(\t r)});
\draw [shift={(9,0)},line width=2pt]  plot[domain=0:1,variable=\t]({1*4*cos(\t r)+0*4*sin(\t r)},{0*4*cos(\t r)+1*4*sin(\t r)});
\draw [shift={(8.5,0)},line width=2pt]  plot[domain=0:1,variable=\t]({1*2.5*cos(\t r)+0*2.5*sin(\t r)},{0*2.5*cos(\t r)+1*2.5*sin(\t r)});
\draw [shift={(3.5,0)},line width=2pt]  plot[domain=0:1,variable=\t]({1*5.5*cos(\t r)+0*5.5*sin(\t r)},{0*5.5*cos(\t r)+1*5.5*sin(\t r)});
\draw [shift={(8.5,0)},line width=2pt]  plot[domain=2.141592653589793:3.141592653589793,variable=\t]({1*6.5*cos(\t r)+0*6.5*sin(\t r)},{0*6.5*cos(\t r)+1*6.5*sin(\t r)});
\draw [shift={(9,0)},line width=2pt]  plot[domain=2.141592653589793:3.141592653589793,variable=\t]({1*5*cos(\t r)+0*5*sin(\t r)},{0*5*cos(\t r)+1*5*sin(\t r)});
\draw [shift={(9,0)},line width=2pt]  plot[domain=2.141592653589793:3.141592653589793,variable=\t]({1*4*cos(\t r)+0*4*sin(\t r)},{0*4*cos(\t r)+1*4*sin(\t r)});
\draw [shift={(8.5,0)},line width=2pt]  plot[domain=2.141592653589793:3.141592653589793,variable=\t]({1*2.5*cos(\t r)+0*2.5*sin(\t r)},{0*2.5*cos(\t r)+1*2.5*sin(\t r)});
\draw [shift={(14,0)},line width=2pt]  plot[domain=2.141592653589793:3.141592653589793,variable=\t]({1*6*cos(\t r)+0*6*sin(\t r)},{0*6*cos(\t r)+1*6*sin(\t r)});
\begin{scriptsize}
\draw [fill=ffffff] (1,0) circle (2.5pt);
\draw [fill=ffffff] (6,0) circle (2.5pt);
\draw [fill=ffffff] (2,0) circle (2.5pt);
\draw [fill=ffffff] (5,0) circle (2.5pt);
\draw [fill=ffffff] (3,0) circle (2.5pt);
\draw [fill=ffffff] (4,0) circle (2.5pt);
\draw [fill=black] (7,0) circle (4pt);
\draw [fill=ffffff] (8,0) circle (2.5pt);
\draw [fill=ffffff] (9,0) circle (2.5pt);
\draw [fill=black] (10,0) circle (4pt);
\draw [fill=ffffff] (11,0) circle (2.5pt);
\draw [fill=ffffff] (14,0) circle (2.5pt);
\draw [fill=ffffff] (15,0) circle (2.5pt);
\draw [fill=ffffff] (16,0) circle (2.5pt);
\draw [fill=ffffff] (12,0) circle (2.5pt);
\draw [fill=ffffff] (13,0) circle (2.5pt);
\end{scriptsize}
\end{tikzpicture}
\end{center}
\end{rema}

\section{A conjecture motivated by the study of vertex algebras}

In this section we explain a conjecture regarding the polynomials associated to the parallel tensors. It is a key step in generalizing the results in \cite{Q-MOSVA-2d} to higher dimensional space forms.  

\subsection{Linking operator}
Let $W$ be the free abelian group generated by all the words. The map associating a word $w$ to the polynomial $\state{w}$ extends naturally to $W$. For each $1\leq i < j \leq 2k$, we consider the following linking operator on W. 
\begin{align*}
L\binom{i}{j} &(``a_1 \cdots a_{i-1} a_i a_{i+1} \cdots a_{j-1} a_j a_{j+1} \cdots a_{2k}") \\
&= \left\{\begin{array}{ll}
n \cdot ``a_1 \cdots a_{i-1}a_{i+1} \cdots a_{j-1}a_{j+1}\cdots a_{2k}" & \text{ if }a_i = a_j,  \\
``I_{a_ia_j} a_1 \cdots a_{i-1} a_{i+1}\cdots a_{j-1}a_{j+1} \cdots a_{2k}"  & \text{ otherwise.} 
\end{array}\right.  \end{align*}
Here $I_{a_ia_j}$ identifies the letter $a_i$ with $a_j$. 

In natural language, we first remove the letters at the $i$-th and $j$-th position. If the removed letters are the same, then we multiply the remaining word by $n$, ending up in a word of length $2k-2$. If the letters are not the same, then these letters appear somewhere else. Identifying these letter gives a word of length $2k-2$. 

The linking operator gets its name from the interpretation of diagrams, as the operator does precisely the same thing of linking vertex $i$ and vertex $j$.\\

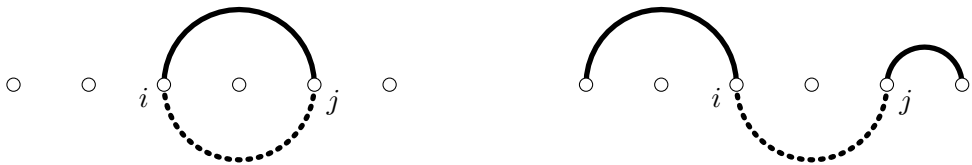
\begin{center}
\begin{tikzpicture}[line cap=round,line join=round,>=triangle 45,x=1cm,y=1cm]
\draw (2.531948051948048,0.09275361243436628) node[anchor=north west] {$i$};
\draw (5.031948051948046,0.06724636901491628) node[anchor=north west] {$j$};
\draw [shift={(4,0)},line width=2pt,dash pattern=on 1pt off 3pt]  plot[domain=3.141592653589793:2*3.141592653589793,variable=\t]({1*1*cos(\t r)+0*1*sin(\t r)},{0*1*cos(\t r)+1*1*sin(\t r)});
\draw [shift={(4,0)},line width=2pt]  plot[domain=0:3.141592653589793,variable=\t]({1*1*cos(\t r)+0*1*sin(\t r)},{0*1*cos(\t r)+1*1*sin(\t r)});
\begin{scriptsize}
\draw [fill=ffffff] (1,0) circle (2.5pt);
\draw [fill=ffffff] (6,0) circle (2.5pt);
\draw [fill=ffffff] (2,0) circle (2.5pt);
\draw [fill=ffffff] (5,0) circle (2.5pt);
\draw [fill=ffffff] (3,0) circle (2.5pt);
\draw [fill=ffffff] (4,0) circle (2.5pt);
\end{scriptsize}
\end{tikzpicture} \qquad \qquad \qquad 
\begin{tikzpicture}[line cap=round,line join=round,>=triangle 45,x=1cm,y=1cm]
\draw [shift={(5.5,0)},line width=2pt]  plot[domain=0:3.141592653589793,variable=\t]({1*0.5*cos(\t r)+0*0.5*sin(\t r)},{0*0.5*cos(\t r)+1*0.5*sin(\t r)});
\draw [shift={(2,0)},line width=2pt]  plot[domain=0:3.141592653589793,variable=\t]({1*1*cos(\t r)+0*1*sin(\t r)},{0*1*cos(\t r)+1*1*sin(\t r)});
\draw (2.531948051948048,0.09275361243436485) node[anchor=north west] {$i$};
\draw (5.031948051948046,0.06724636901491485) node[anchor=north west] {$j$};
\draw [shift={(4,0)},line width=2pt,dash pattern=on 1pt off 3pt]  plot[domain=3.14159265358979:2*3.141592653589793,variable=\t]({1*1*cos(\t r)+0*1*sin(\t r)},{0*1*cos(\t r)+1*1*sin(\t r)});
\begin{scriptsize}
\draw [fill=ffffff] (1,0) circle (2.5pt);
\draw [fill=ffffff] (6,0) circle (2.5pt);
\draw [fill=ffffff] (2,0) circle (2.5pt);
\draw [fill=ffffff] (5,0) circle (2.5pt);
\draw [fill=ffffff] (3,0) circle (2.5pt);
\draw [fill=ffffff] (4,0) circle (2.5pt);
\end{scriptsize}
\end{tikzpicture}
\end{center}
Here each occurrence of a circle brings forth a multiplication by $n$.

For example, consider the word $``aabccb"$. Then 
\begin{align*}
    L\binom 1 2 ``aabccb" &= n \cdot ``\cancel{a}\cancel{a}bccb" = n\cdot ``bccb''\\
    L\binom 1 3 ``aabccb" &= ``I_{ab}\cancel{a}a\cancel{b}ccb" = ``I_{ab}accb" = ``acca". 
\end{align*}
In graphs: 
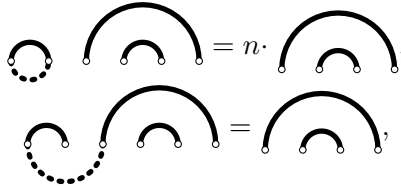
\begin{center}
\raisebox{-.35\height}{\begin{tikzpicture}[line cap=round,line join=round,>=triangle 45,x=1cm,y=1cm, scale = 0.5]
\draw [shift={(1.5,0)},line width=2pt]  plot[domain=0:3.141592653589793,variable=\t]({1*0.5*cos(\t r)+0*0.5*sin(\t r)},{0*0.5*cos(\t r)+1*0.5*sin(\t r)});
\draw [shift={(4.5,0)},line width=2pt]  plot[domain=0:3.141592653589793,variable=\t]({1*1.5*cos(\t r)+0*1.5*sin(\t r)},{0*1.5*cos(\t r)+1*1.5*sin(\t r)});
\draw [shift={(4.5,0)},line width=2pt]  plot[domain=0:3.141592653589793,variable=\t]({1*0.5*cos(\t r)+0*0.5*sin(\t r)},{0*0.5*cos(\t r)+1*0.5*sin(\t r)});
\draw [shift={(1.5,0)},line width=2pt,dash pattern=on 1pt off 3pt]  plot[domain=3.141592653589793:6.283185307179586,variable=\t]({1*0.5*cos(\t r)+0*0.5*sin(\t r)},{0*0.5*cos(\t r)+1*0.5*sin(\t r)});
\begin{scriptsize}
\draw [fill=ffffff] (1,0) circle (2.5pt);
\draw [fill=ffffff] (2,0) circle (2.5pt);
\draw [fill=ffffff] (5,0) circle (2.5pt);
\draw [fill=ffffff] (3,0) circle (2.5pt);
\draw [fill=ffffff] (4,0) circle (2.5pt);
\draw [fill=ffffff] (6,0) circle (2.5pt);
\end{scriptsize}
\end{tikzpicture}} $= n\cdot$ \raisebox{-.3\height}{\begin{tikzpicture}[line cap=round,line join=round,>=triangle 45,x=1cm,y=1cm, scale = 0.5]
\draw [shift={(2.5,0)},line width=2pt]  plot[domain=0:3.141592653589793,variable=\t]({1*1.5*cos(\t r)+0*1.5*sin(\t r)},{0*1.5*cos(\t r)+1*1.5*sin(\t r)});
\draw [shift={(2.5,0)},line width=2pt]  plot[domain=0:3.141592653589793,variable=\t]({1*0.5*cos(\t r)+0*0.5*sin(\t r)},{0*0.5*cos(\t r)+1*0.5*sin(\t r)});
\begin{scriptsize}
\draw [fill=ffffff] (1,0) circle (2.5pt);
\draw [fill=ffffff] (2,0) circle (2.5pt);
\draw [fill=ffffff] (3,0) circle (2.5pt);
\draw [fill=ffffff] (4,0) circle (2.5pt);
\end{scriptsize}
\end{tikzpicture}}, 

\raisebox{-.5\height}{\begin{tikzpicture}[line cap=round,line join=round,>=triangle 45,x=1cm,y=1cm, scale = 0.5]
\draw [shift={(1.5,0)},line width=2pt]  plot[domain=0:3.141592653589793,variable=\t]({1*0.5*cos(\t r)+0*0.5*sin(\t r)},{0*0.5*cos(\t r)+1*0.5*sin(\t r)});
\draw [shift={(4.5,0)},line width=2pt]  plot[domain=0:3.141592653589793,variable=\t]({1*1.5*cos(\t r)+0*1.5*sin(\t r)},{0*1.5*cos(\t r)+1*1.5*sin(\t r)});
\draw [shift={(4.5,0)},line width=2pt]  plot[domain=0:3.141592653589793,variable=\t]({1*0.5*cos(\t r)+0*0.5*sin(\t r)},{0*0.5*cos(\t r)+1*0.5*sin(\t r)});
\draw [shift={(2,0)},line width=2pt,dash pattern=on 1pt off 3pt]  plot[domain=3.141592653589793:6.283185307179586,variable=\t]({1*1*cos(\t r)+0*1*sin(\t r)},{0*1*cos(\t r)+1*1*sin(\t r)});
\begin{scriptsize}
\draw [fill=ffffff] (1,0) circle (2.5pt);
\draw [fill=ffffff] (2,0) circle (2.5pt);
\draw [fill=ffffff] (5,0) circle (2.5pt);
\draw [fill=ffffff] (3,0) circle (2.5pt);
\draw [fill=ffffff] (4,0) circle (2.5pt);
\draw [fill=ffffff] (6,0) circle (2.5pt);
\end{scriptsize}
\end{tikzpicture}} $=$ \raisebox{-.3\height}{\begin{tikzpicture}[line cap=round,line join=round,>=triangle 45,x=1cm,y=1cm, scale = 0.5]
\draw [shift={(2.5,0)},line width=2pt]  plot[domain=0:3.141592653589793,variable=\t]({1*1.5*cos(\t r)+0*1.5*sin(\t r)},{0*1.5*cos(\t r)+1*1.5*sin(\t r)});
\draw [shift={(2.5,0)},line width=2pt]  plot[domain=0:3.141592653589793,variable=\t]({1*0.5*cos(\t r)+0*0.5*sin(\t r)},{0*0.5*cos(\t r)+1*0.5*sin(\t r)});
\begin{scriptsize}
\draw [fill=ffffff] (1,0) circle (2.5pt);
\draw [fill=ffffff] (2,0) circle (2.5pt);
\draw [fill=ffffff] (3,0) circle (2.5pt);
\draw [fill=ffffff] (4,0) circle (2.5pt);
\end{scriptsize}
\end{tikzpicture}}, 
\end{center}

For each $l=1, ..., k$ and for each sequence $i_1, i_2,..., i_{2l-1}, i_{2l}$ of distinct numbers in $\{1, ..., 2k\}$ satisfying $i_{2j-1}<i_{2j}, j = 1, 2, ..., l$, we similarly define the linking operator 
$$L\left(\begin{array}{cccc} 
i_1 & i_3 & \cdots & i_{2l-1}\\
i_2 & i_4 & \cdots & i_{2l}
\end{array}\right)$$
that sends a word of length $2k$ to the polynomial for the word obtained by linking the $i_1$-th vertex with $i_2$-th vertex, $i_3$-th vertex with $i_4$-th vertex, ..., $i_{2l-1}$-th vertex with $i_{2l}$-th vertex.

\subsection{Statement of the problem}
Consider now the word $$\tau = ``a_1a_2\cdots a_{k-1}a_k a_ka_{k-1}\cdots a_2a_1"$$
corresponding to the following graph
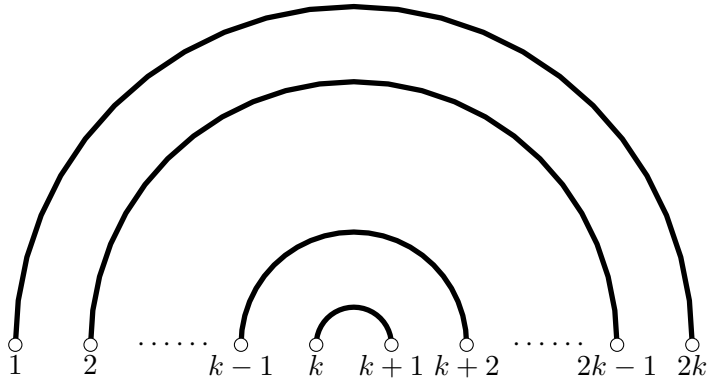
\begin{center}
\begin{tikzpicture}[line cap=round,line join=round,>=triangle 45,x=1cm,y=1cm]
\draw [shift={(5.5,0)},line width=2pt]  plot[domain=0:3.141592653589793,variable=\t]({1*4.5*cos(\t r)+0*4.5*sin(\t r)},{0*4.5*cos(\t r)+1*4.5*sin(\t r)});
\draw [shift={(5.5,0)},line width=2pt]  plot[domain=0:3.141592653589793,variable=\t]({1*3.5*cos(\t r)+0*3.5*sin(\t r)},{0*3.5*cos(\t r)+1*3.5*sin(\t r)});
\draw [shift={(5.5,0)},line width=2pt]  plot[domain=0:3.141592653589793,variable=\t]({1*1.5*cos(\t r)+0*1.5*sin(\t r)},{0*1.5*cos(\t r)+1*1.5*sin(\t r)});
\draw [shift={(5.5,0)},line width=2pt]  plot[domain=0:3.141592653589793,variable=\t]({1*0.5*cos(\t r)+0*0.5*sin(\t r)},{0*0.5*cos(\t r)+1*0.5*sin(\t r)});
\draw (2.486457002284634,0) node[anchor=west] {$\cdots\cdots$};
\draw (7.474594586284257,0) node[anchor=west] {$\cdots\cdots$};
\draw (1,0) node[anchor=north] {$1$};
\draw (2,0) node[anchor=north] {$2$};
\draw (4,0) node[anchor=north] {$k-1$};
\draw (5,0) node[anchor=north] {$k$};
\draw (6,0) node[anchor=north] {$k+1$};
\draw (7,0) node[anchor=north] {$k+2$};
\draw (9,0) node[anchor=north] {$2k-1$};
\draw (10,0) node[anchor=north] {$2k$};
\begin{scriptsize}
\draw [fill=ffffff] (1,0) circle (2.5pt);
\draw [fill=ffffff] (6,0) circle (2.5pt);
\draw [fill=ffffff] (2,0) circle (2.5pt);
\draw [fill=ffffff] (5,0) circle (2.5pt);
\draw [fill=ffffff] (4,0) circle (2.5pt);
\draw [fill=ffffff] (7,0) circle (2.5pt);
\draw [fill=ffffff] (9,0) circle (2.5pt);
\draw [fill=ffffff] (10,0) circle (2.5pt);
\end{scriptsize}
\end{tikzpicture}
\end{center}
To some extent, the polynomial of this word has the ``largest'' highest degree component among those of words of the same length. 

Fix any $l=1,...,[k/2]$ and any sequence $i_1, i_2, \cdots, i_{2l-1}, i_{2l}$ satisfying $i_{2j-1}<i_{2j}, j=1,..., l$. Let 
$$\tau \begin{pmatrix}
i_1 & i_3 & \cdots & i_{2l-1}\\ 
i_2 & i_4 & \cdots & i_{2l}
\end{pmatrix} = L\left(\begin{array}{cccc} 
2r-i_2+1 & 2r-i_4+1 & \cdots & 2r-i_{2l}+1
\\
2r-i_1+1 & 2r-i_3+1 & \cdots & 2r-i_{2l-1}+1
\end{array}\right)\tau$$
Intuitively, $\tau \begin{pmatrix}
i_1 & i_3 & \cdots & i_{2l-1}\\ 
i_2 & i_4 & \cdots & i_{2l}
\end{pmatrix}$ is the word associated to the graph obtained from linking the letters $a_1, ..., a_k$ appearing on the right of the graph, in such a way that $a_{i_1}$ is linked with $a_{i_2}$, ..., $a_{i_{2l-1}}$ is linked with $a_{i_{2l}}$. Observe that no circles will be formed in this process. 

It is also obvious that
\begin{align}
    & L\left(\begin{array}{cccc} 
    i_1 & i_3 & \cdots & i_{2l-1} \\
    i_2 & i_4 & \cdots & i_{2l} 
\end{array}\right)\tau \begin{pmatrix}
i_1 & i_3 & \cdots & i_{2l-1}\\ 
i_2 & i_4 & \cdots & i_{2l}
\end{pmatrix} & = n^k ``a_1\cdots a_{k-2l}a_{k-2l}\cdots a_1", \label{diagonal}
\end{align}
since linking $a_{i_{2j-1}}, a_{i_{2j}}$ both from the left and from the right results in a circle for each $j=1, ..., l$. The following picture shows a simple case when $i_1 = 1, i_2 = 3$:
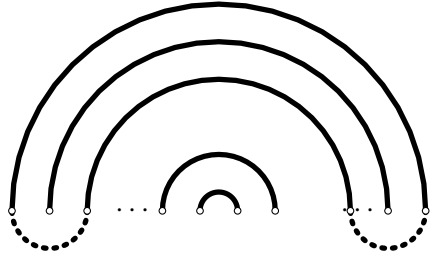
\begin{center}
\begin{tikzpicture}[line cap=round,line join=round,>=triangle 45,x=1cm,y=1cm, scale = 0.5]
\draw [shift={(6.5,0)},line width=2pt]  plot[domain=0:3.141592653589793,variable=\t]({1*5.5*cos(\t r)+0*5.5*sin(\t r)},{0*5.5*cos(\t r)+1*5.5*sin(\t r)});
\draw [shift={(6.5,0)},line width=2pt]  plot[domain=0:3.141592653589793,variable=\t]({1*4.5*cos(\t r)+0*4.5*sin(\t r)},{0*4.5*cos(\t r)+1*4.5*sin(\t r)});
\draw [shift={(6.5,0)},line width=2pt]  plot[domain=0:3.141592653589793,variable=\t]({1*1.5*cos(\t r)+0*1.5*sin(\t r)},{0*1.5*cos(\t r)+1*1.5*sin(\t r)});
\draw [shift={(6.5,0)},line width=2pt]  plot[domain=0:3.141592653589793,variable=\t]({1*0.5*cos(\t r)+0*0.5*sin(\t r)},{0*0.5*cos(\t r)+1*0.5*sin(\t r)});
\draw [shift={(6.5,0)},line width=2pt]  plot[domain=0:3.141592653589793,variable=\t]({1*3.5*cos(\t r)+0*3.5*sin(\t r)},{0*3.5*cos(\t r)+1*3.5*sin(\t r)});
\draw [shift={(2,0)},line width=2pt, dash pattern=on 1pt off 3pt]  plot[domain=3.141592653589793:6.283185307179586,variable=\t]({1*1*cos(\t r)+0*1*sin(\t r)},{0*1*cos(\t r)+1*1*sin(\t r)});
\draw [shift={(11,0)},line width=2pt, dash pattern=on 1pt off 3pt]  plot[domain=3.141592653589793:6.283185307179586,variable=\t]({1*1*cos(\t r)+0*1*sin(\t r)},{0*1*cos(\t r)+1*1*sin(\t r)});
\draw (3.486457002284634,0) node[anchor=west] {$\cdots$};
\draw (9.474594586284257,0) node[anchor=west] {$\cdots$};

\begin{scriptsize}
\draw [fill=ffffff] (1,0) circle (2.5pt);
\draw [fill=ffffff] (2,0) circle (2.5pt);
\draw [fill=ffffff] (6,0) circle (2.5pt);
\draw [fill=ffffff] (3,0) circle (2.5pt);
\draw [fill=ffffff] (5,0) circle (2.5pt);
\draw [fill=ffffff] (8,0) circle (2.5pt);
\draw [fill=ffffff] (11,0) circle (2.5pt);
\draw [fill=ffffff] (12,0) circle (2.5pt);
\draw [fill=ffffff] (7,0) circle (2.5pt);
\draw [fill=ffffff] (10,0) circle (2.5pt);
\end{scriptsize}
\end{tikzpicture} 
\end{center}

Our main interest falls on the following polynomial
\begin{align}
    |\tau\rangle + \sum_{l=1}^{[k/2]}\sum_{\substack{\text{all possible choices }\\
    \text{of }i_1, i_2, ..., i_{2l-1}, i_{2l}}}x^{i_1i_3\cdots i_{2l-1}}_{i_2i_4\cdots i_{2l}} \state{\tau\left(\begin{array}{cccc}
    i_1 & i_3 & \cdots & i_{2l-1}\\
    i_2 & i_4 & \cdots & i_{2l}
\end{array}\right)} \label{target}
\end{align}
where the coefficients $a^{i_1i_3\cdots i_{2l-1}}_{i_2i_4\cdots i_{2l}}$ are determined by the system of linear equations
\begin{align}
    0 = & \state{L\left(\begin{array}{cccc}
    j_1 & j_3 & \cdots & j_{2p-1}\\
    j_2 & j_4 & \cdots & j_{2p}
\end{array}\right)\tau} \nonumber \\
 & + \sum_{k=1}^{[r/2]}\sum_{\substack{\text{all possible choices }\\
    \text{of }i_1, i_2, ..., i_{2l-1}, i_{2l}}}x^{i_1i_3\cdots i_{2l-1}}_{i_2i_4\cdots i_{2l}}\state{L\left(\begin{array}{cccc}
    j_1 & j_3 & \cdots & j_{2p-1} \\
    j_2 & j_4 & \cdots & j_{2p}
\end{array}\right)\tau \begin{pmatrix}
i_1 & i_3 & \cdots & i_{2l-1}\\ 
i_2 & i_4 & \cdots & i_{2l}
\end{pmatrix}} \label{linsystem}
\end{align}
for every possible choice of $p = 1,..., [r/2]$ and every sequence $j_1, ..., j_{2p}$ of distinct numbers in $\{1, ..., r\}$ satisfying $j_{2q-1} < j_{2q}, q = 1, ..., p$. 

\begin{exam}
In the case $k = 2$, the polynomial (\ref{target}) we are considering is 
$$|a_1a_2a_2a_1\rangle + x^1_2 \state{L\binom 3 4 a_1 a_2 a_2 a_1} = |a_1a_2a_2a_1\rangle + x^1_2 |a_1 a_1\rangle$$
In graphs: 
$$\state{\raisebox{-.3\height}{\begin{tikzpicture}[line cap=round,line join=round,>=triangle 45,x=1cm,y=1cm, scale = 0.5]
\draw [shift={(2.5,0)},line width=2pt]  plot[domain=0:3.141592653589793,variable=\t]({1*1.5*cos(\t r)+0*1.5*sin(\t r)},{0*1.5*cos(\t r)+1*1.5*sin(\t r)});
\draw [shift={(2.5,0)},line width=2pt]  plot[domain=0:3.141592653589793,variable=\t]({1*0.5*cos(\t r)+0*0.5*sin(\t r)},{0*0.5*cos(\t r)+1*0.5*sin(\t r)});
\begin{scriptsize}
\draw [fill=ffffff] (1,0) circle (2.5pt);
\draw [fill=ffffff] (2,0) circle (2.5pt);
\draw [fill=ffffff] (3,0) circle (2.5pt);
\draw [fill=ffffff] (4,0) circle (2.5pt);
\end{scriptsize}
\end{tikzpicture}}} + x^1_2 \state{\raisebox{-.45\height}{\begin{tikzpicture}[line cap=round,line join=round,>=triangle 45,x=1cm,y=1cm, scale = 0.5]
\draw [shift={(2.5,0)},line width=2pt]  plot[domain=0:3.141592653589793,variable=\t]({1*1.5*cos(\t r)+0*1.5*sin(\t r)},{0*1.5*cos(\t r)+1*1.5*sin(\t r)});
\draw [shift={(2.5,0)},line width=2pt]  plot[domain=0:3.141592653589793,variable=\t]({1*0.5*cos(\t r)+0*0.5*sin(\t r)},{0*0.5*cos(\t r)+1*0.5*sin(\t r)});
\draw [shift={(3.5,0)},line width=2pt, dash pattern = on 1 pt off 3 pt]  plot[domain=3.141592653589793:6.283185307179586,variable=\t]({1*0.5*cos(\t r)+0*0.5*sin(\t r)},{0*0.5*cos(\t r)+1*0.5*sin(\t r)});
\begin{scriptsize}
\draw [fill=ffffff] (2,0) circle (2.5pt);
\draw [fill=ffffff] (1,0) circle (2.5pt);
\draw [fill=ffffff] (4,0) circle (2.5pt);
\draw [fill=ffffff] (3,0) circle (2.5pt);
\end{scriptsize}
\end{tikzpicture}}} = \state{\raisebox{-.3\height}{\begin{tikzpicture}[line cap=round,line join=round,>=triangle 45,x=1cm,y=1cm, scale = 0.5]
\draw [shift={(2.5,0)},line width=2pt]  plot[domain=0:3.141592653589793,variable=\t]({1*1.5*cos(\t r)+0*1.5*sin(\t r)},{0*1.5*cos(\t r)+1*1.5*sin(\t r)});
\draw [shift={(2.5,0)},line width=2pt]  plot[domain=0:3.141592653589793,variable=\t]({1*0.5*cos(\t r)+0*0.5*sin(\t r)},{0*0.5*cos(\t r)+1*0.5*sin(\t r)});
\begin{scriptsize}
\draw [fill=ffffff] (1,0) circle (2.5pt);
\draw [fill=ffffff] (2,0) circle (2.5pt);
\draw [fill=ffffff] (3,0) circle (2.5pt);
\draw [fill=ffffff] (4,0) circle (2.5pt);
\end{scriptsize}
\end{tikzpicture}}} + x^2_1 \state{\raisebox{-.3\height}{\begin{tikzpicture}[line cap=round,line join=round,>=triangle 45,x=1cm,y=1cm, scale = 0.5]
\draw [shift={(2.5,0)},line width=2pt]  plot[domain=0:3.141592653589793,variable=\t]({1*0.5*cos(\t r)+0*0.5*sin(\t r)},{0*0.5*cos(\t r)+1*0.5*sin(\t r)});
\begin{scriptsize}
\draw [fill=ffffff] (2,0) circle (2.5pt);
\draw [fill=ffffff] (3,0) circle (2.5pt);
\end{scriptsize}
\end{tikzpicture}}} $$
We determine $x_1^2$ by
$$\state{L\binom 1 2 ``a_1a_2a_2a_1"} + x^1_2 \state{L\binom 1 2 ``a_1a_1"} = |a_1a_1\rangle + x^1_2 \cdot n = 0.$$
In graphs: 
$$\state{\raisebox{-.45\height}{\begin{tikzpicture}[line cap=round,line join=round,>=triangle 45,x=1cm,y=1cm, scale = 0.5]
\draw [shift={(2.5,0)},line width=2pt]  plot[domain=0:3.141592653589793,variable=\t]({1*1.5*cos(\t r)+0*1.5*sin(\t r)},{0*1.5*cos(\t r)+1*1.5*sin(\t r)});
\draw [shift={(2.5,0)},line width=2pt]  plot[domain=0:3.141592653589793,variable=\t]({1*0.5*cos(\t r)+0*0.5*sin(\t r)},{0*0.5*cos(\t r)+1*0.5*sin(\t r)});
\draw [shift={(1.5,0)},line width=2pt, dash pattern = on 1 pt off 3 pt]  plot[domain=3.141592653589793:6.283185307179586,variable=\t]({1*0.5*cos(\t r)+0*0.5*sin(\t r)},{0*0.5*cos(\t r)+1*0.5*sin(\t r)});
\begin{scriptsize}
\draw [fill=ffffff] (1,0) circle (2.5pt);
\draw [fill=ffffff] (2,0) circle (2.5pt);
\draw [fill=ffffff] (3,0) circle (2.5pt);
\draw [fill=ffffff] (4,0) circle (2.5pt);
\end{scriptsize}
\end{tikzpicture}}} + x^1_2 \state{\raisebox{-.45\height}{\begin{tikzpicture}[line cap=round,line join=round,>=triangle 45,x=1cm,y=1cm, scale = 0.5]
\draw [shift={(2.5,0)},line width=2pt]  plot[domain=0:3.141592653589793,variable=\t]({1*1.5*cos(\t r)+0*1.5*sin(\t r)},{0*1.5*cos(\t r)+1*1.5*sin(\t r)});
\draw [shift={(2.5,0)},line width=2pt]  plot[domain=0:3.141592653589793,variable=\t]({1*0.5*cos(\t r)+0*0.5*sin(\t r)},{0*0.5*cos(\t r)+1*0.5*sin(\t r)});
\draw [shift={(1.5,0)},line width=2pt, dash pattern = on 1 pt off 3 pt]  plot[domain=3.141592653589793:6.283185307179586,variable=\t]({1*0.5*cos(\t r)+0*0.5*sin(\t r)},{0*0.5*cos(\t r)+1*0.5*sin(\t r)});
\draw [shift={(3.5,0)},line width=2pt, dash pattern = on 1 pt off 3 pt]  plot[domain=3.141592653589793:6.283185307179586,variable=\t]({1*0.5*cos(\t r)+0*0.5*sin(\t r)},{0*0.5*cos(\t r)+1*0.5*sin(\t r)});
\begin{scriptsize}
\draw [fill=ffffff] (2,0) circle (2.5pt);
\draw [fill=ffffff] (1,0) circle (2.5pt);
\draw [fill=ffffff] (4,0) circle (2.5pt);
\draw [fill=ffffff] (3,0) circle (2.5pt);
\end{scriptsize}
\end{tikzpicture}}} = \state{\raisebox{-.3\height}{\begin{tikzpicture}[line cap=round,line join=round,>=triangle 45,x=1cm,y=1cm, scale = 0.5]
\draw [shift={(2.5,0)},line width=2pt]  plot[domain=0:3.141592653589793,variable=\t]({1*0.5*cos(\t r)+0*0.5*sin(\t r)},{0*0.5*cos(\t r)+1*0.5*sin(\t r)});
\begin{scriptsize}
\draw [fill=ffffff] (2,0) circle (2.5pt);
\draw [fill=ffffff] (3,0) circle (2.5pt);
\end{scriptsize}
\end{tikzpicture}}} + x^2_1 \cdot n  = 0 $$
Thus $x^1_2 = -\theta / n$. Then the polynomial simplifies to 
$$\theta(\theta+K(n-1))-\frac \theta n \cdot \theta = \frac{n-1}{n}\theta(\theta+Kn). $$

\end{exam}

\begin{exam}
In the case $k = 3$, we are considering the polynomial 
\begin{align*}
    & |a_1a_2a_3a_3a_2a_1\rangle + x^1_2 \state{L\binom 5 6 ``a_1 a_2a_3 a_3a_2a_1"} + x^1_3 \state{L \binom 4 6 ``a_1a_2a_3a_3a_2a_1"} + x^2_3 \state{L \binom 4 5 ``a_1a_2a_3a_3a_2a_1"} \\
    = & |a_1a_2a_3a_3a_2a_1\rangle + x^1_2 \state{a_1 a_1 a_3 a_3} + x^1_3 \state{a_1a_2a_1a_2} + x^2_3 \state{ a_1a_2a_2a_1}. 
\end{align*}
In graphs: 
\begin{align*}
    & \state{\raisebox{-.3\height}{\begin{tikzpicture}[line cap=round,line join=round,>=triangle 45,x=1cm,y=1cm, scale = 0.3]
    \draw [shift={(3.5,0)},line width=2pt]  plot[domain=0:3.141592653589793,variable=\t]({1*0.5*cos(\t r)+0*0.5*sin(\t r)},{0*0.5*cos(\t r)+1*0.5*sin(\t r)});
    \draw [shift={(3.5,0)},line width=2pt]  plot[domain=0:3.141592653589793,variable=\t]({1*2.5*cos(\t r)+0*2.5*sin(\t r)},{0*2.5*cos(\t r)+1*2.5*sin(\t r)});
    \draw [shift={(3.5,0)},line width=2pt]  plot[domain=0:3.141592653589793,variable=\t]({1*1.5*cos(\t r)+0*1.5*sin(\t r)},{0*1.5*cos(\t r)+1*1.5*sin(\t r)});
    \begin{scriptsize}
    \draw [fill=ffffff] (1,0) circle (2.5pt);
    \draw [fill=ffffff] (2,0) circle (2.5pt);
    \draw [fill=ffffff] (5,0) circle (2.5pt);
    \draw [fill=ffffff] (3,0) circle (2.5pt);
    \draw [fill=ffffff] (4,0) circle (2.5pt);
    \draw [fill=ffffff] (6,0) circle (2.5pt);
    \end{scriptsize}
    \end{tikzpicture}}} 
    + x^1_2 \state{\raisebox{-.45\height}{\begin{tikzpicture}[line cap=round,line join=round,>=triangle 45,x=1cm,y=1cm, scale = 0.3]
    \draw [shift={(3.5,0)},line width=2pt]  plot[domain=0:3.141592653589793,variable=\t]({1*0.5*cos(\t r)+0*0.5*sin(\t r)},{0*0.5*cos(\t r)+1*0.5*sin(\t r)});
    \draw [shift={(3.5,0)},line width=2pt]  plot[domain=0:3.141592653589793,variable=\t]({1*2.5*cos(\t r)+0*2.5*sin(\t r)},{0*2.5*cos(\t r)+1*2.5*sin(\t r)});
    \draw [shift={(3.5,0)},line width=2pt]  plot[domain=0:3.141592653589793,variable=\t]({1*1.5*cos(\t r)+0*1.5*sin(\t r)},{0*1.5*cos(\t r)+1*1.5*sin(\t r)});
    \draw [shift={(5.5,0)},line width=2pt, dash pattern = on 1 pt off 3 pt]  plot[domain=3.141592653589793:6.283185307179586,variable=\t]({1*0.5*cos(\t r)+0*0.5*sin(\t r)},{0*0.5*cos(\t r)+1*0.5*sin(\t r)});
    \begin{scriptsize}
    \draw [fill=ffffff] (1,0) circle (2.5pt);
    \draw [fill=ffffff] (2,0) circle (2.5pt);
    \draw [fill=ffffff] (5,0) circle (2.5pt);
    \draw [fill=ffffff] (3,0) circle (2.5pt);
    \draw [fill=ffffff] (4,0) circle (2.5pt);
    \draw [fill=ffffff] (6,0) circle (2.5pt);
    \end{scriptsize}
    \end{tikzpicture}}} 
    + x^1_3 \state{\raisebox{-.5\height}{\begin{tikzpicture}[line cap=round,line join=round,>=triangle 45,x=1cm,y=1cm, scale = 0.3]
    \draw [shift={(3.5,0)},line width=2pt]  plot[domain=0:3.141592653589793,variable=\t]({1*0.5*cos(\t r)+0*0.5*sin(\t r)},{0*0.5*cos(\t r)+1*0.5*sin(\t r)});
    \draw [shift={(3.5,0)},line width=2pt]  plot[domain=0:3.141592653589793,variable=\t]({1*2.5*cos(\t r)+0*2.5*sin(\t r)},{0*2.5*cos(\t r)+1*2.5*sin(\t r)});
    \draw [shift={(3.5,0)},line width=2pt]  plot[domain=0:3.141592653589793,variable=\t]({1*1.5*cos(\t r)+0*1.5*sin(\t r)},{0*1.5*cos(\t r)+1*1.5*sin(\t r)});
    \draw [shift={(5,0)},line width=2pt, dash pattern = on 1 pt off 3 pt]  plot[domain=3.141592653589793:6.283185307179586,variable=\t]({1*1*cos(\t r)+0*1*sin(\t r)},{0*1*cos(\t r)+1*1*sin(\t r)});
    \begin{scriptsize}
    \draw [fill=ffffff] (1,0) circle (2.5pt);
    \draw [fill=ffffff] (2,0) circle (2.5pt);
    \draw [fill=ffffff] (5,0) circle (2.5pt);
    \draw [fill=ffffff] (3,0) circle (2.5pt);
    \draw [fill=ffffff] (4,0) circle (2.5pt);
    \draw [fill=ffffff] (6,0) circle (2.5pt);
    \end{scriptsize}
    \end{tikzpicture}}} 
    + x^2_3 \state{\raisebox{-.45\height}{\begin{tikzpicture}[line cap=round,line join=round,>=triangle 45,x=1cm,y=1cm, scale = 0.3]
    \draw [shift={(3.5,0)},line width=2pt]  plot[domain=0:3.141592653589793,variable=\t]({1*0.5*cos(\t r)+0*0.5*sin(\t r)},{0*0.5*cos(\t r)+1*0.5*sin(\t r)});
    \draw [shift={(3.5,0)},line width=2pt]  plot[domain=0:3.141592653589793,variable=\t]({1*2.5*cos(\t r)+0*2.5*sin(\t r)},{0*2.5*cos(\t r)+1*2.5*sin(\t r)});
    \draw [shift={(3.5,0)},line width=2pt]  plot[domain=0:3.141592653589793,variable=\t]({1*1.5*cos(\t r)+0*1.5*sin(\t r)},{0*1.5*cos(\t r)+1*1.5*sin(\t r)});
    \draw [shift={(4.5,0)},line width=2pt, dash pattern = on 1 pt off 3 pt]  plot[domain=3.141592653589793:6.283185307179586,variable=\t]({1*0.5*cos(\t r)+0*0.5*sin(\t r)},{0*0.5*cos(\t r)+1*0.5*sin(\t r)});
    \begin{scriptsize}
    \draw [fill=ffffff] (1,0) circle (2.5pt);
    \draw [fill=ffffff] (2,0) circle (2.5pt);
    \draw [fill=ffffff] (5,0) circle (2.5pt);
    \draw [fill=ffffff] (3,0) circle (2.5pt);
    \draw [fill=ffffff] (4,0) circle (2.5pt);
    \draw [fill=ffffff] (6,0) circle (2.5pt);
    \end{scriptsize}
    \end{tikzpicture}}} \\
    = & \state{\raisebox{-.3\height}{\begin{tikzpicture}[line cap=round,line join=round,>=triangle 45,x=1cm,y=1cm, scale = 0.3]
    \draw [shift={(3.5,0)},line width=2pt]  plot[domain=0:3.141592653589793,variable=\t]({1*0.5*cos(\t r)+0*0.5*sin(\t r)},{0*0.5*cos(\t r)+1*0.5*sin(\t r)});
    \draw [shift={(3.5,0)},line width=2pt]  plot[domain=0:3.141592653589793,variable=\t]({1*2.5*cos(\t r)+0*2.5*sin(\t r)},{0*2.5*cos(\t r)+1*2.5*sin(\t r)});
    \draw [shift={(3.5,0)},line width=2pt]  plot[domain=0:3.141592653589793,variable=\t]({1*1.5*cos(\t r)+0*1.5*sin(\t r)},{0*1.5*cos(\t r)+1*1.5*sin(\t r)});
    \begin{scriptsize}
    \draw [fill=ffffff] (1,0) circle (2.5pt);
    \draw [fill=ffffff] (2,0) circle (2.5pt);
    \draw [fill=ffffff] (5,0) circle (2.5pt);
    \draw [fill=ffffff] (3,0) circle (2.5pt);
    \draw [fill=ffffff] (4,0) circle (2.5pt);
    \draw [fill=ffffff] (6,0) circle (2.5pt);
    \end{scriptsize}
    \end{tikzpicture}}} 
    + x^1_2 \state{\raisebox{-.3\height}{\begin{tikzpicture}[line cap=round,line join=round,>=triangle 45,x=1cm,y=1cm, scale = 0.3]
    \draw [shift={(1.5,0)},line width=2pt]  plot[domain=0:3.141592653589793,variable=\t]({1*0.5*cos(\t r)+0*0.5*sin(\t r)},{0*0.5*cos(\t r)+1*0.5*sin(\t r)});
    \draw [shift={(3.5,0)},line width=2pt]  plot[domain=0:3.141592653589793,variable=\t]({1*0.5*cos(\t r)+0*0.5*sin(\t r)},{0*0.5*cos(\t r)+1*0.5*sin(\t r)});
    \begin{scriptsize}
    \draw [fill=ffffff] (1,0) circle (2.5pt);
    \draw [fill=ffffff] (2,0) circle (2.5pt);
    \draw [fill=ffffff] (3,0) circle (2.5pt);
    \draw [fill=ffffff] (4,0) circle (2.5pt);
    \end{scriptsize}
    \end{tikzpicture}}} 
    + x^1_3 \state{\raisebox{-.3\height}{\begin{tikzpicture}[line cap=round,line join=round,>=triangle 45,x=1cm,y=1cm, scale = 0.3]
    \draw [shift={(2,0)},line width=2pt]  plot[domain=0:3.141592653589793,variable=\t]({1*1*cos(\t r)+0*1*sin(\t r)},{0*1*cos(\t r)+1*1*sin(\t r)});
    \draw [shift={(3,0)},line width=2pt]  plot[domain=0:3.141592653589793,variable=\t]({1*1*cos(\t r)+0*1*sin(\t r)},{0*1*cos(\t r)+1*1*sin(\t r)});
    \begin{scriptsize}
    \draw [fill=ffffff] (1,0) circle (2.5pt);
    \draw [fill=ffffff] (2,0) circle (2.5pt);
    \draw [fill=ffffff] (3,0) circle (2.5pt);
    \draw [fill=ffffff] (4,0) circle (2.5pt);
    \end{scriptsize}
    \end{tikzpicture}}} 
    + x^2_3 \state{\raisebox{-.3\height}{\begin{tikzpicture}[line cap=round,line join=round,>=triangle 45,x=1cm,y=1cm, scale = 0.3]
\draw [shift={(2.5,0)},line width=2pt]  plot[domain=0:3.141592653589793,variable=\t]({1*1.5*cos(\t r)+0*1.5*sin(\t r)},{0*1.5*cos(\t r)+1*1.5*sin(\t r)});
\draw [shift={(2.5,0)},line width=2pt]  plot[domain=0:3.141592653589793,variable=\t]({1*0.5*cos(\t r)+0*0.5*sin(\t r)},{0*0.5*cos(\t r)+1*0.5*sin(\t r)});
\begin{scriptsize}
\draw [fill=ffffff] (1,0) circle (2.5pt);
\draw [fill=ffffff] (2,0) circle (2.5pt);
\draw [fill=ffffff] (3,0) circle (2.5pt);
\draw [fill=ffffff] (4,0) circle (2.5pt);
\end{scriptsize}
\end{tikzpicture}}} 
\end{align*}
We determine $x^1_2, x^1_3, x^2_3$ by
\begin{align*}
    & \state{L\binom 1 2 ``a_1a_2a_3a_3a_2a_1"} + x^1_2 \state{L\binom 1 2 ``a_1a_1a_3a_3"} + x^1_3 \state{L\binom 1 2 ``a_1a_2a_1a_2"} + x^2_3 \state{L\binom 1 2 ``a_1a_2a_2a_1"} = 0\\
    & \state{L\binom 1 3 ``a_1a_2a_3a_3a_2a_1"} + x^1_2 \state{L\binom 1 3 ``a_1a_1a_3a_3"} + x^1_3 \state{L\binom 1 3 ``a_1a_2a_1a_2"} + x^2_3 \state{L\binom 1 3 ``a_1a_2a_2a_1"} = 0\\
    & \state{L\binom 2 3 ``a_1a_2a_3a_3a_2a_1"} + x^1_2 \state{L\binom 2 3 ``a_1a_1a_3a_3"} + x^1_3 \state{L\binom 2 3 ``a_1a_2a_1a_2"} + x^2_3 \state{L\binom 2 3 ``a_1a_2a_2a_1"} = 0
\end{align*} 
In graphs:
\begin{align*}
    & \state{\raisebox{-.45\height}{\begin{tikzpicture}[line cap=round,line join=round,>=triangle 45,x=1cm,y=1cm, scale = 0.3]
    \draw [shift={(3.5,0)},line width=2pt]  plot[domain=0:3.141592653589793,variable=\t]({1*0.5*cos(\t r)+0*0.5*sin(\t r)},{0*0.5*cos(\t r)+1*0.5*sin(\t r)});
    \draw [shift={(3.5,0)},line width=2pt]  plot[domain=0:3.141592653589793,variable=\t]({1*2.5*cos(\t r)+0*2.5*sin(\t r)},{0*2.5*cos(\t r)+1*2.5*sin(\t r)});
    \draw [shift={(3.5,0)},line width=2pt]  plot[domain=0:3.141592653589793,variable=\t]({1*1.5*cos(\t r)+0*1.5*sin(\t r)},{0*1.5*cos(\t r)+1*1.5*sin(\t r)});
    \draw [shift={(1.5,0)},line width=2pt, dash pattern = on 1 pt off 3 pt]  plot[domain=3.141592653589793:6.283185307179586,variable=\t]({1*0.5*cos(\t r)+0*0.5*sin(\t r)},{0*0.5*cos(\t r)+1*0.5*sin(\t r)});
    \begin{scriptsize}
    \draw [fill=ffffff] (1,0) circle (2.5pt);
    \draw [fill=ffffff] (2,0) circle (2.5pt);
    \draw [fill=ffffff] (5,0) circle (2.5pt);
    \draw [fill=ffffff] (3,0) circle (2.5pt);
    \draw [fill=ffffff] (4,0) circle (2.5pt);
    \draw [fill=ffffff] (6,0) circle (2.5pt);
    \end{scriptsize}
    \end{tikzpicture}}} 
    + x^1_2 \state{\raisebox{-.45\height}{\begin{tikzpicture}[line cap=round,line join=round,>=triangle 45,x=1cm,y=1cm, scale = 0.3]
    \draw [shift={(3.5,0)},line width=2pt]  plot[domain=0:3.141592653589793,variable=\t]({1*0.5*cos(\t r)+0*0.5*sin(\t r)},{0*0.5*cos(\t r)+1*0.5*sin(\t r)});
    \draw [shift={(3.5,0)},line width=2pt]  plot[domain=0:3.141592653589793,variable=\t]({1*2.5*cos(\t r)+0*2.5*sin(\t r)},{0*2.5*cos(\t r)+1*2.5*sin(\t r)});
    \draw [shift={(3.5,0)},line width=2pt]  plot[domain=0:3.141592653589793,variable=\t]({1*1.5*cos(\t r)+0*1.5*sin(\t r)},{0*1.5*cos(\t r)+1*1.5*sin(\t r)});
    \draw [shift={(5.5,0)},line width=2pt, dash pattern = on 1 pt off 3 pt]  plot[domain=3.141592653589793:6.283185307179586,variable=\t]({1*0.5*cos(\t r)+0*0.5*sin(\t r)},{0*0.5*cos(\t r)+1*0.5*sin(\t r)});
    \draw [shift={(1.5,0)},line width=2pt, dash pattern = on 1 pt off 3 pt]  plot[domain=3.141592653589793:6.283185307179586,variable=\t]({1*0.5*cos(\t r)+0*0.5*sin(\t r)},{0*0.5*cos(\t r)+1*0.5*sin(\t r)});
    \begin{scriptsize}
    \draw [fill=ffffff] (1,0) circle (2.5pt);
    \draw [fill=ffffff] (2,0) circle (2.5pt);
    \draw [fill=ffffff] (5,0) circle (2.5pt);
    \draw [fill=ffffff] (3,0) circle (2.5pt);
    \draw [fill=ffffff] (4,0) circle (2.5pt);
    \draw [fill=ffffff] (6,0) circle (2.5pt);
    \end{scriptsize}
    \end{tikzpicture}}} 
    + x^1_3 \state{\raisebox{-.5\height}{\begin{tikzpicture}[line cap=round,line join=round,>=triangle 45,x=1cm,y=1cm, scale = 0.3]
    \draw [shift={(3.5,0)},line width=2pt]  plot[domain=0:3.141592653589793,variable=\t]({1*0.5*cos(\t r)+0*0.5*sin(\t r)},{0*0.5*cos(\t r)+1*0.5*sin(\t r)});
    \draw [shift={(3.5,0)},line width=2pt]  plot[domain=0:3.141592653589793,variable=\t]({1*2.5*cos(\t r)+0*2.5*sin(\t r)},{0*2.5*cos(\t r)+1*2.5*sin(\t r)});
    \draw [shift={(3.5,0)},line width=2pt]  plot[domain=0:3.141592653589793,variable=\t]({1*1.5*cos(\t r)+0*1.5*sin(\t r)},{0*1.5*cos(\t r)+1*1.5*sin(\t r)});
    \draw [shift={(5,0)},line width=2pt, dash pattern = on 1 pt off 3 pt]  plot[domain=3.141592653589793:6.283185307179586,variable=\t]({1*1*cos(\t r)+0*1*sin(\t r)},{0*1*cos(\t r)+1*1*sin(\t r)});
    \draw [shift={(1.5,0)},line width=2pt, dash pattern = on 1 pt off 3 pt]  plot[domain=3.141592653589793:6.283185307179586,variable=\t]({1*0.5*cos(\t r)+0*0.5*sin(\t r)},{0*0.5*cos(\t r)+1*0.5*sin(\t r)});
    \begin{scriptsize}
    \draw [fill=ffffff] (1,0) circle (2.5pt);
    \draw [fill=ffffff] (2,0) circle (2.5pt);
    \draw [fill=ffffff] (5,0) circle (2.5pt);
    \draw [fill=ffffff] (3,0) circle (2.5pt);
    \draw [fill=ffffff] (4,0) circle (2.5pt);
    \draw [fill=ffffff] (6,0) circle (2.5pt);
    \end{scriptsize}
    \end{tikzpicture}}} 
    + x^2_3 \state{\raisebox{-.45\height}{\begin{tikzpicture}[line cap=round,line join=round,>=triangle 45,x=1cm,y=1cm, scale = 0.3]
    \draw [shift={(3.5,0)},line width=2pt]  plot[domain=0:3.141592653589793,variable=\t]({1*0.5*cos(\t r)+0*0.5*sin(\t r)},{0*0.5*cos(\t r)+1*0.5*sin(\t r)});
    \draw [shift={(3.5,0)},line width=2pt]  plot[domain=0:3.141592653589793,variable=\t]({1*2.5*cos(\t r)+0*2.5*sin(\t r)},{0*2.5*cos(\t r)+1*2.5*sin(\t r)});
    \draw [shift={(3.5,0)},line width=2pt]  plot[domain=0:3.141592653589793,variable=\t]({1*1.5*cos(\t r)+0*1.5*sin(\t r)},{0*1.5*cos(\t r)+1*1.5*sin(\t r)});
    \draw [shift={(4.5,0)},line width=2pt, dash pattern = on 1 pt off 3 pt]  plot[domain=3.141592653589793:6.283185307179586,variable=\t]({1*0.5*cos(\t r)+0*0.5*sin(\t r)},{0*0.5*cos(\t r)+1*0.5*sin(\t r)});
    \draw [shift={(1.5,0)},line width=2pt, dash pattern = on 1 pt off 3 pt]  plot[domain=3.141592653589793:6.283185307179586,variable=\t]({1*0.5*cos(\t r)+0*0.5*sin(\t r)},{0*0.5*cos(\t r)+1*0.5*sin(\t r)});
    \begin{scriptsize}
    \draw [fill=ffffff] (1,0) circle (2.5pt);
    \draw [fill=ffffff] (2,0) circle (2.5pt);
    \draw [fill=ffffff] (5,0) circle (2.5pt);
    \draw [fill=ffffff] (3,0) circle (2.5pt);
    \draw [fill=ffffff] (4,0) circle (2.5pt);
    \draw [fill=ffffff] (6,0) circle (2.5pt);
    \end{scriptsize}
    \end{tikzpicture}}}  = 0 \\
    & \state{\raisebox{-.45\height}{\begin{tikzpicture}[line cap=round,line join=round,>=triangle 45,x=1cm,y=1cm, scale = 0.3]
    \draw [shift={(3.5,0)},line width=2pt]  plot[domain=0:3.141592653589793,variable=\t]({1*0.5*cos(\t r)+0*0.5*sin(\t r)},{0*0.5*cos(\t r)+1*0.5*sin(\t r)});
    \draw [shift={(3.5,0)},line width=2pt]  plot[domain=0:3.141592653589793,variable=\t]({1*2.5*cos(\t r)+0*2.5*sin(\t r)},{0*2.5*cos(\t r)+1*2.5*sin(\t r)});
    \draw [shift={(3.5,0)},line width=2pt]  plot[domain=0:3.141592653589793,variable=\t]({1*1.5*cos(\t r)+0*1.5*sin(\t r)},{0*1.5*cos(\t r)+1*1.5*sin(\t r)});
    \draw [shift={(2,0)},line width=2pt, dash pattern = on 1 pt off 3 pt]  plot[domain=3.141592653589793:6.283185307179586,variable=\t]({1*1*cos(\t r)+0*1*sin(\t r)},{0*1*cos(\t r)+1*1*sin(\t r)});
    \begin{scriptsize}
    \draw [fill=ffffff] (1,0) circle (2.5pt);
    \draw [fill=ffffff] (2,0) circle (2.5pt);
    \draw [fill=ffffff] (5,0) circle (2.5pt);
    \draw [fill=ffffff] (3,0) circle (2.5pt);
    \draw [fill=ffffff] (4,0) circle (2.5pt);
    \draw [fill=ffffff] (6,0) circle (2.5pt);
    \end{scriptsize}
    \end{tikzpicture}}} 
    + x^1_2 \state{\raisebox{-.45\height}{\begin{tikzpicture}[line cap=round,line join=round,>=triangle 45,x=1cm,y=1cm, scale = 0.3]
    \draw [shift={(3.5,0)},line width=2pt]  plot[domain=0:3.141592653589793,variable=\t]({1*0.5*cos(\t r)+0*0.5*sin(\t r)},{0*0.5*cos(\t r)+1*0.5*sin(\t r)});
    \draw [shift={(3.5,0)},line width=2pt]  plot[domain=0:3.141592653589793,variable=\t]({1*2.5*cos(\t r)+0*2.5*sin(\t r)},{0*2.5*cos(\t r)+1*2.5*sin(\t r)});
    \draw [shift={(3.5,0)},line width=2pt]  plot[domain=0:3.141592653589793,variable=\t]({1*1.5*cos(\t r)+0*1.5*sin(\t r)},{0*1.5*cos(\t r)+1*1.5*sin(\t r)});
    \draw [shift={(5.5,0)},line width=2pt, dash pattern = on 1 pt off 3 pt]  plot[domain=3.141592653589793:6.283185307179586,variable=\t]({1*0.5*cos(\t r)+0*0.5*sin(\t r)},{0*0.5*cos(\t r)+1*0.5*sin(\t r)});
    \draw [shift={(2,0)},line width=2pt, dash pattern = on 1 pt off 3 pt]  plot[domain=3.141592653589793:6.283185307179586,variable=\t]({1*1*cos(\t r)+0*1*sin(\t r)},{0*1*cos(\t r)+1*1*sin(\t r)});
    \begin{scriptsize}
    \draw [fill=ffffff] (1,0) circle (2.5pt);
    \draw [fill=ffffff] (2,0) circle (2.5pt);
    \draw [fill=ffffff] (5,0) circle (2.5pt);
    \draw [fill=ffffff] (3,0) circle (2.5pt);
    \draw [fill=ffffff] (4,0) circle (2.5pt);
    \draw [fill=ffffff] (6,0) circle (2.5pt);
    \end{scriptsize}
    \end{tikzpicture}}} 
    + x^1_3 \state{\raisebox{-.5\height}{\begin{tikzpicture}[line cap=round,line join=round,>=triangle 45,x=1cm,y=1cm, scale = 0.3]
    \draw [shift={(3.5,0)},line width=2pt]  plot[domain=0:3.141592653589793,variable=\t]({1*0.5*cos(\t r)+0*0.5*sin(\t r)},{0*0.5*cos(\t r)+1*0.5*sin(\t r)});
    \draw [shift={(3.5,0)},line width=2pt]  plot[domain=0:3.141592653589793,variable=\t]({1*2.5*cos(\t r)+0*2.5*sin(\t r)},{0*2.5*cos(\t r)+1*2.5*sin(\t r)});
    \draw [shift={(3.5,0)},line width=2pt]  plot[domain=0:3.141592653589793,variable=\t]({1*1.5*cos(\t r)+0*1.5*sin(\t r)},{0*1.5*cos(\t r)+1*1.5*sin(\t r)});
    \draw [shift={(5,0)},line width=2pt, dash pattern = on 1 pt off 3 pt]  plot[domain=3.141592653589793:6.283185307179586,variable=\t]({1*1*cos(\t r)+0*1*sin(\t r)},{0*1*cos(\t r)+1*1*sin(\t r)});
    \draw [shift={(2,0)},line width=2pt, dash pattern = on 1 pt off 3 pt]  plot[domain=3.141592653589793:6.283185307179586,variable=\t]({1*1*cos(\t r)+0*1*sin(\t r)},{0*1*cos(\t r)+1*1*sin(\t r)});
    \begin{scriptsize}
    \draw [fill=ffffff] (1,0) circle (2.5pt);
    \draw [fill=ffffff] (2,0) circle (2.5pt);
    \draw [fill=ffffff] (5,0) circle (2.5pt);
    \draw [fill=ffffff] (3,0) circle (2.5pt);
    \draw [fill=ffffff] (4,0) circle (2.5pt);
    \draw [fill=ffffff] (6,0) circle (2.5pt);
    \end{scriptsize}
    \end{tikzpicture}}} 
    + x^2_3 \state{\raisebox{-.45\height}{\begin{tikzpicture}[line cap=round,line join=round,>=triangle 45,x=1cm,y=1cm, scale = 0.3]
    \draw [shift={(3.5,0)},line width=2pt]  plot[domain=0:3.141592653589793,variable=\t]({1*0.5*cos(\t r)+0*0.5*sin(\t r)},{0*0.5*cos(\t r)+1*0.5*sin(\t r)});
    \draw [shift={(3.5,0)},line width=2pt]  plot[domain=0:3.141592653589793,variable=\t]({1*2.5*cos(\t r)+0*2.5*sin(\t r)},{0*2.5*cos(\t r)+1*2.5*sin(\t r)});
    \draw [shift={(3.5,0)},line width=2pt]  plot[domain=0:3.141592653589793,variable=\t]({1*1.5*cos(\t r)+0*1.5*sin(\t r)},{0*1.5*cos(\t r)+1*1.5*sin(\t r)});
    \draw [shift={(4.5,0)},line width=2pt, dash pattern = on 1 pt off 3 pt]  plot[domain=3.141592653589793:6.283185307179586,variable=\t]({1*0.5*cos(\t r)+0*0.5*sin(\t r)},{0*0.5*cos(\t r)+1*0.5*sin(\t r)});
    \draw [shift={(2,0)},line width=2pt, dash pattern = on 1 pt off 3 pt]  plot[domain=3.141592653589793:6.283185307179586,variable=\t]({1*1*cos(\t r)+0*1*sin(\t r)},{0*1*cos(\t r)+1*1*sin(\t r)});
    \begin{scriptsize}
    \draw [fill=ffffff] (1,0) circle (2.5pt);
    \draw [fill=ffffff] (2,0) circle (2.5pt);
    \draw [fill=ffffff] (5,0) circle (2.5pt);
    \draw [fill=ffffff] (3,0) circle (2.5pt);
    \draw [fill=ffffff] (4,0) circle (2.5pt);
    \draw [fill=ffffff] (6,0) circle (2.5pt);
    \end{scriptsize}
    \end{tikzpicture}}}  = 0 \\
    & \state{\raisebox{-.45\height}{\begin{tikzpicture}[line cap=round,line join=round,>=triangle 45,x=1cm,y=1cm, scale = 0.3]
    \draw [shift={(3.5,0)},line width=2pt]  plot[domain=0:3.141592653589793,variable=\t]({1*0.5*cos(\t r)+0*0.5*sin(\t r)},{0*0.5*cos(\t r)+1*0.5*sin(\t r)});
    \draw [shift={(3.5,0)},line width=2pt]  plot[domain=0:3.141592653589793,variable=\t]({1*2.5*cos(\t r)+0*2.5*sin(\t r)},{0*2.5*cos(\t r)+1*2.5*sin(\t r)});
    \draw [shift={(3.5,0)},line width=2pt]  plot[domain=0:3.141592653589793,variable=\t]({1*1.5*cos(\t r)+0*1.5*sin(\t r)},{0*1.5*cos(\t r)+1*1.5*sin(\t r)});
    \draw [shift={(2.5,0)},line width=2pt, dash pattern = on 1 pt off 3 pt]  plot[domain=3.141592653589793:6.283185307179586,variable=\t]({1*0.5*cos(\t r)+0*0.5*sin(\t r)},{0*0.5*cos(\t r)+1*0.5*sin(\t r)});
    \begin{scriptsize}
    \draw [fill=ffffff] (1,0) circle (2.5pt);
    \draw [fill=ffffff] (2,0) circle (2.5pt);
    \draw [fill=ffffff] (5,0) circle (2.5pt);
    \draw [fill=ffffff] (3,0) circle (2.5pt);
    \draw [fill=ffffff] (4,0) circle (2.5pt);
    \draw [fill=ffffff] (6,0) circle (2.5pt);
    \end{scriptsize}
    \end{tikzpicture}}} 
    + x^1_2 \state{\raisebox{-.45\height}{\begin{tikzpicture}[line cap=round,line join=round,>=triangle 45,x=1cm,y=1cm, scale = 0.3]
    \draw [shift={(3.5,0)},line width=2pt]  plot[domain=0:3.141592653589793,variable=\t]({1*0.5*cos(\t r)+0*0.5*sin(\t r)},{0*0.5*cos(\t r)+1*0.5*sin(\t r)});
    \draw [shift={(3.5,0)},line width=2pt]  plot[domain=0:3.141592653589793,variable=\t]({1*2.5*cos(\t r)+0*2.5*sin(\t r)},{0*2.5*cos(\t r)+1*2.5*sin(\t r)});
    \draw [shift={(3.5,0)},line width=2pt]  plot[domain=0:3.141592653589793,variable=\t]({1*1.5*cos(\t r)+0*1.5*sin(\t r)},{0*1.5*cos(\t r)+1*1.5*sin(\t r)});
    \draw [shift={(5.5,0)},line width=2pt, dash pattern = on 1 pt off 3 pt]  plot[domain=3.141592653589793:6.283185307179586,variable=\t]({1*0.5*cos(\t r)+0*0.5*sin(\t r)},{0*0.5*cos(\t r)+1*0.5*sin(\t r)});
    \draw [shift={(2.5,0)},line width=2pt, dash pattern = on 1 pt off 3 pt]  plot[domain=3.141592653589793:6.283185307179586,variable=\t]({1*0.5*cos(\t r)+0*0.5*sin(\t r)},{0*0.5*cos(\t r)+1*0.5*sin(\t r)});
    \begin{scriptsize}
    \draw [fill=ffffff] (1,0) circle (2.5pt);
    \draw [fill=ffffff] (2,0) circle (2.5pt);
    \draw [fill=ffffff] (5,0) circle (2.5pt);
    \draw [fill=ffffff] (3,0) circle (2.5pt);
    \draw [fill=ffffff] (4,0) circle (2.5pt);
    \draw [fill=ffffff] (6,0) circle (2.5pt);
    \end{scriptsize}
    \end{tikzpicture}}} 
    + x^1_3 \state{\raisebox{-.5\height}{\begin{tikzpicture}[line cap=round,line join=round,>=triangle 45,x=1cm,y=1cm, scale = 0.3]
    \draw [shift={(3.5,0)},line width=2pt]  plot[domain=0:3.141592653589793,variable=\t]({1*0.5*cos(\t r)+0*0.5*sin(\t r)},{0*0.5*cos(\t r)+1*0.5*sin(\t r)});
    \draw [shift={(3.5,0)},line width=2pt]  plot[domain=0:3.141592653589793,variable=\t]({1*2.5*cos(\t r)+0*2.5*sin(\t r)},{0*2.5*cos(\t r)+1*2.5*sin(\t r)});
    \draw [shift={(3.5,0)},line width=2pt]  plot[domain=0:3.141592653589793,variable=\t]({1*1.5*cos(\t r)+0*1.5*sin(\t r)},{0*1.5*cos(\t r)+1*1.5*sin(\t r)});
    \draw [shift={(5,0)},line width=2pt, dash pattern = on 1 pt off 3 pt]  plot[domain=3.141592653589793:6.283185307179586,variable=\t]({1*1*cos(\t r)+0*1*sin(\t r)},{0*1*cos(\t r)+1*1*sin(\t r)});
    \draw [shift={(2.5,0)},line width=2pt, dash pattern = on 1 pt off 3 pt]  plot[domain=3.141592653589793:6.283185307179586,variable=\t]({1*0.5*cos(\t r)+0*0.5*sin(\t r)},{0*0.5*cos(\t r)+1*0.5*sin(\t r)});
    \begin{scriptsize}
    \draw [fill=ffffff] (1,0) circle (2.5pt);
    \draw [fill=ffffff] (2,0) circle (2.5pt);
    \draw [fill=ffffff] (5,0) circle (2.5pt);
    \draw [fill=ffffff] (3,0) circle (2.5pt);
    \draw [fill=ffffff] (4,0) circle (2.5pt);
    \draw [fill=ffffff] (6,0) circle (2.5pt);
    \end{scriptsize}
    \end{tikzpicture}}} 
    + x^2_3 \state{\raisebox{-.45\height}{\begin{tikzpicture}[line cap=round,line join=round,>=triangle 45,x=1cm,y=1cm, scale = 0.3]
    \draw [shift={(3.5,0)},line width=2pt]  plot[domain=0:3.141592653589793,variable=\t]({1*0.5*cos(\t r)+0*0.5*sin(\t r)},{0*0.5*cos(\t r)+1*0.5*sin(\t r)});
    \draw [shift={(3.5,0)},line width=2pt]  plot[domain=0:3.141592653589793,variable=\t]({1*2.5*cos(\t r)+0*2.5*sin(\t r)},{0*2.5*cos(\t r)+1*2.5*sin(\t r)});
    \draw [shift={(3.5,0)},line width=2pt]  plot[domain=0:3.141592653589793,variable=\t]({1*1.5*cos(\t r)+0*1.5*sin(\t r)},{0*1.5*cos(\t r)+1*1.5*sin(\t r)});
    \draw [shift={(4.5,0)},line width=2pt, dash pattern = on 1 pt off 3 pt]  plot[domain=3.141592653589793:6.283185307179586,variable=\t]({1*0.5*cos(\t r)+0*0.5*sin(\t r)},{0*0.5*cos(\t r)+1*0.5*sin(\t r)});
    \draw [shift={(2.5,0)},line width=2pt, dash pattern = on 1 pt off 3 pt]  plot[domain=3.141592653589793:6.283185307179586,variable=\t]({1*0.5*cos(\t r)+0*0.5*sin(\t r)},{0*0.5*cos(\t r)+1*0.5*sin(\t r)});
    \begin{scriptsize}
    \draw [fill=ffffff] (1,0) circle (2.5pt);
    \draw [fill=ffffff] (2,0) circle (2.5pt);
    \draw [fill=ffffff] (5,0) circle (2.5pt);
    \draw [fill=ffffff] (3,0) circle (2.5pt);
    \draw [fill=ffffff] (4,0) circle (2.5pt);
    \draw [fill=ffffff] (6,0) circle (2.5pt);
    \end{scriptsize}
    \end{tikzpicture}}}  = 0 
\end{align*}
The system simplifies to 
\begin{align*}
    & \state{a_3a_3a_1a_1} + x^1_2 n \state{a_3a_3} + x^1_3 \state{a_1a_1} + x^2_3 \state{a_1a_1} = \theta^2 + n\theta x^1_2 + \theta x^1_3+ \theta x^2_3 = 0\\
    & \state{a_2a_1a_2a_1} + x^1_2 \state{ a_1a_1} + x^1_3 n \state{a_2a_2} + x^2_3 \state{a_1a_1} = \theta(\theta+K(n-1)) + \theta x^1_2 + n\theta x^1_3 + \theta x^2_3 = 0\\
    & \state{a_1a_2a_2a_1} + x^1_2 \state{a_1a_1} + x^1_3 \state{a_1a_1} + x^2_3 n \state{a_1a_1} = \theta(\theta+K(n-1)) + \theta x^1_2 + \theta x^1 3 + n \theta x^2_3 = 0
\end{align*} 
One can solve the system and get 
$$x^1_2 = \frac{2K-\theta}{n+2}, x^1_3 = x^2_3 = \frac{-Kn-\theta}{n+2}$$
\end{exam}
We can then compute the polynomial (\ref{target}), which is
$$\frac{n-1}{n+2} \theta(\theta+Kn)(\theta+K(2n+2))$$

\begin{conjecture}
Each coefficient $a^{i_1i_3\cdots i_{2l-1}}_{i_2i_4\cdots i_{2l}}$ is a polynomial in $\theta$ of degree $l$. 
\end{conjecture}

\begin{conjecture}
There exists a constant $C$ depending only on $n$, such that the polynomial (\ref{target}) is
$$C \prod_{p=0}^{k-1}(\theta+Kp(n+p-1)) = C \theta (\theta+Kn)\cdots (\theta+K(k-1)(n+k-2))$$
Conjecturally, 
$$C = \frac{(n-1)n(n+1)\cdots (n+r-2)}{n(n+2)(n+4)\cdots (n+2r-2)}$$
\end{conjecture}

\begin{rema}
We first notice that the number of variables in the linear system coincides with the number of equations, which is 
the number of pairings that can be chosen within $1, 2, ..., k$. This number can be computed as
$$\binom{k} 2 + \binom k 4 \cdot (4-1)!! + \cdots + \binom k{2\cdot [k/2]} \cdot (2[k/2] - 1)!!$$
The conjecture has been numerically checked up to $k = 4$. 
\end{rema}

\begin{rema}
Regarding the linear system (\ref{linsystem}), using the conclusion of Proposition \ref{dominant} we can show that the coefficient matrix is diagonally dominant when $n$ is sufficiently large and when $\theta$ is evaluated as a fixed number. If we regard the solution as a polynomial in $n$ then it is uniquely determined. 
\end{rema}

\begin{rema}
Numerical experiments also show that if we start with $\tau = ``a_1\cdots a_k a_{\sigma(1)}\cdots a_{\sigma(k)}"$ for any $\sigma\in \text{Sym}_k$ (whose polynomial has the same highest degree component, as shown in Proposition \ref{dominant}), we will end up with the same linear system (\ref{linsystem}). As a result, the solution $x^{i_1 i_3 \cdots i_{2k-1}}_{i_2 i_4 \cdots i_{2k}}$ remains the same, and Conjecture 2 still holds. A formal proof of this phenomenon is probably not very difficult. 
\end{rema}

\begin{rema}
In the study of MOSVAs over $n$-dimensional space forms, Conjecture 2 is key to understand the irreducible modules generated by eigenfunctions of eigenvalues 
$$\lambda = Kp(n-1+p), p = 0, 1, 2, ...$$ These eigenvalues coincides with the spectrum of the (global) Laplace-Beltrami operator over the sphere with sectional curvature $K$. It is shown in \cite{Q-MOSVA-2d} that when $n=2$, such modules are different from those generate by eigenfunctions with generic eigenvalues. In particular, the extensions of such modules with each other has to be completely reducible. We expect all the properties discussed in \cite{Q-MOSVA-2d} to hold in higher dimensions. Thus the conjecture has a conceptual reason to hold.  
\end{rema}

\noindent {\small \sc Pacific Institute of Mathematical Science | University Of Manitoba\\ 451 Machray Hall, 186 Dysart Road, Winnipeg, MB R3T 2N2, Canada}

\noindent (Current address) {\small \sc Department of Mathematics, University of Denver \\ 216 C. M. Knudson Hall, 2390 South York Street, Denver, CO 80208, USA}

\noindent {\em E-mail address}: fei.qi@umanitoba.ca | fei.qi.math.phys@gmail.com

\end{document}